\documentclass[12pt]{amsart}

\usepackage{jsn_style}

\title{A Functorial Classification of Toric Stacks}
\author{John S.\ Nolan}
\address{Department of Mathematics, The Ohio State University}
\email{nolan.422@osu.edu}

\begin{document}

\begin{abstract}
    Toric stacks enrich the study of toric varieties by accounting for quotients of actions of diagonalizable groups.
    In practice, toric stacks are constructed and understood via ``stacky fans'' generalizing the fans appearing in toric geometry.
    Previous attempts to make this construction into an equivalence of categories have either run into crucial errors or have been restricted to special cases (e.g.\ only considering separated Deligne--Mumford stacks).
    In this paper, we solve this problem by introducing a new 2-category $\StkFan$ of stacky fans and constructing an equivalence of 2-categories between $\StkFan$ and a 2-category $\TStk$ of toric stacks.
    Along the way, we develop a theory of ``toric categories'' which categorifies the classical theory of toric monoids.
    We also explain how working with toric categories provides a simple axiomatic framework for Gale duality.
\end{abstract}

\maketitle

\tableofcontents

\section{Introduction}

Fix an algebraically closed field $k$ of characteristic zero; we will work throughout over $\Spec k$.
Our goal in this paper is to classify toric stacks over $\Spec k$ in terms of ``stacky fans.''
There are many definitions of toric stack in the literature, and for the purposes of this introduction we will focus on the following definition:

\begin{dfn}[{Special case of \zcref{dfn:toric_stack}}] \label{dfn:toric_stack_first_def} 
    A \emph{toric stack with trivial generic stabilizer} is an Artin stack $\Xfr$ of finite type together with:
    \begin{itemize}
        \item An action of an algebraic torus $T$ on $\Xfr$ and
        \item A dense $T$-equivariant open immersion $j: T \hookrightarrow \Xfr$
    \end{itemize}
    such that:
    \begin{itemize}
        \item $\Xfr$ is normal,
        \item $\Xfr$ has affine diagonal, and
        \item All stabilizer groups of $\Xfr$ at geometric points are diagonalizable.\footnote{Because of our hypotheses on $k$, diagonalizable algebraic groups are the same as products of tori and finite abelian groups.}
    \end{itemize}
    We refer to $T$ as the \emph{dense torus} of $\Xfr$.
\end{dfn}

\begin{rmk}
    In the main body of the paper, we use a more general definition of toric stack in which the dense torus is allowed to be a ``stacky torus,'' i.e.\ a product $T_0 \times BT_1$ with $T_0$ an algebraic torus and $T_1$ a diagonalizable algebraic group.
    For such a toric stack, the generic stabilizer group is $T_1$.
    Working with nontrivial generic stabilizers allows us to account for standard constructions such as gerbes and root stacks.
\end{rmk}

\subsection{Why should we care?}

In algebraic geometry, toric varieties are classified via combinatorial objects known as fans.
This classification allows one to reduce many problems about the algebraic geometry to problems in combinatorics and vice versa.
For instance, homological mirror symmetry is known for toric varieties (see \cite{fang2011categorification, kuwagaki2020nonequivariant} among others) in large part because it is possible to use the combinatorics of fans to construct a mirror.
Related questions for more general algebraic varieties are much harder to even state precisely.

Cox \cite{cox1995homogeneous} proved that every smooth toric variety can be realized as a stack quotient $(\AA^n \setminus Z) / G$,\footnote{Throughout this paper, the notation $X / G$ will always refer to the stack quotient of $X$ by $G$.
Whenever we need another notion of quotient, we will mention this explicitly.}
where $Z \subset \AA^n$ is a union of hyperplanes and $G \subset \GG_m^n$ is a diagonalizable group.
In other words, every smooth toric variety is naturally a substack of some quotient stack $\AA^n / G$.

This perspective led Borisov, Chen, and Smith to define (smooth, separated) toric Deligne--Mumford stack as certain open substacks of quotient stacks $\AA^n / G$ \cite{borisov2005orbifold}.
Borisov, Chen, and Smith also produced the first formalism of ``stacky fans'' for working with such toric stacks.
This definition of toric Deligne--Mumford stacks is useful even for the study of classical toric varieties, since toric Deligne--Mumford stacks show up in toric GIT and the toric minimal model program (see e.g. \cite{kawamata2006derived} for applications).

More generally, one can relax the axioms of \cite{borisov2005orbifold} to define smooth toric (Artin) stacks as quotients of open toric subvarieties of $\AA^n$ by diagonalizable groups acting coordinatewise.
Again, this generality is useful even for the study of classical toric varieties: the stack $\AA^n / G$ is a natural home for ``windows'' and variation of GIT (see e.g.\ \cite{ballard2012variation, halpern2015derived}).

It is natural to ask if there is an analogous story in the singular case.
In this generality, the Cox construction breaks down, and the above approaches fail.
Indeed, there are singular toric varieties which cannot be written as stack quotients of quasi-affine toric varieties by actions of diagonalizable groups (e.g.\ this is true for \cite[Example on pp.\ 25--26]{fulton1993introduction}, discussed in more detail in \cite[Example 4.2.13]{cls}).\footnote{The nonexistence of stack quotient presentations of toric varieties does not contradict the ``almost geometric'' quotient presentation of toric varieties obtained from the Cox construction (see e.g.\ \cite[Theorem 5.1.11]{cls}).
Indeed, the latter quotient constructions behave quite differently from stack quotients in general.}
In \cite{geraschenko2015toric_1}, Geraschenko and Satriano circumvented this subtlety by defining a toric stack to be a quotient of a toric variety by a diagonalizable group acting via a homomorphism to the dense torus.
They also produced a theory of ``stacky fans'' giving presentations of such toric stacks.

Geraschenko and Satriano attempted to prove in \cite{geraschenko2015toric_2} that every toric stack (according to our definition) with trivial generic stabilizer arises via their construction.
However, their proof had an error in the non-smooth case, and a counterexample to the main theorem of \cite{geraschenko2015toric_2} was later announced in \cite[Corollary 4.3.9]{gillam2015theory}.
(We also give a new counterexample in \zcref{prop:counterexample}.)
Our goal in this paper is to create a formalism of ``stacky fans'' that accounts for these counterexamples and unifies prior treatments of the subject.

We conclude this subsection with a few remarks on the connections between the geometry of toric stacks and other branches of mathematics.

\begin{rmk}[Combinatorial perspectives]
    Toric DM stacks show up in some interesting combinatorial problems:
    \begin{itemize}
        \item The McKay correspondence (see e.g.\ \cite{bridgeland2001mckay}) in the abelian case can be understood as a relationship between stacky resolutions and crepant resolutions of certain toric singularities.
        \item Batyrev and Blume \cite{batyrev2011functor, batyrev2011generalisations} showed that Losev--Manin spaces (which serve as combinatorial models for moduli of stable genus-0 pointed curves) can be understood as toric varieties associated with type $A$ root systems.
        They then studied analogues of Losev--Manin spaces related to other root systems and found that the analogous story in type $C$ naturally leads one to consider toric DM stacks.
    \end{itemize}
    The torus-equivariant geometry of these stacks (as well as that of classical toric varieties) is encoded by the corresponding quotient stacks $X / T$.
    Having a good understanding of toric stacks will give us better tools to work with these combinatorial objects.
\end{rmk}

\begin{rmk}[Polytopes]
    There is a deep connection between the theory of lattice polytopes and the theory of ample line bundles on toric varieties.
    It would be interesting to develop an analogous theory for toric stacks (e.g.\ relating closed substacks of products of weighted projective stacks to some class of ``decorated polytopes'').
    Such a theory could then leverage known results about Deligne--Mumford stacks (e.g.\ Serre duality) to generalize classical results about lattice polytopes.
    We will leave this as a problem for future work.
\end{rmk}

\begin{rmk}[Commutative algebra perspectives]
    Commutative algebraists often understand smooth toric varieties via the Cox construction, studying questions related to multigraded polynomial rings and irrelevant ideals rather than varieties and sheaves.
    From this perspective, the leap to smooth toric stacks is not difficult: one simply imposes fewer restrictions on the multigradings and irrelevant ideals one considers.
    For example, studying a quotient stack $\AA^n / G$ corresponds to studying a multigraded polynomial ring with empty irrelevant ideal.

    As mentioned above, the main issue for us is that the Cox construction breaks down in the non-smooth case.
    Instead, the theory of multigraded monoid rings serves as a local (but not global) model for the geometry of toric stacks.
    We will make this principle more rigorous throughout the paper.
\end{rmk}

\subsection{What has been done so far?}

Let us now recall the contributions of our illustrious forebears.
Any omissions are due to the author's ignorance.

In \cite{cox1995homogeneous}, Cox introduced a canonical quotient presentation for smooth toric varieties.
Although Cox's paper does not use the language of stacks, its ideas have been quite influential for the development of the theory of toric stacks, and variants of the Cox construction are often taken as definitions of (smooth) toric stacks.
We take an alternate approach, in which the Cox construction is instead a \emph{theorem} (\zcref{cor:cox_construction_stack}).

In \cite{borisov2005orbifold}, Borisov, Chen, and Smith used the Cox construction to introduce a notion of stacky fan and used these to construct smooth separated toric Deligne--Mumford stacks.
Iwanari (\cite{iwanari2009category}, in the case of trivial generic stabilizer) and Fantechi--Mann--Nironi (\cite{fantechi2010smooth}, in the general case) showed that every smooth separated toric Deligne--Mumford stack arises from this construction.

Tyomkin (\cite{tyomkin2012tropical}) introduced a notion of stacky fan which was studied in greater detail (and slightly generalized) by Gillam and Molcho \cite{gillam2015theory} under the name of ``KM fans.''
Gillam and Molcho showed (\cite[Theorem 3.10.7]{gillam2015theory}) that Tyomkin's stacky fans encode all separated toric Deligne--Mumford stacks with trivial generic stabilizer.
Upgrading this theorem to account for toric stacks with nontrivial generic stabilizer is not too hard once one understands how to work with stacky tori (see \zcref{prop:separated_criterion}).

Geraschenko and Satriano (\cite{geraschenko2015toric_1}) defined a notion of stacky fan and used it to construct a large class of toric Artin stacks.
As mentioned above, the sequel paper \cite{geraschenko2015toric_2} purports to prove that every toric stack arises in this way, though some steps of the proof only work in the smooth case.
Despite this, \cite{geraschenko2015toric_1} gives a useful combinatorial framework for working with the most common toric stacks appearing in practice, and we will find some of the steps in \cite{geraschenko2015toric_2} quite useful in our own classification program.

\subsection{How do we classify toric stacks?}

We should begin with our definition of ``stacky fans.''
This relies on some basic monoid-theoretic definitions, and readers unaccustomed to monoid theory should peek ahead to \zcref{sec:monoid}.

We write $\CMon$ for the category of commutative monoids and $\CMon_{/N}$ for the category of commutative monoids equipped with a homomorphism to some fixed $N \in \CMon$.

\begin{dfn}[{Special case of \zcref{dfn:stacky_cone}}]
    Let $P$ be a sharp toric monoid, let $N$ be a finitely generated free abelian group, and let $\phi: P \to N$ be a homomorphism.
    The \emph{stacky cone} corresponding to $\phi$ is the triple $\Cone(N \vec{/}_\phi P) := \big(N, \Face(P), P_{\Cone(N \vec{/}_\phi P)}\big)$ where:
    \begin{itemize}
        \item $\Face(P)$ is the poset of faces of $P$ ordered by inclusion and
        \item $P_{\Cone(N \vec{/}_\phi P)}: \Face(P) \to \CMon_{/N}$ is the functor sending a face $S \subset P$ to the natural morphism $S \hookrightarrow P \xrightarrow{\phi} N$.
    \end{itemize}
\end{dfn}

\begin{dfn}[Special case of \zcref{dfn:stacky_fan}]
    A \emph{stacky fan with free discrete cocharacter group} is a triple $\Sigma := (N_\Sigma, |\Sigma|, P_\Sigma)$ where:
    \begin{itemize}
        \item $N_\Sigma$ is a finitely generated free abelian group,
        \item $|\Sigma|$ is a finite poset, and
        \item $P_\Sigma: |\Sigma| \to \CMon_{/N_{\Sigma}}$ is a functor
    \end{itemize}
    such that:
    \begin{itemize}
        \item For all $\sigma \in |\Sigma|$, the triple $(N_\Sigma, |\Sigma|_{\leq \sigma}, P_\Sigma|_{|\Sigma|_{\leq \sigma}})$ is a stacky cone (up to the obvious notion of isomorphism of triples), and
        \item $|\Sigma|$ admits all finite meets.
    \end{itemize}
\end{dfn}

\begin{rmk}
    For the relationship between our notion of stacky fan and the other notions of stacky fan discussed above, see \zcref{sub:comparison_of_stacky_fans}.
\end{rmk}

Our main theorem (\zcref{thm:classification}) is a generalization of the following:

\begin{thm} \label{mainthm:first_statement}
    Let $\TStk^{\mathrm{tgs}}$ be the 2-category of toric stacks with trivial generic stabilizer, and let $\StkFan^{\mathrm{fdc}}$ be the 2-category\footnote{In this case, $\StkFan^{\mathrm{fdc}}$ is actually a 1-category, so it turns out that both categories in question are 1-categories. This will not be true when we drop the ``trivial generic stabilizer'' hypothesis.} of toric stacks with free discrete cocharacter group.
    There is a natural equivalence of 2-categories
    \[
        \Xfr_{-}: \StkFan^{\mathrm{fdc}} \simto \TStk^{\mathrm{tgs}}.
    \]
\end{thm}

In other words, given a stacky fan $\Sigma$ with free discrete cocharacter group, we can produce a \emph{geometrization} $\Xfr_\Sigma$, which is a toric stack with trivial generic stabilizer.
The data of the toric stack $\Xfr_\Sigma$ exactly encodes that of the stacky fan $\Sigma$.

\begin{rmk}
    We can construct an inverse to $\Xfr_{-}$ as follows.
    Suppose we are given a toric stack $\Xfr$ with trivial generic stabilizer and dense torus $T$.
    Let $\Sigma = (N_\Sigma, |\Sigma|, P_\Sigma)$ be the stacky fan where:
    \begin{itemize}
        \item $N_\Sigma = \Hom(\GG_m, T)$,
        \item $|\Sigma|$ is the collection of $T$-orbits (of $k$-points) in $\Xfr$, and
        \item For $\sigma \in |\Sigma|$, we let $\Ufr_\sigma \subset \Xfr$ be the smallest toric open substack containing the orbit corresponding to $\sigma$.
        Then $\Ufr_\sigma$ can be viewed as a commutative monoid object in the 2-category of stacks.
        We let $P_{\Sigma,\sigma} = \Hom(\AA^1, \Ufr_\sigma)$, where $\AA^1$ is viewed as a commutative monoid via multiplication.
        (The homomorphism $P_{\Sigma,\sigma} \to N_\Sigma$ is restriction along the inclusion $\GG_m \hookrightarrow \AA^1$.)
    \end{itemize}
    With this choice of $\Sigma$, we have $\Xfr \simeq \Xfr_{\Sigma}$ as toric stacks.
\end{rmk}

\subsection{What is geometrization?}

Suppose we are given a finitely generated free abelian group $N$ (with dual $M$) and a classical fan $\Sigma \subset N \otimes \RR$.
To construct a toric variety $X_\Sigma$, we perform the following steps:

\begin{enumerate}
    \item Given a cone $\sigma \in \Sigma$, construct the dual cone $\sigma^\vee = \big\{\phi: N \otimes \RR \to \RR\, \big|\,\phi|_{\sigma} \geq 0\big\}$.
    Despite the standard notation, note that $\sigma^\vee$ depends on the homomorphism $\sigma \to N \otimes \RR$ and not just on $\sigma$.
    \item Let $U_\sigma = \Spec k\angles{\sigma^\vee \cap M}$.
    Thus we may think of $\sigma^\vee \cap M$ as a combinatorial approximation to the coordinate ring of $U_\sigma$.
    \item Glue together the varieties $U_\sigma$ along open immersions induced by inclusions of subcones to get $X_\Sigma$.
\end{enumerate}

To construct toric stacks from stacky fans, we will perform a similar procedure.
However, to ensure that we account for stacky behavior, we work one category level up and consider \emph{toric categories} instead of just toric monoids.
Morally, a toric category is a symmetric monoidal category that behaves like the category of line bundles on a toric stack with morphisms given by monomials.
Thus morphisms in a toric category form sets rather than $k$-vector spaces.

The precise definition of toric categories is a tad involved, so we shall postpone it until \zcref{sec:toric_cats}.
Before then, here are a few key examples to keep in mind:

\begin{ex}
    Any finitely generated abelian group $A$ may be viewed as a toric category with objects $A$ and 
    \[
        \Map_A(a_1, a_2) = \begin{cases}
            \{ * \} & a_1 = a_2 \\
            \varnothing & \textrm{otherwise.}
        \end{cases}
    \]
\end{ex}

\begin{ex}
    Let $P$ be a toric monoid, and let $\Bsf P$ be the category with:
    \begin{itemize}
        \item one object $*$,
        \item $\Map_{\Bsf P}(*, *) = P$, and
        \item Composition given by the operation in $P$.
    \end{itemize}
    There exists a natural symmetric monoidal structure on $\Bsf P$ (given on morphisms by the operation in $P$), and this symmetric monoidal structure makes $\Bsf P$ into a toric category.
\end{ex}

\begin{ex}
    Let $\ZZ \vec{/} \NN$ be the symmetric monoidal category such that:
    \begin{itemize}
        \item $\ob \ZZ \vec{/} \NN = \ZZ$.
        \item For $n_1, n_2 \in \ZZ$,
        \[
            \Map_{\ZZ \vec{/} \NN}(n_1, n_2) = \begin{cases}
                \{ * \} & n_2 - n_1 \in \NN \\
                \varnothing & \textrm{otherwise.}
            \end{cases}
        \]
        \item The symmetric monoidal structure is given by addition in $\ZZ$.
    \end{itemize}
    Then $\ZZ \vec{/} \NN$ is a toric category.
    The underlying category of $\ZZ \vec{/} \NN$ is the category corresponding to the standard poset $(\ZZ, \leq)$.
\end{ex}

Suppose we are given a stacky fan $\Sigma = (N_\Sigma, |\Sigma|, P_\Sigma)$ with free discrete cocharacter group.
To construct a toric stack $\Xfr_\Sigma$, we proceed as follows:

\begin{enumerate}
    \setcounter{enumi}{-1}
    \item Given $\sigma \in |\Sigma|$, we encode the data of the homomorphism $P_{\Sigma,\sigma} \to N_\Sigma$ in a \emph{skew quotient} toric category $N_\Sigma \vec{/} P_{\Sigma,\sigma}$.
    \item Construct the \emph{Gale dual} toric category $\GDual(N_\Sigma \vec{/} P_{\Sigma,\sigma})$.
    \item Let $\Ufr_\sigma$ be the stack defined by
    \[
        \Ufr_\sigma(R) = \iota_0 \Fun^\otimes(\GDual(N_\Sigma \vec{/} P_{\Sigma,\sigma}), \Mod_R)
    \]
    for all affine schemes $\Spec R$ over $\Spec k$.
    Here $\iota_0$ refers to the maximal groupoid contained in a category.
    This is a toric stack with trivial generic stabilizer, and one can show that the full subcategory of $\QCoh(\Ufr_\sigma)$ spanned by line bundles is the $k$-linearization of $\GDual(N_\Sigma \vec{/} P_{\Sigma,\sigma})$.
    \item Obtain $\Xfr_\Sigma$ by gluing together the stacks $\Ufr_\sigma$ along open immersions induced by inclusions of subcones.
\end{enumerate}

\subsection{What is Gale duality?}

The construction of toric stacks in the previous subsection assumes that we are able to take ``Gale duals'' of toric categories.
Most formalisms for toric stacks rely on some version of the derived Gale duality theory developed in \cite[\S 2]{borisov2005orbifold}.
This is a useful, if abstract, machinery that allows one to produce ``dual cones'' beyond the realm of classical toric geometry.

We are able to generalize this machinery to toric categories.
Furthermore, in this language, Gale duality admits a straightforward axiomatic description in terms of the examples we have discussed above:

\begin{thm}[Categorical Gale duality, {\zcref{thm:gale_duality}}] \label{mainthm:gale_duality}
    Let $\TCat$ be the 2-category of toric categories.
    Then:
    \begin{enumerate}
        \item Gale duality defines a contravariant 2-functor $\GDual: \TCat\op \to \TCat$.
        \item There is a natural equivalence $(\GDual)^2 = \id$ (so in particular $\GDual$ is a contravariant autoequivalence).
        \item The 2-functor $\GDual$ is uniquely characterized (up to essentially unique equivalence) by the properties:
        \begin{itemize}
            \item $(\GDual)^2 = \id$,
            \item $\GDual(\ZZ) = B\ZZ$,
            \item $\GDual(\ZZ \vec{/} \NN) = \Bsf \NN$, and
            \item $\GDual$ sends the natural map $\ZZ \to \ZZ \vec{/} \NN$ to the inclusion map $\Bsf \NN \hookrightarrow \Bsf \ZZ$.
        \end{itemize}
    \end{enumerate}
\end{thm}

We are also able to translate this into a geometric theorem about ``basic'' toric stacks (i.e.\ those of the form $U / G$ where $U$ is a basic toric variety and $G$ is a diagonalizable algebraic group acting via a homomorphism to the dense torus of $U$).
See \zcref{thm:geometric_gale_duality} for details.

\begin{rmk}
    Gale duality is actually the main reason we avoid ``trivial generic stabilizer'' hypotheses throughout the paper.
    Indeed, the Gale dual of the basic toric stack $\GG_m$ is $B\GG_m$, which no longer has trivial generic stabilizer.
    Thus making Gale duality into a symmetric operation essentially forces us to consider toric stacks with nontrivial generic stabilizer.
\end{rmk}

\subsection{Organization} \label{subsec:organization}

This paper is split into four parts:
\begin{enumerate}
    \item In \zcref{part:toric_cats}, we develop the basic theory of toric categories.
    \zcref{sec:monoid} is dedicated to a quick review of the theory of toric monoids, mostly following \cite{ogus2018lectures}.
    We also include proofs of some useful claims we could not find in \cite{ogus2018lectures}.
    In \zcref{sec:2_monoid}, we discuss the more general theory of commutative 2-monoids and (strictly) abelian 2-groups, i.e.\ ``commutative monoid objects'' and ``(nice) abelian group objects'' in the 2-category of groupoids.
    This theory is used in \zcref{sec:toric_cats} to discuss toric categories and the skew quotient construction.
    In particular, we prove a structure theorem (\zcref{thm:TCat_structure_thm}) which allows us to reduce the study of toric categories to the study of toric monoids and strictly abelian 2-groups.
    Finally, in \zcref{sec:gale_duality}, we prove our Gale duality theorem (\zcref{thm:gale_duality}) and discuss some consequences.
    \item \zcref{part:fans_and_stacks} centers around defining the objects appearing in \zcref{mainthm:first_statement} and their generalizations in the case of nontrivial generic stabilizer.
    In \zcref{sec:stacky_fans}, we discuss stacky fans and construct a 2-category $\StkFan$ of stacky fans.
    Then, in \zcref{sec:defining_toric_stacks}, we define toric stacks and construct a 2-category $\TStk$ of toric stacks.
    We relate the two via a geometrization functor $\Xfr_{-}: \StkFan \to \TStk$ constructed in \zcref{sec:defining_geometrization}.
    \item In \zcref{part:geometrization_is_equivalence}, we prove that $\Xfr_{-}: \StkFan \to \TStk$ is an equivalence of categories (\zcref{thm:classification}), generalizing \zcref{mainthm:first_statement}.
    To accomplish this, we explain in \zcref{sec:equivariantization} how to quotient a toric stack $\Xfr$ by its dense stacky torus $\Tfr$, producing a toric stack $\Xfr / \Tfr$ with trivial dense torus.
    This allows us to use results of \cite{geraschenko2015toric_2} which imply that toric stacks with trivial dense torus are Zariski-locally covered by nice quotient stacks.\footnote{The local description of toric stacks in \cite{geraschenko2015toric_2} is correct; the aforementioned errors in \cite{geraschenko2015toric_2} only appear when attempting to translate between this local description and a global description.}
    In \zcref{sec:local_description}, we lift this local quotient description of $\Xfr / \Tfr$ to a local quotient description of $\Xfr$.
    Then, in \zcref{sec:proof_of_main_thm}, we combine these ingredients into a proof that $\Xfr_{-}$ is an equivalence of 2-categories.
    We also discuss some geometric consequences of Gale duality (see e.g.\ \zcref{thm:geometric_gale_duality}).
    \item \zcref{part:applications} includes further discussion of the relationship between the combinatorics of stacky fans and the geometry of toric stacks.
    This includes less technical material than the preceding sections.
    \zcref{sec:orbit_cone} focuses on the ``orbit--cone correspondence'' in the setting of general toric stacks.
    In \zcref{sec:dm_stacks}, we explore which properties of stacky fans correspond to standard properties of toric stacks (e.g.\ Deligne--Mumfordness, separatedness, smoothness).
    Finally, in \zcref{sec:picard_groups}, we explain how to compute Picard groups and class groups of toric stacks, and we use this to produce a new counterexample to the main theorem of \cite{geraschenko2015toric_2}.
    Many more generalizations of results in \cite{cls} and beyond to toric stacks are possible (and sometimes easy) in this framework; we include the preceding results merely as a sample of things that can be done.
\end{enumerate}

\subsection{Conventions and notation}

\begin{itemize}
    \item We work throughout over an algebraically closed field $k$ of characteristic zero.
    Our constructions are valid more generally, and the formalism of stacky fans makes no reference to fields at all.
    However, a key step in our proof in the case of nontrivial generic stabilizer relies on results of \cite{demeyer1992cohomological} which seem to be proved only with the above hypotheses on $k$.
    We do not know if it is possible to weaken these hypotheses.
    \item All monoids we consider will be commutative, though for terminological consistency we will generally include the adjective ``commutative'' whenever needed.
    We use additive notation throughout for most commutative monoids and symmetric monoidal categories.
    However, we use multiplicative notation in the corresponding monoid rings and use the symbol $\otimes$ for the standard symmetric monoidal structures on some classical categories (e.g.\ categories of quasicoherent sheaves).
    We trust that context will make the notation clear.
    \item Throughout this paper, we use the language of 2-categories.
    However, in the interest of readability and brevity, we deliberately avoid the finer details of 2-category theory.
    Readers who are interested in filling in these details should be able to do so without too much difficulty.
    \item If $\Cc$ is a category with a natural notion of ``duality'' functor $\Cc\op \to \Cc$, then we will typically write $D_\Cc$ for this functor.
    For example, $D_\CMon = \Hom_\CMon(-, \NN)$ and $D_\Ab = \Hom(-, \ZZ)$.
    \item If $\Cc$ is a category, we write $\iota_0 \Cc$ for the maximal groupoid contained in $\Cc$.
    More generally, if $\Cc$ is a 2-category, we write $\iota_1 \Cc$ for the 2-category obtained from $\Cc$ by applying $\iota_0$ to all hom-categories.
    \item If $\Cc$ is a (2- or 1-)category and $c \in \Cc$, we write $\Cc_{/c}$ for the (2- or 1-)category of objects of $\Cc$ equipped with morphisms to $c$.
    \item We use the symbol ``$=$'' for canonical equivalence / isomorphism and the symbol ``$\simeq$'' for non-canonical equivalence.
    \item If $\Cc$ and $\Dc$ are cocomplete categories, we will write $\Fun^L(\Cc, \Dc)$ for the category of colimit-preserving functors $\Cc \to \Dc$.
    We will also use variants of this notation, e.g.\ $\Fun^{L,\otimes}(\Cc, \Dc)$ for the category of colimit-preserving symmetric monoidal structures (for some symmetric monoidal structures $\otimes_\Cc$ and $\otimes_\Dc$).
    \item We use cohomological indexing for chain complexes.
\end{itemize}

\subsection{Acknowledgments}

The author would like to thank Remy van Dobben de Bruyn for answering a MathOverflow question of the author's and pinpointing an error in \cite{geraschenko2015toric_2} which served as the impetus for this project.\footnote{See \href{https://mathoverflow.net/questions/510252/}{https://mathoverflow.net/questions/510252/} for the relevant discussion, which occurred before the author knew about the counterexample mentioned in \cite{gillam2015theory}.}
The author also benefited from mathematical conversations with and encouragement from Swapnil Garg, Daigo Ito, Roy Joshua, David Nadler, Hsian-Hua Tseng and Constantin Teleman.
Finally, the author would like to thank Boaz Moerman for comments on the exposition in a previous version of this paper.

Part of this research was supported by the NSF Research Training Group ``Arithmetic, Combinatorics, and Topology of Algebraic Varieties'' at The Ohio State University (DMS-2231565).

\subsection{AI disclosure}

The author used ChatGPT and Copilot for literature review and foundational questions.
The mathematical content and exposition here is entirely due to the author, who takes full responsibility for any mistakes.

\part{The algebra of toric categories} \label{part:toric_cats}

\section{A quick review of toric monoids} \label{sec:monoid}

Let $(M, +)$ be a commutative monoid with group of units $M^\pm$ and group completion $M^\gp$.
We say $M$ is:
\begin{itemize}
    \item \emph{sharp} if $M^\pm = 0$,
    \item \emph{finitely generated} if there exists a surjection $\NN^n \twoheadrightarrow M$ for some $n$,
    \item \emph{integral} if the natural morphism $M \to M^\gp$ is injective,
    \item \emph{saturated} if $M$ is integral and, for every $m \in M^\gp$, if $d \cdot m \in M$ for some nonzero natural number $d$, then $m \in M$,
    \item \emph{fs} if $M$ is finitely generated, integral, and saturated, and
    \item \emph{toric} if $M$ is fs and $M^\gp$ is free.
\end{itemize}

\begin{rmk}
    Equivalently, $M$ is toric if and only if $M$ is saturated and embeds as a finitely generated submonoid of $\ZZ^n$ for some $n$.
\end{rmk}

In this section we review the basic theory of toric monoids, using \cite[Chapter I]{ogus2018lectures} as a reference.
We write:
\begin{itemize}
    \item $\CMon$ for the category of commutative monoids,
    \item $\CMon^\sharp$ for the category of sharp commutative monoids,
    \item $\CMon^\toric$ for the category of toric monoids, and
    \item $\CMon^{\sharp,\toric}$ for the category of sharp toric monoids.
\end{itemize}

\subsection{Faces}

Within any commutative monoid $M$ we have the following standard notions:
\begin{itemize}
    \item An \emph{ideal} of $M$ is a subset $I \subset M$ such that, if $i \in I$ and $m \in M$, then $i + m \in I$.
    \item A \emph{prime ideal} of $M$ is an ideal $\pfr \subset M$ such that, if $m_1 + m_2 \in \pfr$, then $m_1 \in \pfr$ or $m_2 \in \pfr$.
    \item A \emph{face} of $M$ is a submonoid $S \subset M$ such that, if $m_1 + m_2 \in S$, then both $m_1$ and $m_2$ are in $S$.
    For any $m \in M$, we write $\angles{m}$ for the smallest face of $M$ containing $m$.
\end{itemize}

The map $\pfr \mapsto M \setminus \pfr$ gives an order-reversing bijection from the poset $\Spec(M)$ of prime ideals in $M$ to the poset $\Face(M)$ of faces of $M$ (both ordered via inclusion).
We will mostly focus on understanding $\Face(M)$, since that is what shows up in the classical toric formalism.

The poset $\Face(M)$ has a unique maximal element $M$ and a unique minimal element $M^\pm$.
Furthermore, when $M$ is toric, or more generally finitely generated, $\Face(M)$ is finite (see e.g.\ \cite[Proposition I.1.4.7]{ogus2018lectures} for the integral case, which is all we need).

The functor $\Face$ preserves products:

\begin{prop} \label{prop:faces_of_product}
    Let $M_1$ and $M_2$ be commutative monoids.
    Then the map
    \begin{align*}
    \Face(M_1) \times \Face(M_2) &\to \Face(M_1 \oplus M_2) \\
    (S_1, S_2) &\mapsto S_1 \oplus S_2
    \end{align*}
    is an isomorphism of posets.
\end{prop}

\begin{proof}
    It is clear that the map in question is injective and order-preserving.
    To see that this map is surjective, suppose $S \subset M_1 \times M_2$ is a face.
    Then
    \[
        S_1 = \bset{m_1 \in M_1}{(m_1, 0) \in S}
    \]
    is a face of $M_1$, and the analogous definition for the other coordinate gives a face $S_2 \subset M_2$.
    For any $(m_1, m_2) \in S$, we must also have $(m_1, 0) \in S$ and $(0, m_2) \in S$ by the face condition.
    Thus $m_1 \in S_1$ and $m_2 \in S_2$, so $(m_1, m_2) \in S_1 \oplus S_2$.
    It follows that $S = S_1 \oplus S_2$, proving surjectivity.
\end{proof}

\subsection{Localizations and quotients by faces}

If $M$ is a commutative monoid and $S \subset M$ is a face, then we can construct a \emph{localization} $M[-S]$ and a \emph{quotient} $M / S$ satisfying the expected universal properties.
These can both be written more explicitly:
the localization $M[-S]$ is
\[
    M[-S] = (M \times S) / \big((m_1, s_1) \sim (m_2, s_2) \textrm{ if there exists $s_3 \in S$ such that } m_1 + s_2 + s_3 = m_2 + s_1 + s_3\big)
\]
and the quotient $M / S$ is
\[
    M / S = M / \big(m_1 \sim m_2 \textrm{ if there exist $s_1, s_2 \in S$ such that } m_1 + s_1 = m_2 + s_2\big).
\]
In both cases the addition operation is defined termwise (i.e.\ $[m_1, s_1] + [m_2, s_2]  = [m_1 + m_2, s_1 + s_2]$ in $M[-S]$ and $[m_1] + [m_2] = [m_1 + m_2]$ in $M / S$).

\begin{ex}
    Let $M$ be any commutative monoid, and let $S = M$.
    The localization $M[-M]$ is the same as the group completion $M^\gp$.
    The quotient $M / M$ is of course $0$.
\end{ex}

\begin{ex}
    Let $M$ be any commutative monoid, and let $S = M^\pm$.
    The localization $M[-M^\pm]$ is $M$ itself.
    We write $M^\sharp$ for the quotient $M / M^\pm$; it is the universal sharp commutative monoid with a homomorphism from $M$.
\end{ex}

The following standard result gives a useful connection between localizations and faces:

\begin{prop} \label{prop:localization_and_quotient_by_face}
    Let $M$ be a commutative monoid, and let $S \subset M$ be a face.
    Then $S[-S] = S^\gp$ may be identified with $(M[-S])^\pm$, and 
    \[
        M / S = M[-S] / S[-S] = M[-S] / (M[-S])^\pm.
    \]
\end{prop}

\begin{proof}
    We first claim that the natural map $S[-S] \to M[-S]$ exhibits $S[-S]$ as a face in $M[-S]$.
    It is clear from the explicit construction of $S[-S]$ and $M[-S]$ that this natural map is injective.
    To see that the face condition holds, let let $[m_1, s_1], [m_2, s_2] \in M[-S]$, and suppose $[m_1 + m_2, s_1 + s_2] \in S[-S]$.
    Then there exists $s_3 \in S$ such that $m_1 + m_2 + s_3 \in S$.
    Because $S$ is a face, this implies $m_1 + m_2 \in S$, so $m_1 \in S$ and $m_2 \in S$.
    This implies that $[m_1, s_1], [m_2, s_2] \in S[-S]$, so $S[-S]$ is a face in $M[-S]$.

    Every element of $S[-S]$ is invertible, so $S[-S] \subset (M[-S])^\pm$.
    Since $S[-S]$ is a face of $M[-S]$ and $(M[-S])^\pm$ is the minimal face of $M[-S]$, we see that in fact $S[-S] = (M[-S])^\pm$, as desired.

    The identification $M[-S] / S[-S] = M / S$ follows from the universal property of quotients.
    More precisely, for any $N$, there are natural identifications between:
    \begin{itemize}
        \item $\Hom(M[-S] / S[-S], N)$,
        \item The set of homomorphisms $M[-S] \to N$ sending $S[-S]$ to $0$,
        \item The set of homomorphisms $M \to N$ sending $S$ to $0$, and
        \item $\Hom(M / S, N)$.
    \end{itemize}
    Thus $M[-S] / S[-S] = M / S$ by Yoneda.
\end{proof}

\begin{rmk}
    In particular, we see that any quotient $M / S$ (with $M$ a commutative monoid and $S$ a face of $M$) is sharp.
    This is proved directly in \cite[Proposition I.1.4.2(4)]{ogus2018lectures} in the integral case (which is all we actually need).
\end{rmk}

By \cite[Theorem I.2.1.17(4)]{ogus2018lectures}, if $M$ is toric, then $M[-S]$ must also be toric.
The analogous result also holds for quotients by faces:

\begin{prop} \label{prop:quotient_by_face_is_toric}
    Let $M$ be a toric monoid, and let $S \subset M$ be a face.
    Then $M / S$ is toric.
\end{prop}

\begin{proof}
    Write $M / S = M[-S] / S[-S]$.
    Thus, replacing $M$ by $M[-S]$ and $S$ by $S[-S]$, we reduce to the case $S = M^\pm$.
    In this case the result holds by \cite[Propositions I.1.3.3 and I.1.3.5]{ogus2018lectures}.
\end{proof}

We can also sometimes go in the reverse direction, using claims about $M^\sharp = M / M^\pm$ to infer claims about $M$.
In particular, we can recognize when $M$ is toric using this approach:

\begin{prop} \label{prop:check_toric_via_units_and_sharp}
    Let $M$ be a commutative monoid, and suppose:
    \begin{itemize}
        \item $M^\pm$ is a finitely generated free abelian group,
        \item $M^\pm$ acts freely on $M$ by translation, and
        \item $M^\sharp$ is toric.
    \end{itemize}
    Then $M$ is toric.
\end{prop}

\begin{proof}
    We need to check that:
    \begin{itemize}
        \item $M$ is finitely generated. This is \cite[Proposition I.2.1.1]{ogus2018lectures}.
        \item $M$ is integral.
        This is \cite[Proposition I.1.3.3]{ogus2018lectures}.
        \item $M$ is saturated.
        This is \cite[Proposition I.1.3.5(4)]{ogus2018lectures}.
        \item $M^\gp$ is free.
        By \cite[Proposition I.1.3.5(1)]{ogus2018lectures}, we have $M^\gp / M^\pm = (M^\sharp)^\gp$, which is free by hypothesis.
        Thus $M^\gp$ is an extension of the free abelian group $M^\gp / M^\pm$ by the free abelian group $M^\pm$, hence is itself free. \qedhere
    \end{itemize}
\end{proof}

\subsection{Duality}

Let $M$ be a commutative monoid.
The \emph{dual monoid} of $M$ is $D_\CMon(M) = \Hom(M, \NN)$.
We say that $M$ is \emph{dualizable} if the natural map $M \to D_\CMon^2(M)$ is an isomorphism.

By \cite[Theorem I.2.2.3]{ogus2018lectures}, we have the following key result: 
\begin{itemize}
    \item A finitely generated monoid $M$ is dualizable if and only if $M$ is sharp and toric.
\end{itemize}
Thus $D_\CMon$ defines an antiequivalence $(\CMon^{\sharp,\toric})\op \simto \CMon^{\sharp,\toric}$.

\begin{rmk}
    In \zcref{sub:Gale_duality}, we will see that there is a reasonable involutive definition of the ``dual'' of a non-sharp toric monoid.
    However, this requires passing to the world of toric categories.
\end{rmk}

If $M$ is a sharp toric monoid and $S \subset M$ is a face, then we may define a face $S^\bot \subset D_\CMon(M)$ via
\[
    S^\bot = \bset{\phi \in \Hom(M, \NN)}{\phi(s) = 0 \textrm{ for all } s \in S}.
\]
By \cite[Corollary I.2.2.4]{ogus2018lectures}, this defines an antiequivalence
\[
(-)^\bot: \Face(M)\op \simto \Face(D_\CMon(M))
\]
satisfying $S^{\bot\bot} = S$ for all faces $S \subset M$.

\subsection{A characterization of sharp toric monoids}

The following result can be useful when working with sharp toric monoids.

\begin{prop} \label{prop:sharp_toric_characterization}
    The following are equivalent for a monoid $M$:
    \begin{enumerate}
        \item $M$ is sharp and toric.
        \item $M$ is fs and embeds into $\NN^n$ for some $n$.
        \item $M$ is sharp and fs.
    \end{enumerate}
\end{prop}

\begin{proof}
    $(1) \Rightarrow (2)$.
    This is a consequence of \cite[Corollary I.2.2.7]{ogus2018lectures}.

    $(2) \Rightarrow (3)$.
    Immediate, since all submonoids of the sharp monoid $\NN^n$ are sharp.

    $(3) \Rightarrow (1)$.
    If $M$ is sharp, then $M^\gp$ is torsion-free by \cite[Proposition I.1.3.5(2)]{ogus2018lectures}.
    If $M$ is also finitely generated, then this implies $M^\gp$ is free.
\end{proof}

\subsection{Limits of toric monoids}

We will find the following result useful.

\begin{prop} \label{prop:toric_monoid_finite_limit}
    The category $\CMon^\toric$ admits finite limits, and these agree with finite limits in $\CMon$.
    The same holds for $\CMon^{\sharp,\toric}$.
\end{prop}

\begin{proof}
    Arbitrary limits of sharp monoids are sharp, so it suffices to prove the claim for $\CMon^\toric$.
    Every finite limit may be written in terms of final objects and fiber products.
    The final object of $\CMon$ is $0$, which is clearly toric.
    Thus it remains to show that $\CMon^\toric$ is closed under fiber products.

    Let
    \[
    \begin{tikzcd}
        M \times_P N \rar \dar & M \dar \\
        N \rar & P
    \end{tikzcd}
    \]
    be a Cartesian diagram in $\CMon$ with $M$, $N$, and $P$ toric.
    Then $M \times_P N$ is fs by \cite[Theorem I.2.1.17(6)]{ogus2018lectures}.

    We still need to show that $(M \times_P N)^\gp$ is torsion-free.
    But this is not hard: because $M \times_P N \subset M \times N$ and all monoids involved are integral, we get an inclusion 
    \[
    (M \times_P N)^\gp \hookrightarrow (M \times N)^\gp = M^\gp \times N^\gp
    \]
    by \cite[Remark I.1.3.2]{ogus2018lectures}.
    The groups $M^\gp$ and $N^\gp$ are both torsion-free by hypothesis, so $(M \times_P N)^\gp$ must also be torsion-free, and we are done.
\end{proof}

\begin{rmk}
    Combining \zcref{prop:toric_monoid_finite_limit} with the antiequivalence $D_\CMon: (\CMon^{\sharp,\toric})\op \simto \CMon^{\sharp,\toric}$, we see that $\CMon^{\sharp,\toric}$ also admits finite colimits.
    These may be computed by
    \[
        \colim_i M_i = D_\CMon\big(\lim_i D_\CMon(M_i)\big).
    \]
    Note that these \emph{do not} agree with finite colimits in $\CMon$: the colimits in $\CMon$ may not be sharp or integral.
    As such, this result is not nearly as helpful as the analogous result for limits.
\end{rmk}

\subsection{Saturated submonoids}

We will need the following (standard) relative notion of saturation:

\begin{dfn}
    Let $A$ be an abelian group.
    A submonoid $M \subset A$ is \emph{saturated in $A$} if, whenever $a \in A$ satisfies $d \cdot a \in M$ for some nonzero natural number $d$, then $a \in M$.
\end{dfn}

In particular, an integral monoid $M$ is saturated if and only if $M$ is a saturated submonoid of $M^\gp$.
Provided that $M$ is saturated (abstractly), the question of whether $M$ is saturated in an ambient abelian group $A$ can be rephrased entirely in terms of groups:

\begin{prop} \label{prop:saturated_submonoid_criterion}
    Let $A$ be an abelian group, and let $M$ be a submonoid of $A$.
    Then the following are equivalent:
    \begin{itemize}
        \item $M$ is saturated in $A$.
        \item $M$ is saturated in $M^\gp$ and $M^\gp$ is saturated in $A$.
        \item $M$ is saturated in $M^\gp$ and $A / M^\gp$ is torsionfree.
    \end{itemize}
\end{prop}

\begin{proof}
    $(1) \Rightarrow (2)$.
    Suppose $M$ is saturated in $A$.
    Because $M^\gp \subset A$, it is clear that $M$ is saturated in $M^\gp$.
    
    Now let $a \in A$ satisfy $d \cdot a \in M^\gp$ for some $d > 0$, so $d \cdot a = m_1 - m_2$ for some $m_1, m_2 \in M$.
    Then $d \cdot a + m_2 = m_1$, so $d \cdot (a + m_2) = d \cdot m_1 + (d - 1) \cdot m_2 \in M$.
    Because $M$ is saturated in $A$, this implies $a + m_2 \in M^\gp$ and thus $a \in M^\gp$.
    Thus $M^\gp$ is saturated in $A$.

    $(2) \Rightarrow (1)$, $(2) \Leftrightarrow (3)$.
    These are straightforward from the definitions.
\end{proof}

The finitely generated submonoids of a finitely generated free abelian group $A$ which are saturated in $A$ are particularly nice.
The following holds by \cite[Exercise 1.3.4 and Proposition 1.2.12]{cls}:

\begin{prop} \label{prop:sat_submon_of_lattice}
    Let $A$ be a finitely generated free abelian group, and let $M \subset A$ be a submonoid.
    Then:
    \begin{enumerate}
        \item $M$ is saturated in $A$ if and only if there exists a rational polyhedral cone $\ul{\sigma} \subset A \otimes \RR$ such that $M = \ul{\sigma} \cap A$.
        \item Such a $\ul{\sigma}$ is unique if it exists.
        \item $M$ is sharp if and only if the corresponding cone $\ul{\sigma}$ is strongly convex.
    \end{enumerate}
\end{prop}

\section{Commutative 2-monoids} \label{sec:2_monoid}

This section is devoted to understanding the theory of symmetric monoidal groupoids, or \emph{commutative 2-monoids}.
We won't have much use for general commutative 2-monoids, which can be quite hard to control.
However, the language of commutative 2-monoids gives a good framework for relating objects which interest us more: toric monoids (discussed above) and abelian 2-groups (discussed below).

\subsection{Some basic definitions and examples}

\begin{dfn}
    A \emph{commutative 2-monoid} is a small symmetric monoidal groupoid $(\Mc, \oplus)$.
    These form the objects of a full sub-2-category $\tCMon \subset \iota_1 \SMCat$.
\end{dfn}

Commutative 2-monoids abound in nature: for any small symmetric monoidal category $\Cc$, the maximal groupoid $\iota_0 \Cc$ is a commutative 2-monoid.
We will give more explicit examples shortly.
First, some more definitions:

\begin{dfn}
    Let $\Mc$ be a commutative 2-monoid.
    The \emph{component monoid} of $\Mc$, written $\pi_0(\Mc)$, is the commutative monoid of isomorphism classes of objects of $\Mc$.
\end{dfn}

\begin{rmk}
    Our use of the symbol $\pi_0$ (and many related symbols in the future) is meant to evoke connections with homotopy theory.
    Indeed, under the ``homotopy hypothesis,'' the 2-category of groupoids embeds into the $\infty$-category of homotopy types, i.e.\ CW-complexes up to homotopy equivalence.
    Commutative 2-monoids are then a very special case of the so-called ``$\EE_\infty$-spaces'' appearing in homotopy theory.
    Under this identification, our $\pi_0$ corresponds to taking the monoid of connected components of an $\EE_\infty$-space.
    See \cite{htt, ha}, among many other references, for a deep treatment of this story.
    We would like to stress that this paper does not require any $\infty$-categorical sophistication.
\end{rmk}

\begin{dfn}
    A commutative 2-monoid $\Ac$ is \emph{grouplike}, or an \emph{abelian 2-group}, if $\pi_0(\Ac)$ is an abelian group.
    We write $\tAb$ for the full subcategory of $\tCMon$ spanned by the abelian 2-groups.
\end{dfn}

A commutative 2-monoid $(\Mc, \oplus)$ is an abelian 2-group if and only if every object of $\Mc$ is $\oplus$-invertible up to isomorphism.
As one might expect, abelian 2-groups are much easier to study than general commutative 2-monoids.

Here are our promised examples:

\begin{ex} \label{ex:CMon_to_2CMon}
    Let $(M, +)$ be a commutative monoid.
    We may view $M$ as a commutative 2-monoid (where the corresponding category has set of objects $M$, morphisms only given by the identity, and symmetric monoidal operation $\oplus$ given by $+$).
    The component monoid of this is the original commutative monoid $M$, and this commutative 2-monoid is an abelian 2-group if and only if $M$ is an abelian group.

    This construction defines a fully faithful functor $\CMon \hookrightarrow \tCMon$.
    We will say that a commutative 2-monoid is (essentially) \emph{discrete} if it lies in the image of this functor.
\end{ex}

\begin{ex} \label{ex:BA}
    Let $A$ be any abelian group.
    Then there is an abelian 2-group $\Bsf A$ in which:
    \begin{itemize}
        \item $\ob \Bsf A = \{*\}$,
        \item $\Map_{\Bsf A}(*, *) = M$,
        \item $\oplus: \Bsf A \times \Bsf A \to \Bsf A$ is trivial on objects and is given on morphisms by addition in $A$, 
        \item All structure morphisms are trivial.
    \end{itemize}
    This defines a fully faithful functor $\Bsf: \Ab \to \tAb$.
\end{ex}

\begin{rmk} \label{rmk:BM}
    The above construction produces a symmetric monoidal category $\Bsf M$ for every commutative monoid $M$, though $\Bsf M$ is a groupoid if and only if $M$ is a group.
    The induced functor $\Bsf: \CMon \to \SMCat$ is fully faithful.
\end{rmk}

The following examples highlight the possible complexities of general commutative 2-monoids and abelian 2-groups.

\begin{ex} \label{ex:free_2_monoid}
    Let $\Fc$ be the groupoid of finite sets and bijections.
    Then $\Fc$ may be viewed as a commutative $2$-monoid\footnote{Modulo size issues, which we can ignore if we choose a representative of each isomorphism class of finite sets.} with operation given by disjoint union.
    More concretely:
    \begin{itemize}
        \item The objects of $\Fc$ may be identified with sets $\{ 1, \dots, n \}$ for $n \in \NN$.
        \item For all $n$, we have $\Aut_\Fc(\{1, \dots, n\}) = \Sigma_n$, the symmetric group on $n$ letters.
        \item The symmetric monoidal structure is given on objects by addition in $\NN$ and given on morphisms by concatenation of permutations.
    \end{itemize}
    One can show that $\Fc$ is the free commutative $2$-monoid on one generator.
    Note that the automorphism groups of objects of $\Fc$ vary with the isomorphism class of the object and are typically nonabelian.
\end{ex}

\begin{ex} \label{ex:free_2_group}
    Let $\mu_2 = \{ \pm 1 \}$, and consider the groupoid $\ZZ \times \Bsf \mu_2$.
    This groupoid admits a trivial symmetric monoidal structure (i.e.\ $\oplus$ is given by addition in $\ZZ$ and $\mu_2$, and all structure maps are the identity).
    However, there is also a distinct symmetric monoidal structure on $\ZZ \times \Bsf \mu_2$ with the same underlying monoidal structure, but with nontrivial braiding
    \[
        \beta_{m,n} = (-1)^{mn} \in \mu_2 = \Aut(m \oplus n)
    \]
    Let us write $\Fc^\gp$ for $\ZZ \times \Bsf \mu_2$ with this symmetric monoidal structure.
    Then $\Fc^\gp$ is an abelian 2-group; in fact, it is the free abelian 2-group on one generator (see \cite[Proposition 3.1]{johnson2012modeling}).
    
    The inclusion functor $\Bsf \mu_2 \hookrightarrow \Fc^\gp$ and the projection functor $\Fc^\gp \to \ZZ$ are naturally symmetric monoidal.
    However, there is no symmetric monoidal functor $\ZZ \to \Fc^\gp$ splitting this projection.
    We may thus think of $\Fc^\gp$ as a non-split extension of $\ZZ$ by $\Bsf \mu_2$.
\end{ex}

Given a homomorphism of commutative 2-monoids $\phi: \Mc \to \Mc'$, we may construct a \emph{fiber} $\fib(\phi)$ and a \emph{cofiber} $\cofib(\phi)$ satisfying the following universal properties:
\[
    \Hom_\tCMon(\Nc, \fib(\phi)) = \Hom_\tCMon(\Nc, \Mc) \times_{\Hom_\tCMon(\Nc, \Mc')} 0
\]
and
\[
    \Hom_\tCMon(\cofib(\phi), \Nc) = \Hom_{\tCMon}(\Mc', \Nc) \times_{\Hom_{\tCMon}(\Mc, \Nc)} 0
\]
for all $\Nc \in \tCMon$.
These are 2-categorical analogues of the kernel and cokernel, respectively.
We will eventually discuss these constructions in more detail in the special cases relevant to toric theory.

There is a natural notion of ``action'' of a symmetric monoidal category $\Mc$ (e.g.\ a commutative 2-monoid) on a category $\Cc$, generalizing the standard notion of monoid actions on sets.
Such an action is given by a functor $\Mc \times \Cc \to \Cc$ together with associativity and unitality isomorphisms satisfying some coherence conditions spelled out in \cite[Definition B.12]{fantechi2010smooth}.\footnote{Actually, \cite{fantechi2010smooth} works in the setting of sheaves of abelian 2-groups, but the required natural isomorphisms and coherence conditions are the same in both settings.}

\begin{ex}
    Given a homomorphism of commutative 2-monoids $\phi: \Mc \to \Mc'$, we obtain an action of $\Mc$ on $\Mc'$ as in \cite[Proposition B.14]{fantechi2010smooth}.
\end{ex}

\begin{ex}
    Let $M$ be a discrete commutative monoid.
    An action of $\Bsf M$ on a category $\Cc$ corresponds to a monoidal functor $\Bsf M \to \Fun(\Cc, \Cc)$, or equivalently to a monoid homomorphism $M \to \End(\id_\Cc)$.
\end{ex}

\subsection{Strictly abelian 2-groups}

We are particularly interested in the following class of well-behaved abelian 2-groups:

\begin{dfn}
    An abelian 2-group $\Ac$ is \emph{strictly abelian} if its symmetry isomorphism $\beta$ satisfies $\beta_{a,a} = \id_a$ for all $a \in \Mc$.
    Let $\tAb^\st$ be the full subcategory of $\tAb$ spanned by the strictly abelian 2-groups.
\end{dfn}

\begin{ex} \label{ex:Fgp_not_strict}
    The abelian 2-group $\Fc^\gp$ of \zcref{ex:free_2_group} is not strictly abelian, since its symmetry isomorphism $\beta$ satifsies $\beta_{m,n} = -1$ when $m$ and $n$ are both odd.
\end{ex}

The easiest way to perform computations with strictly abelian 2-groups is to use the following result to transform the category theory of strictly abelian 2-groups into the homological algebra of abelian groups.
Write $\Dsf^{[-1,0]}(\ZZ)$ for the (dg-)derived category of chain complexes of abelian groups with cohomology concentrated in degrees $0$ and $-1$.
Recall that any dg-category may be viewed as a 2-category where $0$-morphisms are cycles in the mapping complexes and $1$-morphisms are chain homotopies between these cycles.

\begin{prop}[{Special case of \cite[Corollaire 1.4.17]{deligne1971formule}}] \label{prop:deligne_classification}
    There is an equivalence of 2-categories 
    \[
        \Omega^\infty: \Dsf^{[-1,0]}(\ZZ) \simto \tAb^\st
    \]
    such that, for any abelian group $A$, we have $\Omega^\infty(A[0]) = A$ and $\Omega^\infty(A[1]) = \Bsf A$.
\end{prop}

The equivalence of \zcref{prop:deligne_classification} allows us to classify the strictly abelian 2-groups.
This classification is well-known to the experts (see \cite[Corollaire 3.3]{sinh1982categories}), but we provide a detailed statement and proof for future reference.

\begin{prop} \label{prop:strict_trivialization}
    For an abelian 2-group $\Ac$, the following are equivalent:
    \begin{enumerate}
        \item $\Ac$ is strictly abelian.
        \item There exists a symmetric monoidal equivalence $\Ac \simeq \pi_0(\Ac) \times \Bsf \Aut_\Ac(0)$.
    \end{enumerate}
    Furthermore, if $\Ac$ is strictly abelian, the set of symmetric monoidal structures (up to natural isomorphism) on a fixed equivalence of underlying categories $\Ac \simeq \pi_0(\Ac) \times \Bsf \Aut_\Ac(0)$ is an $\Ext^1_\ZZ(\pi_0(\Ac), \Aut_\Ac(0))$-torsor.
\end{prop}

\begin{proof}
    $(1) \Rightarrow (2)$.
    This follows from \zcref{prop:deligne_classification} and the fact that every object of $\Dsf^{[-1,0]}(\ZZ)$ splits as a direct sum of its cohomology objects (see e.g.\ \cite[\S 1.6]{krause2005derived}).
    
    $(2) \Rightarrow (1)$.
    Immediate.
    
    For the ``furthermore'' claim, we use \zcref{prop:deligne_classification} to identify $\Ac$ with the object 
    \[
    \pi_0(\Ac)[0] \oplus \Aut_\Ac(0)[1] \in \Dsf^{[-1,0]}(\ZZ).
    \]
    The set of symmetric monoidal equivalences $\Ac \simeq \pi_0(\Ac) \times \Bsf \Aut_\Ac(0)$ is a torsor for the group
    \[
        G = \pi_0\big(\Aut_{\Dsf^{[-1,0]}(\ZZ)}\big(\pi_0(\Ac)[0] \oplus \Aut_\Ac(0)[1]\big)\big).
    \]
    This $G$ is the subgroup of invertible elements of the matrix algebra 
    \[
        H^0\Big(\End_{\Dsf(\ZZ)}\big(\pi_0(\Ac)[0] \oplus \Aut_\Ac(0)[1]\big)\Big) = \begin{bmatrix}
            \Hom_\ZZ\big(\pi_0(\Ac), \pi_0(\Ac)\big) & \Ext_\ZZ^1\big(\pi_0(\Ac), \Aut_\Ac(0)\big) \\
            0 & \Hom_\ZZ\big(\Aut_\Ac(0), \Aut_\Ac(0)\big)
        \end{bmatrix}.
    \]
    Fixing the equivalence of underlying categories corresponds to fixing elements of $\Hom_\ZZ\big(\pi_0(\Ac), \pi_0(\Ac)\big)$ and $\Hom_\ZZ\big(\Aut_\Ac(0), \Aut_\Ac(0)\big)$ (which we can take to be the identity), so that the set of symmetric monoidal structures on a fixed equivalence is a torsor for the group
    \[
        \begin{bmatrix}
            1 & \Ext_\ZZ^1(\pi_0(\Ac), \Aut_\Ac(0)) \\
            0 & 1
        \end{bmatrix} = \Ext_\ZZ^1(\pi_0(\Ac), \Aut_\Ac(0)). \qedhere
    \]
\end{proof}

We can use this to show:

\begin{prop} \label{prop:2Ab_limits}
    The functor
    \[
        T: \Dsf^{[-1,\infty)}(\ZZ) \xrightarrow{\tau^{\leq 0}} \Dsf^{[-1,0]}(\ZZ) \simeq \tAb
    \]
    preserves limits.
    In particular, $\tAb$ admits arbitrary limits, and these can be computed by applying $T$ to the corresponding limits computed in $\Dsf(\ZZ)$.
\end{prop}

\begin{proof}
    The functor $\tau^{\leq 0}$ is, by definition, the right adjoint to the inclusion $\Dsf^{[-1,0]}(\ZZ) \hookrightarrow \Dsf^{[-1,\infty)}(\ZZ)$, hence must preserve limits.
    It follows that $T$ also preserves limits.
    The ``in particular'' claim follows immediately because $\Dsf^{[-1,\infty)}(\ZZ)$ admits arbitrary limits (which agree with the corresponding limits in $\Dsf(\ZZ)$).
\end{proof}

\subsection{\texorpdfstring{$\tTerm$}{2Term}}

By \zcref{prop:deligne_classification}, the study of strictly abelian 2-groups reduces to homological algebra in $\Dsf^{[-1,0]}(\ZZ)$.
Of course, homological algebra in $\Dsf^{[-1,0]}(\ZZ)$ contains all the complexities of the study of arbitrary abelian groups.
In practice, we only care about objects satisfying a strong finiteness condition:

\begin{dfn}
    Let $\tTerm$ denote the (2-category corresponding to the) dg-category of chain complexes of abelian groups $F^\bullet$ such that:
    \begin{itemize}
        \item $F^0$ and $F^{-1}$ are finitely generated free abelian groups, and
        \item $F^i = 0$ for all $i \neq 0, -1$.
    \end{itemize}
\end{dfn}

\begin{rmk} \label{rmk:2term_vs_DZ}
    The natural map $\tTerm \to \Dsf^{[-1,0]}(\ZZ)$ is fully faithful.
    The image of this map consists of the complexes $A^\bullet$ such that:
    \begin{itemize}
        \item $H^0(A^\bullet)$ is a finitely generated abelian group,
        \item $H^{-1}(A^\bullet)$ is a finitely generated free abelian group, and
        \item $H^i(A^\bullet) = 0$ for all $i \neq 0, -1$.
    \end{itemize}
\end{rmk}

The 2-category $\tTerm$ inherits some nice categorical properties from $\Dsf(\ZZ)$:

\begin{prop} \label{lem:2term_finite_limits}
    The 2-category $\tTerm$ admits finite limits.
    These can be computed by applying $\tau^{\leq 0}$ to the corresponding limits in $\Dsf(\ZZ)$.
\end{prop}

\begin{proof}
    By \zcref{prop:2Ab_limits}, we know that $\Dsf^{[-1,0]}(\ZZ) \simeq \tAb$ admits arbitrary limits, computed by applying $\tau^{\leq 0}$ to the corresponding limits in $\Dsf(\ZZ)$.
    We need to show that the full subcategory $\tTerm \subset \Dsf^{[-1,0]}(\ZZ)$ is closed under finite limits.
    Finite limits in $\Dsf(\ZZ)$ can all be written in terms of fibers and finite products, so it suffices to show that the truncation $\tau^{\leq 0}$ of a fiber or finite product of objects of $\tTerm \subset \Dsf(\ZZ)$ still lies in $\tTerm$.
    The claim is obvious for finite products (no truncation is necessary), so we consider the case of fibers.

    Let $\phi: F^\bullet \to G^\bullet$ be a morphism in $\tTerm$.
    Then $\fib(\phi)$ (computed in $\Dsf(\ZZ)$) can be written as a chain complex
    \[
    \begin{tikzcd}
        \dots \rar & 0 \rar & F^{-1} \rar & F^0 \oplus G^{-1} \rar & G^0 \rar & 0 \rar & \dots
    \end{tikzcd}
    \]
    where the copy of $F^{-1}$ is in degree $-1$.
    Thus $H^{-1}(\fib(\phi))$ is finitely generated free and $H^0(\fib(\phi))$ is finitely generated.
    These are the only nonzero cohomology groups of $\tau^{\leq 0} \fib(\phi)$, so we see by \zcref{rmk:2term_vs_DZ} that $\tau^{\leq 0} \fib(\phi) \in \tTerm$, as needed.
\end{proof}

\subsection{Path 2-monoids}

The following construction of commutative 2-monoids gives a useful invariant of symmetric monoidal categories:

\begin{dfn}
    Let $(\Cc, \oplus)$ be a small symmetric monoidal category.
    The \emph{path $2$-monoid} of $\Cc$ is the commutative $2$-monoid $\PO \Cc := \iota_0(\Cc_{0/})$, where the symmetric monoidal structure on $\PO \Cc$ is given by $\oplus$ in $\Cc$.
\end{dfn}

\begin{rmk}
    The construction of $\PO \Cc$ is only pseudofunctorial in $\Cc$, since symmetric monoidal functors need not preserve $0$ ``on the nose.''
    This will not present serious issues for us, since in our main examples of interest, the path 2-monoids will be discrete (and any pseudofunctor valued in the 1-category $\CMon$ is just a 1-functor).
\end{rmk}

For any symmetric monoidal category $(\Cc, \oplus)$, the path 2-monoid $\PO \Cc$ has:
\begin{itemize}
    \item $\ob \PO \Cc = \sqcup_{c \in \Cc} \Map_\Cc(0, c)$,
    \item morphisms $f_1 \to f_2$ given by commutative triangles
    \[
        \begin{tikzcd}
            & 0 \ar[dl, "f_1", swap] \ar[dr, "f_2"] & \\
            c_1 \ar[rr, "\sim"] & & c_2
        \end{tikzcd}
    \]
    \item $\oplus$ given by $\oplus$ in $\Cc$ (using the natural isomorphism $0 \oplus 0 \simeq 0$ to interpret morphisms $0 \oplus 0 \to c_1 \oplus c_2$ as morphisms $0 \to c_1 \oplus c_2$), and
    \item structure morphisms induced by those of $\Cc$.
\end{itemize}

There is a natural homomorphism of commutative $2$-monoids $\cod: \PO \Cc \to \iota_0 \Cc$ sending any $f \in \PO \Cc$ to its codomain.

\begin{ex}
    Suppose $\Cc$ is the category of line bundles on some algebraic stack $\Xfr$. 
    Then $\PO \Cc$ is the commutative $2$-monoid of \emph{generalized effective Cartier divisors on $\Xfr$}, i.e.\ pairs $(\Lsc, s)$ where $\Lsc$ is a line bundle on $\Xfr$ and $s: \Osc_\Xfr \to \Lsc$ is a section.
    See e.g.\ \cite[Example 2.5(a)]{borne2012parabolic} for a more detailed discussion.
    The map $\cod: \PO \Cc \to \iota_0 \Cc$ sends a generalized effective Cartier divisor to its underlying line bundle.
\end{ex}

\begin{ex}
    Suppose $\Mc$ is a commutative 2-monoid.
    Then $\PO \Mc \simeq 0$.
    Indeed, for any $f: 0 \to m$ in $\Mc$, there is a unique commutative diagram
    \[
    \begin{tikzcd}
        & 0 \ar[dl, "f"'] \ar[dr, "\id_0"] & \\
        m \ar[rr, "f\inv"] & & 0
    \end{tikzcd}
    \]
    giving an isomorphism $f \simto \id_0$ in $\PO \Mc$.
\end{ex}

In some nice cases (including those we will consider throughout most of this paper), the path 2-monoid of a symmetric monoidal category turns out to be a discrete commutative monoid.
However, this is not true in general:

\begin{ex} \label{ex:path_not_discrete}
    Let $A$ be an abelian group, and let $\Cc_A$ be the symmetric monoidal category with:
    \begin{itemize}
        \item $\ob \Cc_A = \ZZ$
        \item For $m, n \in \ZZ$,
        \[
        \Map_{\Cc_A}(m, n) = \begin{cases}
            \varnothing & m > n \\
            A & m = n \\
            \{ * \} & m < n,
        \end{cases}
        \]
        with composition defined as addition in $A$ when this makes sense and trivially otherwise,
        \item $\oplus$ given on objects by addition in $\ZZ$ and on morphisms by addition in $A$ when it makes sense and trivially otherwise, and
        \item Symmetric monoidal structure maps all trivial.
    \end{itemize}
    The commutative $2$-monoid $\PO \Cc_A$ is not discrete, since $\Aut_{\PO \Cc_A}(*: 0 \to n) \cong A$ for any $n > 0$.
    In fact, there is an equivalence of categories $\PO \Cc_A = 0 \sqcup (\sqcup_{n > 0} \Bsf A)$.
\end{ex}

We note that not every commutative 2-monoid can arise as the path 2-monoid of a symmetric monoidal category.
Indeed, every such commutative 2-monoid must satisfy the following conditions (which may be considered as a categorified notion of ``sharpness,'' though we will not investigate the analogy in detail).

\begin{prop} \label{prop:path_monoid_sharp}
    Let $\Cc$ be a small symmetric monoidal category.
    Then:
    \begin{enumerate}
        \item $\Aut_{\PO \Cc}(0)$ is trivial, and
        \item $\pi_0(\PO \Cc)$ is sharp.
    \end{enumerate}
\end{prop}

\begin{proof}
    (1). The group $\Aut_{\PO \Cc}(0)$ consists of commutative triangles
    \[
        \begin{tikzcd}
            & 0 \ar[dl, "\id_{0}", swap] \ar[dr, "\id_{0}"] & \\
            0 \ar[rr, "a"] & & 0
        \end{tikzcd}
    \]
    in $\Cc$, where $a$ is an isomorphism.
    But this forces $a = \id_0$, so $\Aut_{\PO \Cc}(0) = \{ \id_0 \}$ is trivial.

    (2). Note that $\id_0: 0 \to 0$ is initial in $\Cc_{0/}$, so for any $f: 0 \to c$, we have
    \[
        \Map_{\Cc_{0/}}(\id_0, f) = \{*\}.
    \]
    Now suppose $f$ represents an invertible class in $\pi_0(\PO \Cc)$.
    Then $f$ is $\oplus$-invertible in $\PO \Cc$, hence also in $\Cc_{0/}$.
    It follows that
    \[
        \Map_{\Cc_{0/}}(f, \id_0) = \Map_{\Cc_{0/}}(\id_0, f^{\oplus-1}) = \{*\},
    \]
    so that $f \simeq \id_0$ in $\Cc_{0/}$.
    But this means that $[f] = [\id_0] = 0 \in \pi_0(\PO \Cc)$, i.e.\ $\pi_0(\PO \Cc)$ is sharp.
\end{proof}

\subsection{Relative path 2-monoids} \label{sub:relative_path}

The following generalization of path $2$-monoids will also be useful.

\begin{dfn} \label{dfn:relative_path_monoid}
    Let $\Cc$ be a small symmetric monoidal category, let $\Mc$ be a commutative 2-monoid, and let $F: \Mc \to \iota_0 \Cc$ be a homomorphism of commutative 2-monoids.
    Write $0 \downarrow_\Cc \Mc$ for the corresponding comma category, i.e.\ the category of morphisms $0 \to F(m)$ with $m \in \Mc$.
    The \emph{path 2-monoid of $\Cc$ relative to $F$} is $\Pop_{0,\Mc} \Cc := \iota_0(0 \downarrow_\Cc \Mc)$, considered with its natural symmetric monoidal structure.
\end{dfn}

More explicitly, $\Pop_{0,\Mc} \Cc$ has:
\begin{itemize}
    \item $\ob \Pop_{0,\Mc} \Cc = \big\{ \big(m \in \ob \Mc, f \in \Hom_\Cc(0, F(m))\big)\big\}$,
    \item For $(m_1, f_1), (m_2, f_2) \in \ob \Pop_{0,\Mc}$, $\Map_{\Pop_{0,\Mc} \Cc}((m_1, f_1), (m_2, f_2))$ consists of all $g: m_1 \simto m_2$ in $\Mc$ forming commutative triangles
    \[
        \begin{tikzcd}
            & 0 \ar[dl, "f_1", swap] \ar[dr, "f_2"] & \\
            F(c_1) \ar[rr, "F(g)"] & & F(c_2),
        \end{tikzcd}
    \]
    and
    \item $\oplus$ and the symmetric monoidal structure morphisms are induced by those of $\Cc$ and $\Mc$.
\end{itemize}
There is a natural homomorphism of commutative 2-monoids $\Pop_{0,\Mc} \Cc \to \Mc$ given by $(m, f) \mapsto m$.

\begin{ex}
    For any small symmetric monoidal category $\Cc$, we may write $\PO \Cc = \Pop_{0,\iota_0 \Cc} \Cc$ (where the functor $F: \iota_0 \Cc \to \iota_0 \Cc$ is the identity).
\end{ex}

\begin{ex}
    At the other extreme, $\Pop_{0,0} \Cc$ is the discrete commutative monoid $\Map_\Cc(0,0)$.
    If $\Aut_\Cc(0)$ is nontrivial, then $\Pop_{0,0} \Cc$ is not sharp, so \zcref{prop:path_monoid_sharp} does not extend to relative path monoids.
\end{ex}

If $\Cc$ is a small symmetric monoidal category and 
\[
\begin{tikzcd}
    {\Mc_1} \ar[dr, "{F_1}"'] \ar[rr, "G"] & & {\Mc_2} \ar[dl, "{F_2}"] \\
    & \iota_0 \Cc &
\end{tikzcd}
\]
is a commutative triangle in $\tCMon$ then there is an induced homomorphism of commutative 2-monoids $\Pop_{0,\Mc_1} \Cc \to \Pop_{0,\Mc_2} \Cc$ given by $(m, f) \mapsto (G(m), f)$.
In particular, for any $\Mc \to \iota_0 \Cc$, there are natural homomorphisms $\Map_\Cc(0, 0) = \Pop_{0,0} \Cc \to \Pop_{0,\Mc} \Cc$ and $\Pop_{0,\Mc} \Cc \to \PO \Cc$.

\section{Toric categories} \label{sec:toric_cats}

In this section we introduce the notion of a \emph{toric category}.
This is a categorification of the notion of toric monoid: just like how toric monoids control affine toric varieties, toric categories control ``basic'' toric stacks.
Toric categories satisfy a simple structure theorem (\zcref{thm:TCat_structure_thm}) which allows us to understand a general toric category in terms of abelian 2-groups and commutative monoids.

\subsection{Toric categories and the skew quotient construction}

\begin{dfn} \label{dfn:T_cat}
    A \emph{toric category} is a grouplike symmetric monoidal category $(\Cc, \oplus)$ such that:
    \begin{itemize}
        \item $\iota_0 \Cc$ is a strict abelian 2-group,
        \item $\pi_0(\iota_0 \Cc)$ is a finitely generated abelian group,
        \item $\Aut_\Cc(0)$ is a finitely generated free abelian group, and
        \item $\PO \Cc$ is a toric monoid (in particular, $\PO \Cc$ is essentially discrete).
    \end{itemize}
    Let $\TCat$ denote the full sub-2-category of $\iota_1 \SMCat$ spanned by the toric categories.
\end{dfn}

We single out the following special case, primarily to avoid having to repeat hypotheses:

\begin{dfn}
    A \emph{toric groupoid} is a toric category which is also a groupoid.
    Let $\TGpd$ denote the full sub-2-category of $\TCat$ spanned by the toric groupoids.
\end{dfn}

In other words, a toric groupoid is a strictly abelian 2-group $\Ac$ with $\pi_0(\Ac)$ finitely generated and $\Aut_\Ac(0)$ finitely generated free.
The following result is obvious:

\begin{prop} \label{prop:2Term_vs_TGpd}
    The functor $\Omega^\infty$ of \zcref{prop:deligne_classification} restricts to an equivalence $\tTerm \simto \TGpd$.
\end{prop}

Our principal method for constructing toric categories is the following:

\begin{cons} \label{cons:skew_quotient}
    Suppose:
    \begin{itemize}
        \item $M$ is a toric monoid,
        \item $\Ac$ is a toric groupoid, and
        \item $\phi: M \to \Ac$ is a homomorphism of commutative 2-monoids.
    \end{itemize}
    Then we may define a \emph{skew quotient} symmetric monoidal category $\Ac \vec{/}_\phi M$ as follows:\footnote{When $\phi$ is obvious from context, we will write $\Ac \vec{/} M$ for the skew quotient.}
    \begin{itemize}
        \item $\ob \Ac \vec{/} M = \ob \Ac$.
        \item For $a_1, a_2 \in \ob \Ac \vec{/} M$, we set 
        \[
            \Map_{\Ac \vec{/} M}(a_1, a_2) = \big\{ \big(m \in M, f \in \Map_\Ac(\phi(m) \oplus a_1, a_2)\big) \big\}.
        \]
        \item For $a_1, a_2, a_3 \in \ob \Ac \vec{/} M$, the composition map
        \[
            \Map_{\Ac \vec{/} M}(a_1, a_2) \times \Map_{\Ac \vec{/} M}(a_2, a_3) \to \Map_{\Ac \vec{/} M}(a_1, a_3)
        \]
        is given by
        \[
            (m_2, f_2) \circ (m_1, f_1) = \big(m_2 + m_1, f_2 \circ (\id_{\phi(m_2)} \oplus f_1) \in \Map_\Ac(\phi(m_2 + m_1) \oplus a_1, a_3)\big)
        \]
        In particular, the identity morphism $\id_{a_1}$ is $(0, \id_a)$.
        \item There is a natural functor $\Ac \to \Ac \vec{/} M$ given by the identity on objects and given by $f \mapsto (0, f)$ on morphisms.
        The symmetric monoidal structure morphisms of $\Ac \vec{/} M$ are the images of those of $\Ac$ under this functor.
    \end{itemize}
\end{cons}

\begin{rmk}
    When defining composition in $\Ac \vec{/} M$, we have generally suppressed mention of the monoidal structure morphisms of $\Ac$ and of $\phi$.
    In particular, the morphism in $\Map_\Ac(\phi(m_2 + m_1) \oplus a_1, a_3)$ arising as the second component in the composite $(m_2, f_2) \circ (m_1, f_1)$ is not literally $f_2 \circ (\id_{\phi(m_2)} \oplus f_1)$.
    Indeed, this formula does not even parse without inserting structure morphisms.
    Instead, the morphism in question is the composite
    \begin{align*}
        \phi(m_2 + m_1) \oplus a_1 &\simto \big(\phi(m_2) \oplus \phi(m_1)\big) \oplus a_1 \\
        &\simto \phi(m_2) \oplus (\phi(m_1) \oplus a_1) \\
        &\xrightarrow{\id_{\phi(m_2)} \oplus f_1} \phi(m_2) \oplus a_2 \\
        &\xrightarrow{f_2} a_3. 
    \end{align*}
    Likewise, each identity morphism $\id_a$ in $\Map_{\Ac \vec{/} M}$ is literally defined as
    \begin{align*}
        \big(0 \in M, \phi(0) \oplus a \simto 0 \oplus a \simto a\big).
    \end{align*}
    The usual coherence diagrams imply that all of the natural commutative diagrams we'd wish to construct in $\Ac \vec{/} M$ (e.g.\ those witnessing associativity and unitality of composition) do actually commute.
    Thus our choice to suppress the monoidal structure morphisms in our notation does not lead to any serious harm.
\end{rmk}

\begin{rmk}
    Our skew quotient construction is inspired by two notable prior works:
    \begin{itemize}
        \item The notion of ``$\Delta$-collection'' in \cite{cox1995functor} is essentially a symmetric monoidal functor out of a skew quotient.
        \item An $\infty$-categorical version of the skew quotient construction is studied in \cite[\S 8]{lehner2024group}.
        Our treatment of skew quotients is logically independent from \cite{lehner2024group}, and we will not use any $\infty$-categorical machinery.
    \end{itemize}
\end{rmk}

Here are a few key examples of the skew quotient construction.

\begin{ex}
    Let $M$ be a toric monoid.
    Then $0 \vec{/} M = \Bsf M$, the one-object symmetric monoidal category constructed in \zcref{rmk:BM}.
\end{ex}

\begin{ex}
    Let $\Ac$ be a toric groupoid.
    Then $\Ac \vec{/} 0 = \Ac$.
\end{ex}

\begin{ex} \label{ex:toric_poset}
    Let $A$ be a finitely generated abelian group and let $M \subset A$ be a toric submonoid.
    Then $A \vec{/} M$ is the category associated with the poset $(A, \leq)$, where $a_1 \leq a_2$ if and only if $a_2 - a_1 \in M$.
\end{ex}

It is fairly easy to describe the underlying abelian 2-group and the path 2-monoid of a skew quotient:

\begin{prop} \label{prop:skew_quotient_invariants}
    Let $\phi: M \to \Ac$ be as in \zcref{cons:skew_quotient}.
    Then:
    \begin{enumerate}
        \item $\iota_0(\Ac \vec{/} M) = \cofib(M^\pm \to \Ac)$.
        \item $\PO\,(\Ac \vec{/} M) = M^\sharp$.
    \end{enumerate}
\end{prop}

\begin{proof}
    (1). By the definition of composition in $\Ac \vec{/} M$, the isomorphisms in $\Ac \vec{/} M$ are precisely those $(m, f)$ with $m \in M^\pm$.
    Up to the action of automorphism groups in $\Ac$, every such isomorphism can be written as $(m, \id_{\phi(m) \oplus a}): a \simto \phi(m) \oplus a$ for some $m \in M^\pm$ and $a \in \ob \Ac$.
    Thus $\iota_0 \Ac \vec{/} M$ is the abelian 2-group constructed from $\Ac$ by adjoining isomorphisms $a \simto \phi(m) \oplus a$ for all $a \in \ob \Ac$ and $m \in M^\pm$, which is precisely the definition of $\cofib(M^\pm \to \Ac)$.

    (2). We may view $\PO\, (\Ac \vec{/} M)$ as the groupoid of pairs $(m \in M, f: \phi(m) \simto a)$ (as $a \in \ob \Ac$ varies).
    The automorphisms of such an $(m, f)$ consist of pairs $(m', f')$ fitting into commutative triangles
    \[
    \begin{tikzcd}
        & 0 \ar[dl, "{(m, f)}"'] \ar[dr, "{(m, f)}"] & \\
        a \ar[rr, "{(m', f')}"] & & a.
    \end{tikzcd}
    \]
    In particular, $m + m' = m$ in $\PO \Cc$.
    Since $M$ is assumed integral, we must have $m' = 0$ in $\PO \Cc$, so without loss of generality we may assume $f' \in \Map_\Ac(a, a)$.
    Commutativity of the above diagram is then equivalent to the equation $f' \circ f = f$.
    Because $f$ is an isomorphism, we must have $f' = \id_a$.
    Thus $(m', f')$ is the identity morphism, i.e.\ $\Aut_{\PO\, (\Ac \vec{/} M)}((m, f))$ is trivial.
    Since $(m, f)$ is arbitrary, we see that $\PO\, (\Ac \vec{/} M)$ is essentially discrete.
    
    There is a natural homomorphism $\psi: M \to \PO\, (\Ac \vec{/} M)$ given by
    \[
        \psi(m) = \big(m, \id_{\phi(m)} \in \Map_{\Ac \vec{/} M}(\phi(m), \phi(m))\big).
    \]
    The homomorphism $\psi$ is (essentially) surjective: any pair $(m, f: \phi(m) \simto a)$ is isomorphic to a pair $(m, \id_{\phi(m)})$ via a commutative triangle
    \[
    \begin{tikzcd}
        & 0 \ar[dl, "{(m, \id_{\phi(m)})}"'] \ar[dr, "{(m, f)}"] & \\
        \phi(m) \ar[rr, "{(0, f)}"] & & a.
    \end{tikzcd}
    \]
    Furthermore, $\psi(m_1) \simeq \psi(m_2)$ if and only if there is a commutative triangle
    \[
    \begin{tikzcd}
        & 0 \ar[dl, "{(m_1, \id_{\phi(m_1)})}"'] \ar[dr, "{(m_2, \id_{\phi(m_2)})}"] & \\
        \phi(m_1) \ar[rr, "{(m_3, f)}"] & & \phi(m_2)
    \end{tikzcd}
    \]
    for some isomorphism $(m_3, f)$.
    Such an $f$ is unique given the existence of such an $m_3$, so we see that $\psi(m_1) \simeq \psi(m_3)$ if and only if there exists $m_3 \in M^\pm$ such that $m_3 + m_1 = m_2$.
    In other words, we have $\im \psi = M / M^\pm$.
    Because $\psi$ is surjective, we must have $\PO\, (\Ac \vec{/} M) = M / M^\pm = M^\sharp$.
\end{proof}

\begin{cor} \label{cor:skew_quotient_is_T_cat}
    With the notation of \zcref{cons:skew_quotient}, $\Ac \vec{/} M$ is a toric category.
\end{cor}

\begin{proof}
    We need to check the conditions of \zcref{dfn:T_cat}.
    The abelian 2-group $\iota_0(\Ac \vec{/} M) = \cofib(M^\pm \to \Ac)$ is the cofiber of a homomorphism of strictly abelian 2-groups, hence is itself strictly abelian.
    We may write $\pi_0(\iota_0(\Ac \vec{/} M)) = \pi_0(\Ac) / \pi_0(M^\pm)$, and this group is finitely generated because $\pi_0(\Ac)$ is.
    By our construction, the group $\Aut_{\Ac \vec{/} M}(0)$ is an extension of the finitely generated free abelian group $\Aut_\Ac(0)$ by a subgroup of the finitely generated free abelian group $M^\pm$, hence is itself finitely generated free.
    Finally, the path 2-monoid $\PO\, (\Ac \vec{/} M) = M^\sharp$ is toric by \zcref{prop:quotient_by_face_is_toric}.
\end{proof}

\subsection{A universal property of skew quotients}

In \zcref{cons:skew_quotient} we have given an explicit construction of skew quotients $\Ac \vec{/} M$.
Categorically inclined readers may be wondering about the universal property of this construction and why we have chosen the name ``skew quotient.''
The following useful result should illuminate these matters.

\begin{prop} \label{prop:skew_quotient_UP}
    Fix notation as in \zcref{cons:skew_quotient}.
    For all small symmetric monoidal categories $\Cc$, there is a natural equivalence
    \[
        \iota_0 \Fun^\oplus(\Ac \vec{/} M, \Cc) = \Hom_{\tCMon}(\Ac, \iota_0 \Cc) \times_{\Hom_{\tCMon}(M, \iota_0 \Cc)} \Hom_{\tCMon}(M, \PO \Cc).
    \]
\end{prop}

\begin{proof}
    Let $\Cc$ be a toric category.
    Restriction along the natural map $\Ac \to \iota_0 (\Ac \vec{/} \Mc)$ gives a map
    \[
        \iota_0 \Fun^\oplus(\Ac \vec{/} M, \Cc) \to \iota_0 \Fun^\oplus(\iota_0(\Ac \vec{/} M), \Cc) \to \Hom_{\tCMon}(\Ac, \iota_0 \Cc).
    \]
    Likewise, restriction along the natural map $M \to \PO\, (\Ac \vec{/} M)$ (constructed in \zcref{prop:skew_quotient_invariants}) gives a map 
    \[
        \iota_0 \Fun^\oplus(\Ac \vec{/} M, \Cc) \to \Hom_{\tCMon}(\PO\, (\Ac \vec{/} M), \PO \Cc) \to \Hom_{\tCMon}(M, \PO \Cc).
    \]
    The induced maps $\iota_0 \Fun^\oplus(\Ac \vec{/} M, \Cc) \to \Hom_{\tCMon}(M, \iota_0 \Cc)$ are naturally isomorphic since they both essentially send $F \in \iota_0 \Fun^\oplus(\Ac \vec{/} M, \Cc)$ to the homomorphism $m \mapsto F(\phi(m))$).
    Combining all of these maps gives a natural map
    \[
        \iota_0 \Fun^\oplus(\Ac \vec{/} M, \Cc) \to \Hom_{\tCMon}(\Ac, \iota_0 \Cc) \times_{\Hom_{\tCMon}(M, \iota_0 \Cc)} \Hom_{\tCMon}(M, \PO \Cc).
    \]

    To define an inverse to this functor, suppose that 
    \[
    (F_\Ac: \Ac \to \iota_0 \Cc, F_M: M \to \PO \Cc, \eta: F_\Ac \circ \phi \simto \cod \circ \PO \Cc)
    \]
    is a general object of $\Hom_{\tCMon}(\Ac, \iota_0 \Cc) \times_{\Hom_{\tCMon}(M, \iota_0 \Cc)} \Hom_{\tCMon}(M, \PO \Cc)$.
    We can define a corresponding symmetric monoidal functor $F: \Ac \vec{/} M \to \Cc$ as follows:
    \begin{itemize}
        \item On objects, we set $F(a) = F_\Ac(a)$.
        \item On morphisms, for $(m, f) \in \Map_{\Ac \vec{/} M}(a_1, a_2)$, we define $F(m, f)$ as the composite
        \begin{align*}
            F_\Ac(a_1) &\simto 0 \oplus F_\Ac(a_1) \\
            &\xrightarrow{F_M(m) \oplus \id_{F_\Ac(a_1)}} \cod F_M(m) \oplus F_\Ac(a_1) \\
            &\xrightarrow{\eta \oplus \id_{F_\Ac(a_1)}} F_\Ac(\phi(m)) \oplus F_\Ac(a_1) \\
            &\simto F_\Ac(\phi(m) \oplus a_1) \\
            &\xrightarrow{F_\Ac(f)} F_\Ac(a_2).
        \end{align*}
        \item The symmetric monoidal structure morphisms of $F$ are those of $F_\Ac$.
    \end{itemize}
    The dutiful reader can verify that this construction defines a functor 
    \[
        \Hom_{\tCMon}(\Ac, \iota_0 \Cc) \times_{\Hom_{\tCMon}(M, \iota_0 \Cc)} \Hom_{\tCMon}(M, \PO \Cc) \to \iota_0 \Fun^\oplus(\Ac \vec{/} M, \Cc)
    \]
    which is the inverse of the functor of the previous paragraph.
\end{proof}

\begin{ex} \label{ex:i0_corepresented}
    If we think of $\ZZ$ as the skew quotient $\ZZ \vec{/} 0$, then for all toric categories $\Cc$, we have
    \begin{align*}
        \iota_0 \Fun^\oplus(\ZZ \vec{/} 0, \Cc) &= \Hom_{\tAb}(\ZZ, \iota_0 \Cc) \times_{\Hom_{\tCMon}(0, \iota_0 \Cc)} \Hom_{\CMon}(0, \PO \Cc) \\
        &= \iota_0 \Cc \times_0 0 \\
        &= \iota_0 \Cc.
    \end{align*}
    Thus $\ZZ$ corepresents the 2-functor $\iota_0: \TCat \to \TGpd$.
\end{ex}

\begin{ex} \label{ex:P0_corepresented}
    Let $i: \NN \to \ZZ$ be the inclusion of the positive natural numbers.
    Then for all toric categories $\Cc$, we have
    \begin{align*}
        \iota_0 \Fun^\oplus(\ZZ \vec{/}_i \NN, \Cc) &= \Hom_{\tAb}(\ZZ, \iota_0 \Cc) \times_{\Hom_{\tCMon}(\NN, \iota_0 \Cc)} \Hom_{\CMon}(\NN, \PO \Cc) \\
        &= \Hom_{\tAb}(\ZZ, \iota_0 \Cc) \times_{\Hom_{\tAb}(\ZZ, \iota_0 \Cc)} \Hom_{\CMon}(\NN, \PO \Cc) \\
        &= \iota_0 \Cc \times_{\iota_0 \Cc} \PO \Cc \\
        &= \PO \Cc.
    \end{align*}
    Thus $\ZZ \vec{/}_i \NN$ corepresents the 2-functor $\PO: \TCat \to \CMon^{\sharp,\toric}$.
\end{ex}

\subsection{The structure theory of toric categories}

We would now like to show that every toric category arises as a skew quotient in a canonical way.
More precisely:

\begin{dfn}
    Let $\Arr^{\sharp,\toric\to\TGpd}(\tCMon)$ be the full subcategory of the 2-category of 1-morphisms and commutative squares in $\tCMon$ spanned by the morphisms $M \to \Ac$ where:
    \begin{itemize}
        \item $M$ is a discrete sharp toric monoid, and
        \item $\Ac$ is a toric groupoid.
    \end{itemize}
\end{dfn}

\begin{thm} \label{thm:TCat_structure_thm}
    The 2-functor
    \begin{align*}
        \TCat &\to \Arr^{\sharp,\toric\to\TGpd}(\tCMon) \\
        \Cc &\mapsto (\PO \Cc \to \iota_0 \Cc)
    \end{align*}
    is an equivalence, with inverse given by $(M \to \Ac) \mapsto \Ac \vec{/} M$.
    In particular, for $\Cc \in \TCat$, there is a natural isomorphism $\Cc \simeq \iota_0 \Cc \vec{/} \PO \Cc$.
\end{thm}

The proof will appear after the proof of \zcref{prop:structure_thm_general}.

As part of \zcref{thm:TCat_structure_thm}, we see that, for every toric category $\Cc$, we can write $\Cc = \iota_0 \Cc \vec{/} \PO \Cc$.
This is the unique skew quotient presentation $\Cc = \Ac \vec{/} M$ with $M$ sharp.
Morally, one can think of this as something like a distinguished, functorial ``resolution'' of $\Cc$ by commutative 2-monoids.

\begin{ex} \label{ex:toric_poset_converse}
    Let $\Cc$ be a toric category which is also a poset.
    Then $A := \iota_0 \Cc$ is a discrete abelian group, and the morphisms of $\Cc$ define a partial order on $A$.
    The path monoid $\PO \Cc = A_{\geq 0}$ is the monoid of nonnegative elements of $A$.
    Thus \zcref{thm:TCat_structure_thm} implies $\Cc = A \vec{/} A_{\geq 0}$, giving a converse to \zcref{ex:toric_poset}.
\end{ex}

We will find it useful to prove \zcref{thm:TCat_structure_thm} in a somewhat greater level of generality: instead of just writing $\Cc = \iota_0 \Cc \vec{/} \PO \Cc$, we will write $\Cc = \Ac \vec{/} \Pop_{0,\Ac} \Cc$ whenever we are given a sufficiently nice homomorphism of toric groupoids $\Ac \to \iota_0 \Cc$ (see \zcref{prop:structure_thm_general}).
This gives us a more general source of ``resolutions'' of $\Cc$.
To work in this generality, we will need some preparatory lemmas:

\begin{lem} \label{lem:translation_lemma}
    Let $\Cc$ be a toric category, and let $f: c \to c'$ be any morphism in $\Cc$.
    Then there exists some $f' \in \PO \Cc$ such that $f = f' \oplus \id_{c'}$ (up to post-composition with an isomorphism). 
\end{lem}

\begin{proof}
    Because $\pi_0(\iota_0 \Cc)$ is a group, the translation functor $(-) \oplus c$ is an equivalence, hence is full.
\end{proof}

\begin{lem} \label{lem:morphisms_in_T_cat_mono}
    Let $\Cc$ be a toric category.
    Then every morphism in $\Cc$ is both a monomorphism and an epimorphism.
\end{lem}

\begin{proof}
    Let $f: c_1 \to c_2$ be any morphism in $\Cc$.
    We want to show that $f$ is a monomorphism showing that $f$ is an epimorphism is analogous.
    By \zcref{lem:translation_lemma} may assume without loss of generality that $f \in \PO \Cc$.
    
    Suppose $g_1, g_2: c \to c'$ are two morphisms such that $g_1 \circ f = g_2 \circ f$.
    Let $g_i' = g_i \oplus \id_{-c} \in \Map_\Cc(0, c' \oplus -c)$ for $i = 1, 2$.
    Then the equation $g_1 \circ f = g_2 \circ f$ (as morphisms in $\Cc$) translates into the isomorphism $g_1' \oplus f \cong g_2' \oplus f$ (as objects in $\PO \Cc$).
    Because $\PO \Cc$ is a toric monoid (and in particular discrete and integral), we must have $g_1' \cong g_2'$ uniquely as objects of $\PO \Cc$, and thus $g_1 = g_2$.
    Because $g_1$ and $g_2$ were arbitrary, $f$ must be a monomorphism.

    The proof that $f$ is an epimorphism is entirely analogous.
\end{proof}

\begin{lem} \label{lem:relative_path_monoid_toric}
    Let $\Cc$ be a toric category.
    Let $\Ac$ be a toric groupoid, and let $F: \Ac \to \iota_0 \Cc$ be a homomorphism of abelian 2-groups such that:
    \begin{itemize}
        \item The induced map $\Aut_\Ac(0) \to \Aut_\Cc(0)$ is injective, and
        \item The quotient $\Aut_\Cc(0) / \Aut_\Ac(0)$ is torsion-free.
    \end{itemize}
    Then $\Pop_{0,\Ac} \Cc$ is a discrete toric monoid.
\end{lem}

\begin{proof}
    The functor $F$ is injective on $\Aut_\Ac(0)$ and every $\Aut_\Ac(a)$ is an $\Aut_\Ac(a)$-torsor, so $F$ is faithful.
    
    By the discussion following \zcref{dfn:relative_path_monoid}, we may write $\Pop_{0,\Ac} \Cc$ as the groupoid of pairs $(a \in \Ac, f \in \Map_\Cc(0, a))$.
    The automorphisms of any such $(a, f)$ are given by $g \in \Aut_\Ac(a)$ fitting into commutative triangles
    \[
        \begin{tikzcd}
            & 0 \ar[dl, "f", swap] \ar[dr, "f"] & \\
            F(a) \ar[rr, "F(g)"] & & F(a),
        \end{tikzcd}
    \]
    By \zcref{lem:morphisms_in_T_cat_mono}, the equation $F(g) \circ f = f = \id_{F(a)} \circ f$ implies that $F(g) = \id_{F(a)}$.
    Because $F$ is faithful, $g = \id_a$, so $\Aut_{\Pop_{0,\Ac} \Cc}(a, f)$ is trivial.
    The object $(a, f)$ was arbitrary, so $\Pop_{0,\Ac} \Cc$ is discrete.

    To see that $\Pop_{0,\Ac} \Cc$ is toric, first note that the units in $\Pop_{0,\Ac} \Cc$ may be written (up to isomorphism) as pairs $(0, f \in \Aut_\Cc(0))$.
    Two such pairs are isomorphic if and only if they differ by an element of $\Aut_\Ac(0)$, so $(\Pop_{0,\Ac} \Cc)^\pm = \Aut_\Cc(0) / \Aut_\Ac(0)$.
    
    Now recall that there is a natural homomorphism $\psi: \Pop_{0,\Ac} \Cc \to \PO \Cc$ given by $(a, f) \mapsto f$.
    We claim that $(\Pop_{0,\Ac} \Cc)^\sharp = \im \psi$, or equivalently $\psi$ identifies two objects of $\Pop_{0,\Ac} \Cc$ if and only if the two objects differ by a unit in $\Pop_{0,\Ac} \Cc$
    To see this claim, observe that $\psi(a_1, f_1) = \psi(a_2, f_2)$ if and only if there exists an isomorphism $h: F(a_1) \simto F(a_2)$ in $\Cc$ fitting into a commutative triangle
    \[
        \begin{tikzcd}
            & 0 \ar[dl, "f_1", swap] \ar[dr, "f_2"] & \\
            F(a_1) \ar[rr, "h"] & & F(a_2),
        \end{tikzcd}
    \]
    Given such an $h$, choose an additive inverse $-a_1$ for $a_1 \in \Ac$, and let $h' = h \oplus \id_{F(-a_1)} \in \Map_\Cc(0, F(a_2 \oplus -a_1))$.
    Then $(a_2, f_2) \cong (a_1, f_1) \oplus (a_2 \oplus -a_1, h')$ in $\Pop_{0,\Ac} \Cc$, and $(a_2 \oplus -a_1, h')$ is a unit in $\Pop_{0,\Ac} \Cc$.
    Conversely, by our above description of units in $\Pop_{0,\Ac} \Cc$ as pairs $(0, f)$, we see that $\psi$ identifies any two objects of $\Pop_{0,\Ac} \Cc$ which differ by a unit.
    It follows that $\im \psi = (\Pop_{0,\Ac} \Cc) / (\Pop_{0,\Ac} \Cc)^\pm = (\Pop_{0,\Ac} \Cc)^\sharp$ as claimed.

    We may now verify that the conditions of \zcref{prop:check_toric_via_units_and_sharp} hold for $\Pop_{0,\Ac} \Cc$:
    \begin{itemize}
        \item The group $(\Pop_{0,\Ac} \Cc)^\pm = \Aut_\Cc(0) / \Aut_\Ac(0)$ is finitely generated and free by hypothesis. 
        \item To see that $(\Pop_{0,\Ac} \Cc)^\pm$ acts freely on $\Pop_{0,\Ac} \Cc$, note by \zcref{lem:morphisms_in_T_cat_mono} that $\Aut_\Cc(0)$ acts freely on each hom-set $\Map_\Cc(0, a)$.
        If we quotient these hom-sets by the action of $\Aut_\Ac(0)$, we see that $(\Pop_{0,\Ac} \Cc)^\pm$ acts freely on each such quotient.
        Since $\Pop_{0,\Ac} \Cc$ can be viewed as a union of such quotients, $(\Pop_{0,\Ac} \Cc)^\pm$ acts freely on $\Pop_{0,\Ac} \Cc$.
        \item Note that $\im \psi$ is the submonoid of $\PO \Cc$ consisting of maps $0 \to F(a)$ with $a \in \Ac$.
        It follows that $\im \psi = \im \pi_0 F \times_{\pi_0(\iota_0 \Cc)} \PO \Cc$, so $\im \psi$ (and thus $(\Pop_{0,\Ac} \Cc)^\sharp$) is toric by \zcref{prop:toric_monoid_finite_limit}.
    \end{itemize}
    Thus $\Pop_{0,\Ac} \Cc$ is toric.
\end{proof}

We are now ready to construct the desired skew quotient presentations:

\begin{prop} \label{prop:structure_thm_general}
    Let $\Cc$ be a toric category, let $\Ac$ be a toric groupoid, and let $F: \Ac \to \iota_0 \Cc$ be a homomorphism such that:
    \begin{itemize}
        \item The induced map $\Aut_\Ac(0) \to \Aut_\Cc(0)$ is injective, 
        \item The quotient $\Aut_\Cc(0) / \Aut_\Ac(0)$ is torsion-free, and
        \item $F$ is essentially surjective.
    \end{itemize}
    Then $F$ induces a symmetric monoidal equivalence $F': \Ac \vec{/} \Pop_{0,\Ac} \Cc \simto \Cc$.
\end{prop}

\begin{proof}
    The monoid $\Pop_{0,\Ac} \Cc$ is toric by \zcref{lem:relative_path_monoid_toric}, so we are in the situation of \zcref{cons:skew_quotient} and can construct the skew quotient $\Ac \vec{/} \Pop_{0,\Ac} \Cc$.
    By \zcref{prop:skew_quotient_UP}, to construct a functor $F: \Ac \vec{/} \Pop_{0,\Ac} \Cc \to \Cc$, it suffices to construct an object of the groupoid
    \[
        \Hom_{\tAb}(\Ac, \iota_0 \Cc) \times_{\Hom_{\tCMon}(\Pop_{0,\Ac} \Cc, \iota_0 \Cc)} \Hom_{\CMon}(\Pop_{0,\Ac} \Cc, \PO \Cc),
    \]
    i.e.\ a triple $(G_1: \Ac \to \iota_0 \Cc, G_2: \Pop_{0,\Ac} \Cc \to \PO \Cc, \eta: G_1 \circ \cod \simto \cod \circ G_2)$.
    The triple in question is specified by:
    \begin{itemize}
        \item The homomorphism $F: \Ac \to \iota_0 \Cc$,
        \item The forgetful map $\Pop_{0,\Ac} \Cc \to \PO \Cc$, and
        \item The identity map from $F \circ \cod: \Pop_{0,\Ac} \Cc \to \iota_0 \Cc$ to itself.
    \end{itemize}
    The induced symmetric monoidal functor $F': \Ac \vec{/} \Pop_{0,\Ac} \Cc \to \Cc$ is essentially surjective, since $F'$ agrees with $F$ on objects.
    
    To see that $F'$ is fully faithful, let $a_1, a_2 \in \ob \Ac$ be arbitrary.
    We need to show that the natural map
    \[
        \Map_{\Ac \vec{/} \Pop_{0,\Ac} \Cc}(a_1, a_2) \to \Map_{\Cc}(F'(a_1), F'(a_2))
    \]
    is an isomorphism.
    Using \zcref{lem:translation_lemma}, we reduce to showing that the natural maps
    \[
        F'_{0,a}: \Map_{\Ac \vec{/} \Pop_{0,\Ac} \Cc}(0, a) \to \Map_{\Cc}(0, F'(a)) 
    \]
    are isomorphisms for all $a \in \ob \Ac$.
    
    By \zcref{cons:skew_quotient}, $\Map_{\Ac \vec{/} \Pop_{0,\Ac} \Cc}(0, a)$ consists of pairs $\big((a', f) \in \Pop_{0,\Ac} \Cc, g \in \Map_\Ac(a', a)\big)$ (where $(a', f)$ is considered up to the action of isomorphisms in $\Ac$).
    Equivalently, $\Map_{\Ac \vec{/} \Pop_{0,\Ac} \Cc}(0, a)$ consists of triples
    \[
        \big(a' \in \ob \Ac, f \in \Map_\Ac(0, F(a')), g \in \Map_\Ac(a', a)\big)
    \]
    again modulo suitable isomorphisms.
    The maps $F'_{0,a}$ send such a triple to $F(g) \circ f \in \Map_{\Cc}(0, F'(a))$.
    To see that $F'_{0,a}$ is injective, we note that if $F'_{0,a}(a_1', f_1, g_1) = F'_{0,a}(a_2', f_2, g_2)$, then $g_2\inv \circ g_1$ gives an isomorphism from $(a_1', f_1, g_1)$ to $(a_2', f_2, g_2)$, so the two triples must lie in the same equivalence class.
    To see that $F'_{0,a}$ is surjective, we note that any $h \in \Map_{\Cc}(0, F'(a))$ is the image of $(a, h, \id_a)$.
    Thus each $F'_{0,a}$ is an isomorphism, so $F'$ is fully faithful and therefore an equivalence.
\end{proof}

\begin{proof}[Proof of \zcref{thm:TCat_structure_thm}]
    We need to prove that:
    \begin{itemize}
        \item For all toric categories $\Cc$, there is a natural symmetric monoidal equivalence $\iota_0 \Cc \vec{/} \PO \Cc \simto \Cc$.
        This follows from \zcref{prop:structure_thm_general} applied to the identity homomorphism $\iota_0 \Cc \to \iota_0 \Cc$, since $\PO \Cc = \Pop_{0,\iota_0 \Cc} \Cc$.
        \item Whenever $(M \to \Ac)$ is an object of $\Arr^{\sharp,\toric\to\TGpd}(\tCMon)$, there is a natural equivalence $(M \to \Ac) \simto \big(\PO \, (\Ac \vec{/} M) \to \iota_0 (\Ac \vec{/} M)\big)$.
        
        Using the universal property of skew quotients (\zcref{prop:skew_quotient_UP}), we see that the identity functor $\id_{\Ac \vec{/} M}$ induces a commutative diagram
        \[
        \begin{tikzcd}
            M \rar \dar & \Ac \dar \\
            \PO\, (\Ac \vec{/} M) \rar & \iota_0 (\Ac \vec{/} M).
        \end{tikzcd}
        \]
        By \zcref{prop:skew_quotient_invariants} and sharpness of $M$, we have natural equivalences $\PO\, (\Ac \vec{/} M) = M$ and $\iota_0 (\Ac \vec{/} M) = \cofib(0 \to \Ac) = \Ac$, giving the claim. \qedhere
    \end{itemize}
\end{proof}

We can also use \zcref{prop:structure_thm_general} to present arbitrary toric categories as skew quotients of homomorphisms of \emph{discrete} commutative monoids.
This is not functorial (due to the non-uniqueness issues mentioned in \zcref{prop:strict_trivialization}), but it is still  quite useful.

\begin{cor} \label{cor:skew_quotient_of_discrete}
    Let $\Cc$ be a toric category.
    Then $\Cc$ is (non-canonically) equivalent to a skew quotient $A \vec{/} M$ where $A$ is a discrete abelian group and $M$ is a toric monoid.
\end{cor}

\begin{proof}
    Let $A = \pi_0(\iota_0 \Cc)$.
    By \zcref{prop:strict_trivialization}, we may choose an essentially surjective functor $F: A \to \iota_0 \Cc$.
    The hypotheses of \zcref{prop:structure_thm_general} hold for this $F$ because $\Aut_{A}(0) = 0$ and $\Aut_\Cc(0)$ is torsionfree.
    Thus we get an equivalence $\Cc \simeq A \vec{/} \Pop_{0,A} \Cc$.
    By \zcref{lem:relative_path_monoid_toric}, $M = \Pop_{0,A} \Cc$ is discrete and toric, so we win.
\end{proof}

\subsection{Localizations of toric categories}

We'd now like to develop a basic theory of localization for toric categories generalizing the theory of localization of toric monoids.
This will be useful for understanding open immersions of toric stacks later.

\begin{dfn}
    Let $(\Cc, \oplus)$ be a small symmetric monoidal category, and let $S$ be a collection of morphisms in $\Cc$.
    The \emph{symmetric monoidal localization} of $\Cc$ along $S$ is the universal symmetric monoidal category $\Cc[S\inv]$ equipped with a symmetric monoidal functor $L: \Cc \to \Cc[S\inv]$ sending all morphisms in $S$ to isomorphisms.
\end{dfn}

\begin{notn}
    Since we will not need any other notion of localization of categories, we will often just refer to $\Cc[S\inv]$ as the \emph{localization} of $\Cc$ along $S$.
\end{notn}

Here are some key results controlling localizations of toric categories.

\begin{prop} \label{prop:localization_properties}
    Let $\Cc$ be a toric category.
    Then:
    \begin{enumerate}
        \item Each localization of $\Cc$ can be written as the localization along a face of $\PO \Cc$.
        \item If $S \subset \PO \Cc$ is a face, then
        \[
            \Cc[S\inv] = \cofib(S^\gp \to \iota_0 \Cc) \vec{/} \big(\PO \Cc / S\big).
        \]
        \item If $\Cc = \Ac \vec{/} M$ as in \zcref{cons:skew_quotient} and $S \subset M$ is a face, then $\Cc[S\inv] = \Ac \vec{/} M[-S]$.
        \item Every localization of $\Cc$ is a toric category.
    \end{enumerate}
\end{prop}

\begin{proof}
    (1). Suppose $f: c \to c'$ is a morphism in $\Cc$.
    By \zcref{lem:translation_lemma}, we may assume without loss of generality that $f \in \PO \Cc$.

    Now we claim that localizing along an element $f \in \PO \Cc$ is the same as localizing along the whole face $\angles{f} \subset \PO \Cc$.
    Recall that 
    \[
        \angles{f} = \bset{f_1 \in \PO \Cc}{\textrm{ there exist $f_2 \in \PO \Cc$ and $n \geq 0$ such that } f_1 + f_2 = n \cdot f}
    \]
    where we use additive notation for $\PO \Cc$.
    Suppose $f_1 + f_2 = n \cdot f$ in $\PO \Cc$.
    Inverting $f$ allows us to write $f_1 + f_2 - n \cdot f = 0$, so that $f_2 - n \cdot f$ gives an inverse to $f_1$ in $\PO \Cc$.
    A suitable translate of this morphism $f_2 - n \cdot f$ gives an inverse to $f_1$ when viewed as a morphism in $\PO \Cc$.
    Thus inverting $f$ automatically also inverts any $f_1 \in \angles{f}$, giving the claim.

    So far we have seen that localizing along any individual morphism in $\Cc$ is the same as localizing along a face of $\Cc$.
    If $S$ is any set of morphisms of $\Cc$, then it follows that localizing along $S$ is equivalent to localizing along all faces of $\PO \Cc$ corresponding to morphisms in $S$.
    This is equivalent to localizing along the single face of $\PO \Cc$ generated (as a face) by translates of morphisms in $S$.

    (2). We use the universal property of skew quotients (\zcref{prop:skew_quotient_UP}).
    Indeed, for any small symmetric monoidal category $(\Cc', \oplus)$, let
    \begin{align*}
        \Fc_{\Cc'} &:= \iota_0 \Fun^\oplus\big(\cofib(S^\gp \to \iota_0 \Cc) \vec{/} (\PO \Cc / S), \Cc'\big) \\
        &= \Hom_{\tCMon}(\cofib(S^\gp \to \iota_0 \Cc), \iota_0 \Cc') \times_{\Hom_{\tCMon}(\PO \Cc / S, \iota_0 \Cc')} \Hom_{\tCMon}(\PO \Cc / S, \PO \Cc').
    \end{align*}
    Each term in the expression for $\Fc_{\Cc'}$ can be simplified:
    \begin{align*}
    \Hom_{\tCMon}(\cofib(S^\gp \to \iota_0 \Cc), \iota_0 \Cc') &= \fib(\Hom_{\tCMon}(\iota_0 \Cc, \iota_0 \Cc') \to \Hom_{\tCMon}(S^\gp, \iota_0 \Cc')), \\
    \Hom_{\tCMon}(\PO \Cc / S, \iota_0 \Cc') &= \fib(\Hom_{\tCMon}(\PO \Cc, \iota_0 \Cc') \to \Hom_{\tCMon}(S, \iota_0 \Cc')), \textrm{ and} \\
    \Hom_{\tCMon}(\PO \Cc / S, \PO \Cc') &= \fib(\Hom_{\tCMon}(\PO \Cc, \PO \Cc') \to \Hom_{\tCMon}(S, \PO \Cc')).
    \end{align*}
    Because fibers and fiber products commute, $\Fc_{\Cc'}$ can be computed as the fiber of the natural map from
    \[
    \iota_0 \Fun^\oplus\big(\Cc, \Cc'\big) = \iota_0 \Fun^\oplus\big(\iota_0 \Cc \vec{/} \PO \Cc, \Cc'\big) = \Hom_{\tCMon}(\iota_0 \Cc, \iota_0 \Cc') \times_{\Hom_{\tCMon}(\PO \Cc, \iota_0 \Cc')} \Hom_{\tCMon}(\PO \Cc, \PO \Cc')
    \]
    (identified using \zcref{thm:TCat_structure_thm} and \zcref{prop:skew_quotient_UP}) to 
    \[
    \Hom_{\tCMon}(S, \PO \Cc') = \Hom_{\tCMon}(S^\gp, \iota_0 \Cc') \times_{\Hom_{\tCMon}(S, \iota_0 \Cc')} \Hom_{\tCMon}(S, \PO \Cc').
    \]
    Because $\PO \Cc$ is sharp, the natural morphism $0 \to \Hom_{\tCMon}(S, \PO \Cc')$ is fully faithful by \zcref{prop:path_monoid_sharp}, and thus the natural morphism $\Fc_{\Cc'} \to \iota_0 \Fun^\oplus\big(\Cc, \Cc'\big)$ is also fully faithful.
    Thus we may view $\Fc_{\Cc'}$ is the full subcategory of $\iota_0 \Fun^\oplus\big(\Cc, \Cc'\big)$ consisting of symmetric monoidal functors $F$ such that $\PO F(S) = 0$.
    This last condition is equivalent to saying that $F$ sends all morphisms in $S$ to invertible morphisms in $\Cc'$, so 
    \[
        \iota_0 \Fun^\oplus\big(\cofib(S^\gp \to \iota_0 \Cc) \vec{/} (\PO \Cc / S), \Cc'\big) = \Fc_{\Cc'} = \iota_0 \Fun^\oplus\big(\Cc[S\inv], \Cc'\big)
    \]
    and we see that the needed universal property is satisfied.

    (3). We again use the universal property of skew quotients (\zcref{prop:skew_quotient_UP}) to write (for any $\Cc'$)
    \begin{align*}
        \iota_0 \Fun^\oplus(\Ac \vec{/} M[-S], \Cc') &= \Hom_{\tCMon}(\Ac, \iota_0 \Cc') \times_{\Hom_{\tCMon}(M[-S], \iota_0 \Cc')} \Hom_{\tCMon}(M[-S], \PO \Cc') \\
        &= \Hom_{\tCMon}(\Ac, \iota_0 \Cc') \times_{\Hom_{\tCMon}(M[-S], \iota_0 \Cc')} \Hom_{\tCMon}(M[-S]^\sharp, \PO \Cc') \\
        &= \Hom_{\tCMon}(\Ac, \iota_0 \Cc') \times_{\Hom_{\tCMon}(M[-S], \iota_0 \Cc')} \Hom_{\tCMon}(M / S, \PO \Cc') \textrm{ by \zcref{prop:localization_and_quotient_by_face}} \\
        &\subset \Hom_{\tCMon}(\Ac, \iota_0 \Cc') \times_{\Hom_{\tCMon}(M, \iota_0 \Cc')} \Hom_{\tCMon}(M, \PO \Cc') \\
        &= \iota_0 \Fun^\oplus(\Ac \vec{/} M, \Cc')
    \end{align*}
    where the inclusion identifies $\iota_0 \Fun^\oplus(\Ac \vec{/} M[-S], \Cc')$ with the collection of symmetric monoidal $F: \Ac \vec{/} M \to \Cc'$ which invert $S$, precisely as needed.

    (4). This follows immediately from (2) (or (3)) and \zcref{cons:skew_quotient}.
\end{proof}

The following special case of localization will be important for us:

\begin{dfn}
    Let $\Cc$ be a toric category.
    The \emph{groupoid completion of $\Cc$}, written $\Cc^\gpd$, is the localization of $\Cc$ along all morphisms in $\Cc$.
\end{dfn}

\begin{prop} \label{prop:Cgpd_calculation}
    Let $\Cc$ be a toric category, and write $\Cc = \Ac \vec{/} M$ for some $\phi: M \to \Ac$ as in \zcref{cons:skew_quotient}.
    Then 
    \[
        \Cc^\gpd = \cofib(M^\gp \to \Ac) = \Ac \vec{/} M^\gp.
    \]
\end{prop}

\begin{proof}
    This follows from (3) in \zcref{prop:localization_properties} applied to the face $M \subset M$.
\end{proof}

\begin{ex}
    In particular, we see that $(\ZZ \vec{/} \NN)^\gpd = \cofib(\id: \ZZ \to \ZZ) = 0$.
    Pictorially, this corresponds to inverting all arrows in
    \[
    \begin{tikzcd}
        \dots \rar & -1 \rar & 0 \rar & 1 \rar & \dots
    \end{tikzcd}
    \]
    to obtain
    \[
    \begin{tikzcd}
        \dots \rar[leftrightarrow] & -1 \rar[leftrightarrow] & 0 \rar[leftrightarrow] & 1 \rar[leftrightarrow] & \dots.
    \end{tikzcd}
    \]
    This latter category is not zero ``on the nose,'' but all objects are canonically isomorphic to each other (and all hom-sets are singletons), so it is equivalent to zero.
\end{ex}

\subsection{Presheaves on toric categories}

To obtain a geometric object from a toric category $\Cc$, we will need to pass through the world of (multigraded) commutative algebra.
This can be accomplished by considering the category of presheaves $\PSh_k(\Cc) := \Fun(\Cc\op, \Vect_k)$.

Recall that $\PSh_k(\Cc)$ is a cocomplete $k$-linear category, and is in fact the universal such category admitting a functor from $\Cc$.
More precisely, if $\Vc$ is any cocomplete $k$-linear category, then there is a natural equivalence
\[
    \Fun^L_k(\PSh_k(\Cc), \Vc) \simto \Fun(\Cc, \Vc)
\]
obtained by precomposing with the ($k$-linearized) Yoneda embedding $y_\Cc: \Cc \to \PSh_k(\Cc)$.

Furthermore, there exists a \emph{Day convolution} symmetric monoidal structure $\star_\Cc$ on $\PSh_k(\Cc)$ which commutes with colimits in each variable.
Day convolution satisfies a similar universal property: if $(\Vc, \otimes)$ is any cocomplete $k$-linear symmetric monoidal category such that $\otimes$ commutes with colimits in each variable, then there is a natural equivalence
\[
    \Fun^{\otimes,L}_k\big((\PSh_k(\Cc), \star_\Cc), (\Vc, \otimes)\big) \simto \Fun^\otimes(\Cc, \Vc).
\]
In practice, the Day convolution product $\Fsc \star \Gsc$ can be computed by writing $\Fsc$ and $\Gsc$ as colimits of representable presheaves and taking the symmetric monoidal products of the representing objects in $\Cc$.

These abstract universal properties are perfectly nice, but we should justify the claim that we are really doing multigraded commutative algebra here.

\begin{notn}
    If $A$ is an abelian group and $R$ is an $A$-graded commutative ring, we let $\GrMod_R^A$ denote the category of $A$-graded $R$-modules.
\end{notn}

\begin{prop} \label{prop:presheaves_as_graded_modules}
    Let $M$ be a toric monoid, $A$ a discrete abelian group, and $\phi: M \to A$ a homomorphism.
    Then there is a symmetric monoidal equivalence 
    \[
        \big(\PSh_k(A \vec{/} M), \star_{A \vec{/} M}\big) \simeq \big(\GrMod^A_{k\angles{M}}, \otimes_{k\angles{M}}\big).
    \]
\end{prop}

\begin{proof}
    By \zcref{cons:skew_quotient} and discreteness of $A$, the morphism spaces in $A \vec{/} M$ are given by 
    \[
        \Map_{A \vec{/} M}(a_1, a_2) = \bset{m \in M}{\phi(m) + a_1 = a_2}.
    \]
    We may therefore define a symmetric monoidal functor $F: \PSh_k(A \vec{/} M) \to \GrMod^A_{k\angles{M}}$ via $F(\Esc) = \oplus_{a \in A} \Esc(a)$.
    Here $m \in M$ acts on $e \in \Esc(a)$ by $m \cdot e = \Esc\big(\phi(m)\big)(e) \in \Esc(\phi(m) + a)$.
    Note that $F$ has quasi-inverse given by
    \begin{align*}
        \GrMod^A_{k\angles{M}} &\to \PSh_k(A \vec{/} M) \\
        \oplus_{a \in A} E_a &\mapsto (a \mapsto E_a),
    \end{align*}
    so $F$ is a symmetric monoidal equivalence.
\end{proof}

\subsection{Invertible objects in presheaf categories}

For a toric category $\Cc$, we'd like to better understand the relationship between $\Cc$ and $\PSh_k(\Cc)$.
We will show here that the $k$-linearization $k\angles{\Cc}$ is captured by the full subcategory of $\otimes$-invertible objects in $\PSh_k(\Cc)$.

\begin{notn}
    The notation here is potentially ambiguous, since if $M$ is a commutative monoid, then $k\angles{M}$ could refer to the monoid ring of $M$ or to the $k$-linearization of the discrete category $M$.
    We will always take it to be the former whenever this is \emph{a priori} well defined.
    That is, if $M$ is assumed to be a commutative monoid, then we use $k\angles{M}$ for the monoid ring, whereas if $\Cc$ is a symmetric monoidal category (which may or may not happen to be a discrete commutative monoid), we use $k\angles{\Cc}$ for the corresponding $k$-linear category.
\end{notn}

We first make our discussion of $\otimes$-invertible objects somewhat more explicit.

\begin{dfn}
    Let $(\Vc, \otimes)$ be a symmetric monoidal category.
    Write $\Line(\Vc)$ for the full subcategory of $\Cc$ spanned by the $\otimes$-invertible objects.
\end{dfn}

\begin{rmk}
    If $\Cc$ is a toric category and $\Vc$ is any cocomplete $k$-linear symmetric monoidal category such that $\otimes$ commutes with colimits in each variable, we have
    \[
        \Fun^{\otimes,L}\big(\PSh_k(\Cc), \Vc\big) = \Fun^\otimes\big(\Cc, \Line(\Vc)\big).
    \]
    Indeed, this follows by combining the universal property of Day convolution with the fact that any functor $\Cc \to \Line(\Vc)$ sends all objects of $\Cc$ to $\otimes$-invertible objects of $\Vc$.
\end{rmk}

\begin{prop} \label{prop:line_objects_in_presheaves}
    Let $\Cc$ be a toric category.
    Then the natural symmetric monoidal functor $k\angles{\Cc} \to \Line(\PSh_k(\Cc))$ is an equivalence.
\end{prop}

\begin{proof}
    Call the functor in question $F$.
    It is enough to show that $F$ is essentially surjective and fully faithful.
    Write $\Cc \simeq A \vec{/}_\phi M$ for some $\phi: M \to A$ as in \zcref{cor:skew_quotient_of_discrete}.
    
    By \zcref{prop:presheaves_as_graded_modules}, we may identify $\Line(\PSh_k(\Cc))$ with the full subcategory of $\GrMod^A_{k\angles{M}}$ spanned by the $\otimes_{k\angles{M}}$-invertible objects.
    The underlying $k\angles{M}$-module of any such module is isomorphic to $k\angles{M}$ (by \cite[Proposition 4.2.2]{cls} applied to the affine toric variety $\Spec k\angles{M}$).
    Thus the $\otimes_{k\angles{M}}$-invertible objects of $\GrMod^A_{k\angles{M}}$ are precisely the shifts $k\angles{M}(a)$ for $a \in A$.
    The functor $F$ sends $a \in \ob A \vec{/} M$ to $k\angles{M}(a)$, so $F$ is essentially surjective.
    
    To see that $F$ is fully faithful, let $a_1, a_2 \in A$ and compute
    \begin{align*}
        \Hom_{\GrMod^A_{k\angles{M}}}\big(k\angles{M}(a_1), k\angles{M}(a_2)\big) &= k\angles{M}_{a_2 - a_1} \\
        &= k\angles{\phi\inv(a_2 - a_1)} \\
        &= k\angles{\Map_\Cc(a_1, a_2)} \\
        &= \Hom_{k\angles{M}}(a_1, a_2).
    \end{align*}
    Thus $F$ is fully faithful.
\end{proof}

In other words, if $\Cc$ is a toric category, then we can recover the $k$-linearization $k\angles{\Cc}$ from the symmetric monoidal category $\PSh_k(\Cc)$.

\subsection{Reconstruction of toric categories from linear data} \label{sub:reconstruction}

The preceding results suggest a natural question: what extra data do we need to recover a toric category $\Cc$ (rather than just $k\angles{\Cc}$) from the category $\PSh_k(\Cc)$?
The answer turns out to be closely related to the basic theory of coalgebras.
This discussion is somewhat technical but will be important in \zcref{sub:fully_faithful}; readers are welcome to skip this subsection until it is needed.

Recall that, if $(C, \delta, \epsilon)$ is a coalgebra, an element $c \in C$ is said to be \emph{setlike} if $\delta(c) = c \otimes c$ and $\epsilon(c) = 1$.
Write $C^\setlike$ for the set of setlike elements of $C$.
Note that $(-)^\setlike$ defines a symmetric monoidal functor $\CoAlg_k \to \Set$.
In fact, there is an adjunction
\[
    k\angles{-}: \Set \rightleftarrows \CoAlg_k : (-)^\setlike.
\]
If $S$ is a set, then a standard exercise shows that the natural map $S \to k\angles{S}^\setlike$ is an isomorphism.
Equivalently, the functor $k\angles{-}$ is fully faithful when thought of as a functor valued in coalgebras.

This whole story extends to the setting of enriched (symmetric monoidal) categories.
Let $\SMCat_{\CoAlg_k}$ be the 2-category of symmetric monoidal categories enriched in $\CoAlg_k$.
If $\Cc \in \SMCat_{\CoAlg_k}$, we say a morphism $f: c_1 \to c_2$ in $\Cc$ is \emph{setlike} if it is setlike as an element of the coalgebra $\Map_\Cc(c_1, c_2)$.
There is a natural adjunction
\[
    k\angles{-}: \SMCat \rightleftarrows \SMCat_{\CoAlg_k} : (-)^\setlike.
\]
where the notation for each functor describes what it does to the mapping objects in a (possibly enriched) symmetric monoidal category.

Because $k\angles{-}^\setlike$ is naturally isomorphic to the identity functor on $\Set$, we immediately obtain:

\begin{prop} \label{prop:setlike_morphisms_recover_category}
    The composite 2-functor $k\angles{-}^\setlike: \SMCat \to \SMCat$ is naturally equivalent to the identity. 
\end{prop}

This is an extension of the fact that, if $M$ is a toric monoid, then $M$ can be recovered from the multiplication on the affine toric variety $\Spec k\angles{M}$ (which is a commutative monoid scheme).
Unfortunately, this result is not so useful when developing the theory of toric varieties.
Instead, the classical theory makes great use of the fact that $M$ can be recovered from the affine toric variety $\Spec k\angles{M}$ together with the action of the dense torus $\Spec k\angles{M^\gp}$ (with no reference to the commutative monoid structure on all of $\Spec k\angles{M}$).
We will prove an analogous result in \zcref{prop:linear_reconstruction_algebraic}, but we first need some lemmas.

\begin{lem} \label{lem:isos_in_kC}
    Let $\Cc$ be a toric category.
    Then any isomorphism in $k\angles{\Cc}$ is a scalar multiple of a setlike isomorphism.
\end{lem}

\begin{proof}
    Suppose $f$ is an isomorphism in $k\angles{\Cc}$.
    Replacing $f$ by its image in $k\angles{\Cc^\gpd}$ and noting that a morphism in $k\angles{\Cc}$ is setlike if and only if the same is true for its image in $k\angles{\Cc^\gpd}$, we see that it suffices to show the claim when $\Cc$ is a toric groupoid.
    In this case we may use \zcref{prop:deligne_classification} to write $\Cc \simeq A_0 \times \Bsf A_1$ where $A_0$ is a finitely generated abelian group and $A_1$ is a finitely generated free abelian group.
    Every isomorphism in $k\angles{\Cc}$ may therefore be represented by a unit in the ring $k\angles{A_1}$.
    But the units of $k\angles{A_1}$ are all scalar multiples of setlike morphisms (e.g.\ by \cite[Exercise 3.3.4]{cls}), so we win.
\end{proof}

\begin{lem} \label{lem:toric_cat_as_fiber_product}
    Let $\Cc$ be a toric category.
    Then $\Cc = k\angles{\Cc} \times_{k\angles{\Cc^\gpd}} \Cc^\gpd$ as symmetric monoidal categories.
\end{lem}

\begin{proof}
    There's an obvious symmetric monoidal functor $F: \Cc \to k\angles{\Cc} \times_{k\angles{\Cc^\gpd}} \Cc^\gpd$ induced by the universal property of fiber products.
    We need to show that this is an isomorphism.

    Recall that $k\angles{\Cc} \times_{k\angles{\Cc^\gpd}} \Cc^\gpd$ has:
    \begin{itemize}
        \item Objects given by triples $(c_1, c_2, \phi)$ where $c_1 \in k\angles{\Cc}$, $c_2 \in \Cc^\gpd$, and $\phi: c_1 \simto c_2$ as objects of $k\angles{\Cc^\gpd}$.
        \item Morphisms $(c_1, c_2, \phi) \to (c_1', c_2', \phi')$ given by pairs $(f_1, f_2)$ with $f_1 \in \Hom_{k\angles{\Cc}}(c_1, c_1')$ and $f_2 \in \Map_{\Cc^\gpd}(c_2, c_2')$ such that $f_2 \circ \phi = \phi' \circ f_1$.
    \end{itemize}
    The functor $F$ is given on objects by $c \mapsto (c, c, \id_c)$ and on morphisms by $f \mapsto (f, f)$.
    
    To show that $F$ is essentially surjective, it suffices to show that any object of $k\angles{\Cc} \times_{k\angles{\Cc^\gpd}} \Cc^\gpd$ can be represented up to isomorphism by an object of the form $(c, c, \id_c)$ for $c \in \Cc$.
    To achieve this, let $(c_1, c_2, \phi) \in k\angles{\Cc} \times_{k\angles{\Cc^\gpd}} \Cc^\gpd$ be arbitrary.
    By \zcref{lem:isos_in_kC}, we may write $\phi = a \cdot \phi'$ for some $a \in k$ and some $\phi' \in \Cc^\gpd$.
    Then the pair $(a, (\phi')\inv)$ gives an isomorphism $(c_1, c_2, \phi) \simto (c_1, c_1, \id_{c_1})$, as needed.

    It remains to show that $F$ is fully faithful, i.e.\ that every $(f_1, f_2): (c, c, \id_{c}) \to (c', c', \id_{c'})$ is of the form $(f, f)$ for some $f \in \Map_\Cc(c, c')$.
    Commutativity of the diagram
    \[
    \begin{tikzcd}
        c \rar["\id"] \dar["f_1"] & c \dar["f_2"] \\
        c' \rar["\id"] & c'
    \end{tikzcd}
    \]
    is precisely saying that $f_1 = f_2$ as morphisms in $k\angles{\Cc^\gpd}$.
    Because $f_2$ is setlike in $k\angles{\Cc^\gpd}$, $f_1$ must itself be setlike in $k\angles{\Cc}$, hence must be induced by some $f \in \Map_\Cc(c, c')$.
    The image of $f$ in $\Cc^\gpd$ is $f_2$, so $(f_1, f_2) = (f, f)$, as claimed.
\end{proof}

By combining our previous results, we obtain:

\begin{prop} \label{prop:linear_reconstruction_algebraic}
    Let $\Cc$ be a toric category.
    Then
    \[
        \Cc = \Line(\PSh_k(\Cc)) \times_{\Line(\PSh_k(\Cc^\gpd))} \Line(\PSh_k(\Cc^\gpd))^\setlike.
    \]
\end{prop}

\begin{proof}
    \zcref{prop:line_objects_in_presheaves} tells us that $\Line(\PSh_k(\Cc)) = k\angles{\Cc}$ and $\Line(\PSh_k(\Cc^\gpd)) = k\angles{\Cc^\gpd}$.
    Furthermore, by \zcref{prop:setlike_morphisms_recover_category}, we have
    \[
        \Line\big(\PSh_k(\Cc^\gpd)\big)^\setlike = k\angles{\Cc^\gpd}^\setlike = \Cc^\gpd.
    \]
    Combining these claims with \zcref{lem:toric_cat_as_fiber_product} gives the result.
\end{proof}

\section{Gale duality and consequences} \label{sec:gale_duality}

We can reformulate the (derived) Gale duality of \cite[\S 2]{borisov2005orbifold} as an anti-automorphism of the 2-category $\TCat$.
Using this language, we are able to give a simple axiomatic characterization of the Gale duality functor (see \zcref{thm:gale_duality}).

\subsection{Gale duality for toric categories} \label{sub:Gale_duality}

We begin by giving a concrete construction of the Gale duality functor, generalizing that of \cite[\S 2]{borisov2005orbifold}.

\begin{cons} \label{cons:gale_dual}
    Let $\Cc$ be a toric category.
    We define a new toric category $\GDual(\Cc)$, the \emph{Gale dual of $\Cc$}, as follows.
    Write $\Cc$ as a skew quotient $\Omega^\infty(F^\bullet) \vec{/} M$ (using e.g.\ \zcref{thm:TCat_structure_thm}).
    Let $E^\bullet = \fib(M^\gp \to F^\bullet) \in \Dsf(\ZZ)$.
    By \zcref{lem:gale_duality_WD1} below, $\RHom(E^\bullet, \ZZ)$ may be represented by an object of $\tTerm$, so $\Omega^\infty \RHom(E^\bullet, \ZZ)$ is a well-defined strictly abelian 2-group.
    Furthermore, there is a natural homomorphism
    \[
        D_\CMon(M) \to \Hom(M^\gp, \ZZ) = \RHom(M^\gp, \ZZ) \to \RHom(E^\bullet, \ZZ).
    \]
    Let
    \[
        \GDual(\Cc) = \Omega^\infty(\RHom(E^\bullet, \ZZ)) \vec{/} D_\CMon(M).
    \]
    Then $\GDual(\Cc)$ is a toric category by \zcref{lem:gale_duality_WD2} below and is independent of the choices made (up to natural equivalence) by \zcref{lem:gale_duality_WD3} below.
\end{cons}

\begin{lem} \label{lem:gale_duality_WD1}
    Let $F^\bullet \in \tTerm$, and let $A$ be a finitely generated free abelian group.
    Then 
    \[
    \RHom(\fib(A \to F^\bullet), \ZZ)
    \]
    lies in the essential image of the inclusion $\tTerm \hookrightarrow \Dsf^{[-1,0]}(\ZZ)$.
\end{lem}

\begin{proof}
    Let
    \[
    C^\bullet = \RHom(\fib(A \to F^\bullet), \ZZ) = \cofib\big(\RHom(F^\bullet, \ZZ) \to \RHom(A, \ZZ)\big).
    \]
    We may represent $\RHom(F^\bullet, \ZZ)$ by a complex of finitely generated free abelian groups with nonzero terms only in degrees $0$ and $1$.
    Likewise, we may represent $\RHom(A, \ZZ)$ by the finitely generated free abelian group $\Hom(A, \ZZ)$ concentrated in degree $0$.
    By the standard formula for cofibers / mapping cones (see e.g.\ \stacks{09KA}), this implies $C^\bullet$ may be represented by a complex of finitely generated abelian groups with nonzero terms only in degrees $-1$ and $0$, i.e.\ a complex in $\tTerm$.
\end{proof}

\begin{lem} \label{lem:gale_duality_WD2}
    Let $\Cc$ be a toric category.
    Then $\GDual(\Cc)$ is a toric category.
\end{lem}

\begin{proof}
    In the setting of \zcref{cons:gale_dual}, we have already seen that $\RHom(E^\bullet, \ZZ) \in \tTerm$.
    Furthermore, by \cite[Theorem I.2.2.3]{ogus2018lectures}, $D_\CMon(M)$ is toric.
    Thus \zcref{cor:skew_quotient_is_T_cat} implies $\GDual(\Cc)$ is a toric category.
\end{proof}

\begin{lem} \label{lem:gale_duality_WD3}
    The above construction of $\GDual(\Cc)$ is independent of the choice of skew quotient presentation of $\Cc$.
\end{lem}

\begin{proof}
    Write $\Cc = \Omega^\infty(F^\bullet) \vec{/} M$ and use the notation of \zcref{cons:gale_dual} for the corresponding computation of $\GDual(\Cc)$.
    Let $(\GDual)'(\Cc)$ be the Gale dual of $\Cc$ constructed using the skew quotient presentation
    \[
        \Cc = \Omega^\infty(\cofib(M^\pm \to F^\bullet)) \vec{/} M^\sharp.
    \]
    We claim that there is a natural equivalence $D'_\Gale(\Cc) = \GDual(\Cc)$.

    Let 
    \[
        (E')^\bullet = \fib\big((M^\sharp)^\gp \to \cofib(M^\pm \to F^\bullet)\big) \in \Dsf(\ZZ).
    \]
    By \cite[Proposition I.1.3.5]{ogus2018lectures}, we can write $(M^\sharp)^\gp = M^\gp / M^\pm$.
    As complexes, this gives $(M^\sharp)^\gp = \cofib(M^\pm \to M^\gp)$.
    Thus all columns and the bottom two rows in the following square are cofiber sequences in $\Dsf(\ZZ)$:
    \[
    \begin{tikzcd}
        0 \rar \dar & E^\bullet \rar \dar & (E')^\bullet \dar \\
        M^\pm \rar \dar & M^\gp \rar \dar & \cofib(M^\pm \to M^\gp) \dar \\
        M^\pm \rar & F^\bullet \rar & \cofib(M^\pm \to F^\bullet)
    \end{tikzcd}
    \]
    It follows that the top row is a cofiber sequence, i.e.\ the natural map $E^\bullet \to (E')^\bullet$ is an equivalence.
    Thus
    \[
        \iota_0 (\GDual)'(\Cc) = \Omega^\infty(\RHom((E')^\bullet, \ZZ)) = \Omega^\infty(\RHom(E^\bullet, \ZZ)) = \iota_0 \GDual(\Cc).
    \]

    We can also see that
    \[
        \PO (\GDual)'(\Cc) = D_\CMon(M^\sharp) = D_\CMon(M) = \PO \GDual(\Cc).
    \]
    The natural maps $\PO \GDual(\Cc) \to \iota_0 \GDual(\Cc)$ and $\PO D'_\Gale(\Cc) \to \iota_0 D'_\Gale(\Cc)$ agree under these equivalences.
    Therefore \zcref{thm:TCat_structure_thm} implies that $\GDual(\Cc) = (\GDual)'(\Cc)$, as claimed.
    
    Because $M^\sharp$ is sharp, \zcref{thm:TCat_structure_thm} implies that the skew quotient presentation used in computing $D'_\Gale(\Cc)$ is unique.
    In other words, there exists a canonical skew quotient presentation we can use to compute the Gale dual, and this agrees with the Gale dual as computed using an arbitrary skew quotient presentation.
    Thus the construction of the Gale dual is independent of the choice of skew quotient presentation.
\end{proof}

Here are a few key examples to keep in mind:

\begin{ex} \label{ex:ZmodN_Gale_dual}
    Let $\Cc = \ZZ \vec{/} \NN$, so $F^\bullet = \ZZ$.
    Then $E^\bullet = \fib(\ZZ \simto \ZZ) = 0$, so $\RHom(E^\bullet, \ZZ) = 0$.
    Using $\Hom(\NN, \NN) = \NN$, we see that
    \[
        \GDual(\ZZ \vec{/} \NN) = 0 \vec{/} \NN = \Bsf \NN.
    \]
\end{ex}

\begin{ex} \label{ex:Z_Gale_dual}
    Let $\Cc = \ZZ = \ZZ \vec{/} 0$, so $F^\bullet = \ZZ$.
    Then $E^\bullet = \fib(0 \to \ZZ) = \ZZ[-1]$, so $\RHom(E^\bullet, \ZZ) = \ZZ[1]$ and $\Omega^\infty \RHom(E^\bullet, \ZZ) = \Bsf \ZZ$.
    Therefore
    \[
        \GDual(\ZZ) = (\Bsf \ZZ) \vec{/} 0 = \Bsf \ZZ.
    \]
\end{ex}

The above definition of Gale duality is quite concrete but admittedly a bit hard to understand ``conceptually.''
We have defined it in this way to highlight the connections with the classical theory of Gale duality.
By using the language of toric categories, we are able to provide an alternative axiomatic description of Gale duality:

\begin{thm}[Categorical Gale duality] \label{thm:gale_duality}
    The following are true:
    \begin{enumerate}
        \item Gale duality defines a contravariant 2-functor $\GDual: \TCat\op \to \TCat$.
        \item There is a natural equivalence $(\GDual)^2 = \id$ (so in particular $\GDual$ is a contravariant autoequivalence).
        \item The 2-functor $\GDual$ is uniquely characterized (up to essentially unique equivalence) by the properties:
        \begin{itemize}
            \item $(\GDual)^2 = \id$,
            \item $\GDual(\ZZ) = B\ZZ$,
            \item $\GDual(\ZZ \vec{/} \NN) = \Bsf \NN$, and
            \item $\GDual$ sends the skew quotient map $\ZZ \twoheadrightarrow \ZZ \vec{/} \NN$ to the inclusion map $\Bsf \NN \hookrightarrow \Bsf \ZZ$.
        \end{itemize}
    \end{enumerate}
\end{thm}

\begin{proof}
    (1). By \zcref{thm:TCat_structure_thm}, it is possible to write every toric category as a skew quotient in a functorial manner.
    All of the other steps in the construction of $\GDual$ are functorial, so we see that $\GDual$ is a (contravariant) 2-functor.

    (2). Let $\Cc$ be a toric category.
    Take the skew quotient presentation $\Cc = \Omega^\infty(F^\bullet) \vec{/} M$ of \zcref{thm:TCat_structure_thm}, so $M$ is sharp and $D_\CMon^2(M) = M$.
    Using the notation of \zcref{cons:gale_dual}, we have $\GDual(\Cc) = \Omega^\infty(\RHom(E^\bullet, \ZZ)) \vec{/} D_\CMon(M)$.
    Let $G^\bullet = \fib\big(D_\CMon(M)^\gp \to \RHom(E^\bullet, \ZZ)\big) \in \Dsf(\ZZ)$.
    By \cite[Theorem I.2.2.3]{ogus2018lectures}, we have $D_\CMon(M)^\gp = \Hom(M^\gp, \ZZ) = \RHom(M^\gp, \ZZ)$.
    Therefore
    \[
        G^\bullet = \fib\big(\RHom(M^\gp, \ZZ) \to \RHom(E^\bullet, \ZZ)\big) = \RHom\big(\cofib(E^\bullet \to M^\gp), \ZZ\big) = \RHom(F^\bullet, \ZZ),
    \]
    so $\RHom(G^\bullet, \ZZ) = F^\bullet$.
    We see that 
    \[
        \GDual^2(\Cc) = \Omega^\infty(\RHom(G^\bullet, \ZZ)) \vec{/} D_\CMon^2(M) = \Omega^\infty(F^\bullet) \vec{/} M = \Cc
    \]
    as desired.

    (3). By \zcref{thm:TCat_structure_thm}, for any toric category $\Cc$, the toric category $\GDual(\Cc)$ is determined by the strictly abelian 2-group $\iota_0 \GDual(\Cc)$, the sharp toric monoid $\PO \GDual(\Cc)$, and the natural homomorphism $\PO \GDual(\Cc) \to \iota_0 \GDual(\Cc)$.
    We can compute $\iota_0 \GDual(\Cc)$ directly using the stated properties:
    \begin{align*}
        \iota_0 \GDual(\Cc) &= \iota_0 \Fun^\oplus(\ZZ, \GDual(\Cc)) \textrm{ by \zcref{ex:i0_corepresented}} \\
        &= \iota_0 \Fun^\oplus((\GDual)^2(\Cc), \GDual(\ZZ)) \\
        &= \iota_0 \Fun^\oplus(\Cc, \Bsf \ZZ).
    \end{align*}
    Likewise, we can compute
    \begin{align*}
        \PO \GDual(\Cc) &= \iota_0 \Fun^\oplus(\ZZ \vec{/} \NN, \GDual(\Cc)) \textrm{ by \zcref{ex:P0_corepresented}} \\
        &= \iota_0 \Fun^\oplus((\GDual)^2(\Cc), \GDual(\ZZ \vec{/} \NN)) \\
        &= \iota_0 \Fun^\oplus(\Cc, \Bsf \NN).
    \end{align*}
    From these descriptions and our description of $\GDual$ of the skew quotient map $\ZZ \to \ZZ \vec{/} \NN$, we can recover the map $\PO \GDual(\Cc) \to \iota_0 \GDual(\Cc)$: it is the map $\iota_0 \Fun^\oplus(\Cc, \Bsf \NN) \to \iota_0 \Fun^\oplus(\Cc, \Bsf \ZZ)$ obtained from the inclusion $\Bsf \NN \hookrightarrow \Bsf \ZZ$.
    Thus $\GDual(\Cc)$ is uniquely specified by the stated properties.
\end{proof}

\begin{cor} \label{cor:gale_dual_as_skew_quotient}
    Let $\Cc$ be a toric category.
    Then 
    \[
        \GDual(\Cc) = \iota_0 \Fun^{\oplus}(\Cc, \Bsf \ZZ) \vec{/} \iota_0 \Fun^{\oplus}(\Cc, \Bsf \NN),
    \]
    where the map $\Fun^{\oplus}(\Cc, \Bsf \NN) \to \Fun^{\oplus}(\Cc, \Bsf \ZZ)$ is induced by the standard inclusion $\NN \hookrightarrow \ZZ$.
\end{cor}

\begin{proof}
    This follows by the proof of \zcref{thm:gale_duality}(3) combined with \zcref{thm:TCat_structure_thm}.
\end{proof}

\begin{qn} \label{qn:gale_real_params}
    In convex geometry, Gale duality is often phrased in terms of real vector spaces, with no reference to the questions of rationality which appear in toric geometry.
    Can the axiomatic perspective on Gale duality (\zcref{thm:gale_duality}) be extended to account for real parameters?
\end{qn}

\begin{rmk}
    An affirmative answer to \zcref{qn:gale_real_params} would likely have applications to the study of quantum toric stacks (see e.g.\ \cite{katzarkov2021quantum}) and symplectic duality (see e.g.\ \cite{braden2010gale}).
\end{rmk}

The axiomatic perspective on Gale duality can make it easier to understand what Gale duality does to morphisms.
For instance:

\begin{prop} \label{prop:gale_duality_localization}
    Gale duality exchanges localizations and inclusions of faces: if $\Cc$ is a toric category and $S \subset \PO \Cc$ is a face, then $\GDual$ sends the localization functor $\Cc \to \Cc[S\inv]$ to the inclusion functor 
    \[
        \iota_0 \GDual(\Cc) \vec{/} S^\bot \hookrightarrow \iota_0 \GDual(\Cc) \vec{/} \Hom(\PO \Cc, \NN) = \GDual(\Cc).
    \]
    In particular, $\iota_0 \GDual(\Cc) = \GDual(\Cc^\gpd)$.
\end{prop}

\begin{proof}
    We can compute
    \begin{align*}
        \iota_0 \GDual(\Cc[S\inv]) &= \iota_0 \Fun^\oplus(\Cc[S\inv], \Bsf\ZZ) \\
        &= \iota_0 \Fun^\oplus(\Cc, \Bsf\ZZ) \textrm{ by the universal property of $\Cc[S\inv]$} \\
        &= \iota_0 \Cc
    \end{align*}
    and
    \begin{align*}
        \PO \GDual(\Cc[S\inv]) &= \iota_0 \Fun^\oplus(\Cc[S\inv], \Bsf\NN) \\
        &= \iota_0 \Fun^\oplus\big(\cofib(S^\gp \to \iota_0 \Cc) \vec{/} (\PO \Cc / S), \Bsf \NN\big) \textrm{ by \zcref{prop:localization_properties}} \\
        &= 0 \times_0 \Hom(\PO \Cc / S, \NN) \textrm{ by \zcref{prop:skew_quotient_UP}} \\
        &= \Hom(\PO \Cc / S, \NN) \\
        &= \Hom(\PO \Cc, \NN) \times_{\Hom(S, \NN)} 0 \\
        &= \bset{\phi: S \to \NN}{\phi(S) = 0} \\
        &= S^\bot
    \end{align*}
    so 
    \[
        \GDual(\Cc[S\inv]) = \iota_0 \GDual(\Cc) \vec{/} \PO \GDual(\Cc) = \iota_0 \GDual(\Cc) \vec{/} S^\bot
    \]
    by \zcref{thm:TCat_structure_thm}.
    Tracing through the definitions shows that the natural morphism $\GDual(\Cc[S\inv]) \to \GDual(\Cc)$ agrees with the morphism induced by face inclusion, as needed.
    The ``in particular'' claim follows if we take $S = \PO \Cc$, so $\Cc[S\inv] = \Cc^\gpd$ and $S^\bot = 0$.
\end{proof}

\subsection{A collection of useful examples}

\begin{ex} \label{ex:Gale_dual_of_lattice}
    Let $A$ be a finitely generated free abelian group.
    Then $A = A \vec{/} 0$, so $\PO \GDual(A) = D_\CMon(0) = 0$ and 
    \[
        \iota_0 D_\Gale(A) = \iota_0 \Fun^\oplus(A, \Bsf \ZZ) = \iota_0 \Omega^\infty \RHom(A, \ZZ[1]) = \iota_0 \Omega^\infty D_\Ab(A)[1] = \Bsf D_\Ab(A).
    \]
    Thus
    \[
        \GDual(A) = \Bsf D_\Ab(A).
    \]
    Since $\GDual$ is an involution, we must also have
    \[
        \GDual(\Bsf A) = D_\Ab(A).
    \]
\end{ex}

\begin{ex} \label{ex:Gale_dual_classical_toric}
    Let $M$ be a toric monoid.
    Then \zcref{cor:gale_dual_as_skew_quotient} implies
    \begin{align*}
        \GDual(\Bsf M) &= \iota_0 \Fun^\oplus(\Bsf M, \Bsf \ZZ) \vec{/} \iota_0 \Fun^\oplus(\Bsf M, \Bsf \NN) \\
        &= \Hom(M, \ZZ) \vec{/} \Hom(M, \NN).
    \end{align*}
    The skew quotient categories arising in this way are precisely those of the form $N \vec{/} (\ul{\sigma} \cap N)$, where $N$ is a finitely generated free abelian group and $\ul{\sigma}$ is a strongly convex rational polyhedral cone in $N \otimes \RR$.
    
    Indeed, $N = \Hom(M, \ZZ)$ is a finitely generated free abelian group and $\Hom(M, \NN)$ is a finitely generated submonoid of $N$ which is saturated in $N$, so \zcref{prop:sat_submon_of_lattice} implies every skew quotient $\Hom(M, \ZZ) \vec{/} \Hom(M, \NN)$ arises in this way.
    Conversely, given such a pair $(N, \ul{\sigma})$, let
    \[
        M = \bset{\phi: N \to \ZZ}{\phi(\ul{\sigma} \cap N) \geq 0}.
    \]
    Then the theory of dual cones (specifically \cite[Proposition 1.2.4 and 1.2.12]{cls}) implies $N = \Hom(M, \ZZ)$ (using strong convexity of $\ul{\sigma}$) and $\ul{\sigma} \cap N = \Hom(M, \NN)$.
    Thus $\GDual(\Bsf M) = N \vec{/} (N \cap \ul{\sigma})$.
    Because $\GDual$ is an involution, we obtain
    \[
        \GDual(N \vec{/} (\ul{\sigma} \cap N)) = \Bsf M = \Bsf \big(\bset{\phi: N \to \ZZ}{\phi(\ul{\sigma} \cap N) \geq 0}\big).
    \]
\end{ex}

\begin{ex}
    As a special case of the above, we see that if $N$ is a finitely generated free abelian group and $\ul{\sigma} \subset N \otimes \RR$ is a strongly convex rational polyhedral cone, then 
    \[
        \GDual(\Bsf(\ul{\sigma} \cap N)) = \Hom(\ul{\sigma} \cap N, \ZZ) \vec{/} \Hom(\ul{\sigma} \cap N, \NN)
    \]
    One can show $\GDual(\Bsf (\ul{\sigma} \cap N))$ is a poset with underlying objects given by cosets of $\bset{\phi: N \to \ZZ}{\phi(\ul{\sigma} \cap N) = 0}$ in $\Hom(N, \ZZ)$.

    Let $M = \Hom(N, \ZZ)$, and let $\Dsf_c^b(M \otimes \RR, k)$ be the derived dg-category of constructible sheaves of $k$-vector spaces on the real vector space $M \otimes \RR$.
    Passing to the $k$-linearization and using \cite[Proposition 3.3]{fang2011categorification}, we see that $k\angles{\GDual(\Bsf (\ul{\sigma} \cap N))}$ can be identified with the full subcategory of $\Dsf_c^b(M \otimes \RR, k)$ generated by certain ``costandard sheaves'' $\Theta(\ul{\sigma}, \chi)$.
    In particular, Gale duality is closely related to homological mirror symmetry and the coherent-constructible correspondence for toric varieties.
\end{ex}

\begin{ex} \label{ex:gale_dual_of_product}
    Let $\Cc_1$ and $\Cc_2$ be toric categories.
    The Cartesian product 
    \[
        \Cc_1 \times \Cc_2 = (\iota_0 \Cc_1 \times \iota_0 \Cc_2) \vec{/} (\PO \Cc_1 \times \PO \Cc_2)
    \]
    is also toric, and following the construction of \zcref{cons:gale_dual} shows that
    \[
        \GDual(\Cc_1 \times \Cc_2) = \GDual(\Cc_1) \times \GDual(\Cc_2).
    \]
\end{ex}

\subsection{Limits and colimits of toric categories}

\begin{prop} \label{prop:Tcat_finite_lim_colim}
    The 2-category $\TCat$ admits finite limits and colimits.
\end{prop}

\begin{proof}
    Existence of finite limits in $\TCat$ follows from \zcref{thm:TCat_structure_thm}.
    Indeed, if $\Cc_{-}: I \to \TCat$ is a diagram, then 
    \begin{equation} \label{eq:finite_limit_tcat}
        \lim_i \Cc_i = \lim_i (\iota_0 \Cc_i \vec{/} \PO \Cc_i) = \big(\lim_i \iota_0 \Cc_i\big) \vec{/} \big(\lim_i \PO \Cc_i\big)
    \end{equation}
    where the limits on the right hand side exist by \zcref{lem:2term_finite_limits} (using the equivalence $\TGpd \simeq \tTerm$) and \zcref{prop:toric_monoid_finite_limit}.

    Existence of finite colimits then follows from \zcref{thm:gale_duality}, since $\GDual$ interchanges finite limits and finite colimits.
    More precisely, given $\Cc_{-}: I \to \TCat$, we have
    \begin{equation} \label{eq:finite_colimit_tcat_via_gale}
        \colim_i \Cc_i = \colim_i (\GDual)^2(\Cc_i) = \GDual\left(\lim_i \GDual(\Cc_i)\right).
    \end{equation}
    Alternatively, we can compute 
    \begin{equation} \label{eq:finite_colimit_tcat}
        \colim_i \Cc_i = \big(\colim_i \iota_0 \Cc_i\big) \vec{/} \big(\colim_i \PO \Cc_i\big)
    \end{equation}
    where the colimits in question are computed in $\TGpd$ and $\CMon^{\sharp,\toric}$.
\end{proof}

\eqref{eq:finite_limit_tcat} and \eqref{eq:finite_colimit_tcat_via_gale} give us effective means of computing finite limits and colimits in $\TCat$.
General finite colimits of toric categories can be difficult to understand conceptually, but we do have the following result:

\begin{prop} \label{prop:finite_product_coproduct_TCat}
    Finite products agree with finite coproducts in $\TCat$.
\end{prop}

\begin{proof}
    It suffices to consider the case of binary coproducts $\Cc_1 \oplus \Cc_2$.
    We can compute
    \begin{align*}
        \Cc_1 \oplus \Cc_2 &= \GDual(\GDual(\Cc_1) \times \GDual(\Cc_2)) \textrm{ by \zcref{eq:finite_colimit_tcat_via_gale}} \\
        &= (\GDual)^2(\Cc_1) \times (\GDual)^2(\Cc_2) \textrm{ by \zcref{ex:gale_dual_of_product}} \\
        &= \Cc_1 \times \Cc_2 \textrm{ by \zcref{thm:gale_duality}.} \qedhere
    \end{align*}
\end{proof}

\part{Stacky fans and toric stacks} \label{part:fans_and_stacks}

Our goal in this part is to give detailed definitions of the 2-categories $\StkFan$ and $\TStk$, as well as the geometrization functor $\Xfr_{-}$, appearing in our classification theorem for toric stacks.

\section{Stacky fans} \label{sec:stacky_fans}

We first introduce and study a 2-category $\StkFan$ of stacky fans.

\subsection{Toric decorated posets and stacky cones}

The category of schemes is typically constructed as a full subcategory of the ambient category of locally ringed spaces.
The following definition gives an analogous ambient category for stacky cones and fans:

\begin{dfn}
    A \emph{toric decorated poset} is a triple $\Sigma = (\Nc_\Sigma, |\Sigma|, S)$ where:
    \begin{itemize}
        \item $\Nc_\Sigma$ is a toric groupoid,
        \item $|\Sigma|$ is a finite poset, and
        \item $P_\Sigma: |\Sigma| \to (\CMon^{\sharp,\toric})_{/\Nc_\Sigma}$ is a functor sending each $\sigma \in |\Sigma|$ to a homomorphism of commutative 2-monoids $b_\sigma: P_{\Sigma,\sigma} \to \Nc_\Sigma$ where $P_{\Sigma,\sigma}$ is discrete, sharp, and toric.
    \end{itemize}
\end{dfn}

\begin{notn}
    When we are considering a single toric decorated poset $\Sigma$, we will often omit the subscript $\Sigma$ from our notation, writing $\Nc$ for the corresponding poset and $P: |\Sigma| \to (\CMon^{\sharp,\toric})_{/\Nc}$ for the corresponding functor.
\end{notn}

We can assemble these into a 2-category as follows:

\begin{dfn}
    The \emph{2-category of toric decorated posets} $\TDP$ is the 2-category where:
    \begin{itemize}
        \item Objects are toric decorated posets $\Sigma$,
        \item A 1-morphism $f: \Sigma \to \Sigma'$ is a triple $(f_{\#,\Nc}, |f|, \{ f_{\#,\sigma} \}_{\sigma \in \Sigma})$ where:
        \begin{itemize}
            \item $f_{\#,\Nc}: \Nc_\Sigma \to \Nc_{\Sigma'}$ is a homomorphism of toric groupoids,
            \item $|f|: |\Sigma| \to |\Sigma'|$ is an order-preserving map, and
            \item For $\sigma \in |\Sigma|$, $f_{\#,\sigma}: P_{\Sigma,\sigma} \to P_{\Sigma',f(\sigma)}$ is a 1-morphism in $(\CMon^{\sharp,\toric})_{/\Nc_{\Sigma'}}$ (where the structure morphism $P_{\Sigma,\sigma} \to \Nc_{\Sigma'}$ is $f_{\#,\Nc} \circ b_\sigma$),
        \end{itemize}
        such that:
        \begin{itemize}
            \item For all $\sigma_1 \leq \sigma_2 \in |\Sigma|$, the diagram
            \begin{equation} \label{eq:tdp_1_morphism_square}
            \begin{tikzcd}
                P_{\Sigma,\sigma_1} \rar \dar["f_{\#,\sigma_1}"] & P_{\Sigma,\sigma_2} \dar["f_{\#,\sigma_2}"] \\
                P_{\Sigma',f(\sigma_1)} \rar & P_{\Sigma',f(\sigma_2)}
            \end{tikzcd}
            \end{equation}
            commutes in $(\CMon^{\sharp,\toric})_{/\Nc_{\Sigma'}}$.
            \item For all $\sigma \in |\Sigma|$, we have $P_{f(\sigma)} = \angles{\im f_{\#,\sigma}}$.\footnote{This ``density'' condition should be thought of as analogous to the requirement that a homomorphism of locally ringed spaces sends maximal ideals of local rings to maximal ideals of local rings.}
        \end{itemize}
        \item 2-morphisms $\alpha: f \Rightarrow g$ between 1-morphisms exist if and only if $|f| = |g|$, in which case they are given by $(\alpha_\Nc, \{ \alpha_\sigma \}_{\sigma \in |\Sigma|})$
        where:
        \begin{itemize}
            \item $\alpha_\Nc: f_{\#,\Nc} \Rightarrow g_{\#,\Nc}$ is a symmetric monoidal natural transformation, and
            \item For $\sigma \in |\Sigma|$, $\alpha_\sigma: f_{\#,\sigma} \to g_{\#,\sigma}$ is a 2-morphism in $(\CMon^{\sharp,\toric})_{/\Nc_{\Sigma'}}$ where $\alpha_\Nc$ is used to identify the structure morphisms $f_{\#,\Nc} \circ b_\sigma$ and $g_{\#,\Nc} \circ b_\sigma$ as homomorphisms $P_{\Sigma,\sigma} \to \Nc_{\Sigma'}$.
        \end{itemize} 
    \end{itemize}
\end{dfn}

\subsection{Definition of stacky cones}

Our first important examples of toric decorated posets are stacky cones.

\begin{dfn} \label{dfn:stacky_cone}
    Let $\Cc$ be a toric category.
    The \emph{stacky cone associated with $\Cc$}, written $\Cone(\Cc)$, is the toric decorated poset $(\iota_0 \Cc, \Face(\PO \Cc), \incl_\Cc)$ where:
    \begin{itemize}
        \item $\Face(\PO \Cc)$ is the poset of faces of $\PO \Cc$, and
        \item $\incl_\Cc: \Face(\PO \Cc) \to (\CMon^{\sharp,\toric})_{/\iota_0 \Cc}$ sends $S \subset \PO \Cc$ to the natural morphism $S \to \iota_0 \Cc$.
    \end{itemize}
    More generally, a \emph{stacky cone} is a toric decorated poset which is equivalent to $\Cone(\Cc)$ for some $\Cc$.
    Let $\StkCone$ be the full subcategory of $\TDP$ with objects stacky cones.
\end{dfn}

Here are some key examples to keep in mind:

\begin{ex} \label{ex:classical_toric_cone}
    Let $N$ be a finitely generated free abelian group, and let $\ul{\sigma} \subset N \otimes_\ZZ \RR$ be a strongly convex rational polyhedral cone.
    From this data we obtain a toric category $N \vec{/} (\ul{\sigma} \cap N)$.
    The corresponding stacky cone $K_{\ul{\sigma}} = \Cone(N \vec{/} (\ul{\sigma} \cap N))$ is the toric decorated poset $(N, |K_{\ul{\sigma}}|, P)$ where:
    \begin{itemize}
        \item The poset $|K_{\ul{\sigma}}|$ is the poset of faces of $\ul{\sigma}$ and
        \item The functor $P$ sends a face $\ul{\sigma}' \subset \ul{\sigma}$ to the natural morphism $\ul{\sigma}' \cap N \to N$ and sends inclusions of faces to the obvious commutative triangles over $N$.
    \end{itemize}
    This allows us to view the cones of classical toric geometry as stacky cones.
\end{ex}

\begin{ex} \label{ex:Gm_cone}
    As a sub-example of \zcref{ex:classical_toric_cone}, consider $N = \ZZ$ and the cone $\ul{\sigma} = 0$.
    The underlying poset $|K_{\ul{\sigma}}|$ is $\{ \sigma_0 \}$, and the functor $P$ sends $\sigma_0$ to the inclusion $0 \to \ZZ$.
    This is the stacky cone corresponding to the toric variety $\GG_m$.
\end{ex}

\begin{ex} \label{ex:A1_cone}
    As another sub-example of \zcref{ex:classical_toric_cone}, consider $N = \ZZ$ and the cone $\ul{\sigma} = \RR_{\geq 0}$.
    The underlying poset $|K_{\ul{\sigma}}|$ is $\{ \sigma_0 < \sigma_1 \}$, and the functor $P$ satisfies $P_{\sigma_0} = 0 \subset \ZZ$ and $P_{\sigma_1} = \NN \subset \ZZ$.  
    This is the stacky cone corresponding to the toric variety $\AA^1$.
    The minimal element $\sigma_0$ of the poset corresponds to the dense torus $\GG_m$ while the maximal element $\sigma_1$ corresponds to $0 \in \AA^1$.
\end{ex}

Here are some genuinely ``stacky'' examples:

\begin{ex} \label{ex:A1_mod_mu_d_cone}
    Let $\Cc = \ZZ \vec{/}_d \NN$ where the homomorphism $\NN \to \ZZ$ sends $1 \in \NN$ to $d \in \ZZ$ for some $d > 0$.
    Then $\Cone(\Cc)$ has:
    \begin{itemize}
        \item $\Nc = \ZZ$,
        \item $|\Cone(\Cc)| = \{ \sigma_0 < \sigma_1 \}$, and
        \item $P_{\sigma_0} = 0 \subset \ZZ$ and $P_{\sigma_1} = \NN \xrightarrow{d} \ZZ$.
    \end{itemize}
    This is the stacky cone corresponding to the toric stack $[\AA^1 / \mu_d]$.
    The minimal element $\sigma_0$ of the poset corresponds to the dense torus $\GG_m / \mu_d$, while $\sigma_1$ corresponds to the copy of $B\mu_d$ at the origin.
\end{ex}

\begin{ex} \label{ex:A1_mod_Gm_cone}
    Let $\Cc = \Bsf \NN$.
    Then $\Cone(\Cc)$ has:
    \begin{itemize}
        \item $\Nc = 0$,
        \item $|\Cone(\Cc)| = \{ \sigma_0 < \sigma_1 \}$, and
        \item $P_{\sigma_0} = 0$ and $P_{\sigma_1} = \NN$ (with the unique map $\ZZ \to 0$).
    \end{itemize}
    This is the stacky cone corresponding to $\AA^1 / \GG_m$.
    Here $\sigma_0$ corresponds to the trivial dense torus, while $\sigma_1$ corresponds to the copy of $B\GG_m$ located at the origin of $\AA^1 / \GG_m$.
\end{ex}

\subsection{Functoriality of stacky cones}

Note that $\Cone$ upgrades to a 2-functor $\Cone: \TCat \to \StkCone$.
This sends a symmetric monoidal functor $F: \Cc \to \Dc$ to the 1-morphism 
\[
    \Cone(F) = (\iota_0 F, |\Cone(F)|, \{\Cone(F)_{\#,\sigma}\}_{\sigma \in \Sigma})
\]
where, for $\sigma \in \Face(\PO \Cc)$:
\begin{itemize}
    \item $|\Cone(F)|(\sigma) = \angles{(\PO F) (P_{\sigma})} \in \Face(\PO \Dc)$, and 
    \item $\Cone(F)_{\#,\sigma}: P_{\Cone(\Cc),\sigma} \to P_{\Cone(\Dc),|\Cone(F)|(\sigma)}$ agrees with the natural morphism induced by $\PO F$.
\end{itemize}
(The behavior on 2-isomorphisms is obvious, so we will ignore it for brevity.)

\begin{ex}
    The projection functor $\ZZ \vec{/} \NN \to \Bsf \NN$ gives a morphism $p: \Cone(\ZZ \vec{/} \NN) \to \Cone(\Bsf \NN)$, where
    \begin{itemize}
        \item $p_{\#,\Nc}: \ZZ \to 0$ is the unique map,
        \item $|p|: \Face(\NN) \to \Face(\NN)$ is the identity, and
        \item The maps $p_{\#,\sigma_0}: 0 \to 0$ and $p_{\#,\sigma_1}: \NN \to \NN$ are identity maps.
    \end{itemize}
    At the level of toric stacks, $p$ corresponds to the projection map $\AA^1 \to \AA^1 / \GG_m$.
\end{ex}

There is a natural functor from $\TDP$ to $\TCat$ which behaves somewhat like ``global sections of the structure sheaf'' does in scheme theory (but with the variance reversed):

\begin{dfn}
    Let $\Oc: \TDP \to \TCat$ be the (2-)functor given by
    \[
        \Oc(\Sigma) = \colim_{\sigma \in \Sigma} \Nc_\Sigma \vec{/} P_{\Sigma,\sigma}
    \]
    This is well-defined by \zcref{prop:Tcat_finite_lim_colim}.
\end{dfn}

\begin{ex} \label{ex:global_sections_of_cone}
    Let $\Cc$ be a toric category.
    The poset $|\Cone(\Cc)|$ has a final object (namely, the face $\PO \Cc$), so
    \[
        \Oc(\Cone(\Cc)) = \Nc_{\Cone(\Cc)} \vec{/} P_{\Cone(\Cc),\PO \Cc} = \iota_0 \Cc \vec{/} \PO \Cc = \Cc
    \]
    by \zcref{thm:TCat_structure_thm}.
\end{ex}

It turns out that $\Oc$ is the left adjoint to $\Cone$:

\begin{prop} \label{prop:Cone_Gamma_adjunction}
    Let $\Sigma$ be a toric decorated poset, and let $\Cc$ be a toric category.
    Then there is a natural equivalence
    \[
        \Map_{\TDP}(\Sigma, \Cone(\Cc)) = \iota_0 \Fun^\oplus(\Oc(\Sigma), \Cc).
    \]
\end{prop}

\begin{proof}
    Functoriality of $\Oc$ gives a natural map
    \[
        \Map_{\TDP}(\Sigma, \Cone(\Cc)) \to \iota_0 \Fun^\oplus\big(\Oc(\Sigma), \Oc(\Cone(\Cc))\big)
    \]
    By \zcref{ex:global_sections_of_cone}, we may write $\Oc(\Cone(\Cc)) = \Cc$.
    Write 
    \[
    \Phi: \Map_{\TDP}(\Sigma, \Cone(\Cc)) \to \iota_0 \Fun^\oplus\big(\Oc(\Sigma), \Cc)
    \]
    for the induced map.
    
    We need to construct an inverse $\Psi$ to $\Phi$.
    Given $F: \Oc(\Sigma) \to \Cc$, define $\Psi(F): \Sigma \to \Cone(\Cc)$ as follows.
    For $\sigma \in |\Sigma|$, we let
    \begin{itemize}
        \item $F_\sigma$ be the induced symmetric monoidal functor $\Nc_\Sigma \vec{/} P_\sigma \to \Oc(\Sigma) \xrightarrow{F} \Cc = \iota_0 \Cc \vec{/} \PO \Cc$,
        \item $F_\Nc: \Nc_\Sigma \to \iota_0 \Cc = \Nc_{\Cone(\Cc)}$ be the functor induced by any $F_\sigma$ (all such functors are naturally isomorphic),
        \item $|\Psi(F)|(\sigma)$ be $\angles{\im \PO F_\sigma} \in \Face(\PO \Cc)$, and
        \item $\Psi(F)_{\#,\sigma}: P_\sigma \to \angles{\im \PO F_\sigma}$ be the symmetric monoidal functor induced by $F_\sigma$.
    \end{itemize}

    To see that $\Phi \Psi = \id$, suppose we are given $F: \Oc(\Sigma) \to \Cc$.
    Then $\Phi \Psi (F)$ is the symmetric monoidal functor $\Oc(\Sigma) \to \Cc$ induced by the universal property of colimits from the maps $F_\sigma$.
    But the $F_\sigma$ are themselves defined as the restrictions of $F$ to the toric categories $\Nc_\Sigma \vec{/} P_\sigma$, so the universal property of colimits tells us that $\Phi \Psi (F) = F$ naturally.

    To see that $\Psi \Phi = \id$, suppose we are given $f: \Sigma \to \Cone(\Cc)$.
    For any $\sigma \in |\Sigma|$, we can see that:
    \begin{itemize}
        \item $\Phi(f)_\sigma$ agrees with the map $\Nc_\Sigma \vec{/} P_\sigma \to \Cc$ induced by $f_{\#,\sigma}$.
        \item $(\Psi \Phi(f))_\Nc: \Nc_\Sigma \to \Nc_{\Cone(\Cc)}$ agrees with $f_\Nc$ by the preceding statement.
        \item $(\Psi \Phi (f))_{\#,\sigma}: P_\sigma \to \angles{\im \PO f_{\#,\sigma}}$ agrees with $f_{\#,\sigma}$.
        \item $|\Psi \Phi(f)|(\sigma) = \angles{\im \PO \Phi(f)_\sigma} = \angles{\im \PO f_{\#,\sigma}} = |f|(\sigma)$.
    \end{itemize}
    Thus there is a natural equivalence $\Psi \Phi = \id$, and we win.
\end{proof}

\begin{cor} \label{cor:stacky_cone_equivalence}
    The functor $\Cone$ defines an equivalence $\TCat \simto \StkCone$ with quasi-inverse $\Oc$.
\end{cor}

\begin{proof}
    It is obvious that $\Cone$ is essentially surjective onto $\StkCone$.
    To see that $\Cone$ is fully faithful, note that for any $\Cc, \Dc \in \TCat$, we have
    \begin{align*}
        \iota_0 \Fun^\oplus(\Cc, \Dc) &= \iota_0 \Fun^\oplus\big(\Oc(\Cone(\Cc)), \Dc\big) \textrm{ by \zcref{ex:global_sections_of_cone}} \\
        &= \Map_\TDP(\Cone(\Cc), \Cone(\Dc)) \textrm{ by \zcref{prop:Cone_Gamma_adjunction}}.
    \end{align*}
    Thus $\Cone$ is an equivalence.
    The functor $\Oc$ must be its quasi-inverse because any adjoint to an equivalence is a quasi-inverse of that equivalence.
\end{proof}

\subsection{Subcones}

We can try to understand general toric decorated posets by looking at their subcones.
More precisely:

\begin{dfn}
    A \emph{subcone} of a toric decorated poset is a downward closed subset which is also a stacky cone (with the induced structure).
\end{dfn}

The subcones of a stacky cone are particularly easy to understand:

\begin{prop} \label{prop:subcones_of_cone}
    Let $K$ be a stacky cone.
    Then:
    \begin{enumerate}
        \item The subcones of $K$ are precisely the subsets $K_{\leq \kappa}$ for $\kappa \in K$.
        \item The intersection of any (necessarily finite) collection of subcones of $K$ is itself a subcone of $K$.
    \end{enumerate}
\end{prop}

\begin{proof}
    (1). The poset underlying any subcone of $K$ must have a unique maximal element.
    If this maximal element is $\kappa$, then the fact that the poset is downward closed in $K$ lets us identify the subcone in question with $K_{\leq \kappa}$.

    It remains to show that every $K_{\leq \kappa}$ is a subcone of $K$.
    If $K = \Cone(\Cc)$, then $\kappa$ corresponds to some $P \in \Face(\PO \Cc)$, and it is immediate from the definitions that $K_{\leq \kappa} = \Cone(\iota_0 \Cc \vec{/} P)$.
    Thus $K_{\leq \kappa}$ is in fact a subcone of $K$.

    (2). By induction, it suffices to show that the intersection of two subcones of $K$ is itself a subcone of $K$.
    So let $\kappa_1, \kappa_2 \in K$, and look at the intersection $K_{\leq \kappa_1} \cap K_{\leq \kappa_2}$.
    If $K = \Cone(\Cc)$, then each $\kappa_i$ corresponds to some face $P_i$ of $\PO \Cc$.
    The intersection $P_1 \cap P_2$ is itself a face of $\PO \Cc$ and thus corresponds to some $\kappa_3 \in K$.
    We see that $K_{\leq \kappa_1} \cap K_{\leq \kappa_2} = K_{\leq \kappa_3}$, which is a subcone of $K$ as needed.
\end{proof}

\subsection{Definition of stacky fans}

\begin{dfn} \label{dfn:stacky_fan}
    A \emph{stacky fan} is a toric decorated poset $\Sigma = (\Nc, |\Sigma|, P)$ such that:
    \begin{itemize}
        \item For every $\sigma \in |\Sigma|$, $\Sigma_{\leq \sigma}$ is a subcone of $\Sigma$, and
        \item For every (necessarily finite) collection of subcones of $\Sigma$, the intersection is itself a subcone.
    \end{itemize}
    Let $\StkFan$, the \emph{2-category of stacky fans}, denote the full subcategory of $\TDP$ spanned by the stacky fans.
\end{dfn}

Before giving examples of stacky fans, we should make it easier to check that a given toric decorated poset is a stacky fan.
For this, we will first introduce an \emph{a priori} weaker condition:

\begin{dfn}
    Let $\Sigma$ be a toric decorated poset.
    A collection of downward closed subsets $\{ K_i \subset \Sigma \}_{i \in I}$ is a \emph{cover} of $\Sigma$ if $|\Sigma| = \cup_{i \in I} |K_i|$.
    We say that a cover $\{ K_i \}_{i \in I}$ of $\Sigma$ satisfies the \emph{fan condition} if:
    \begin{itemize}
        \item Every $K_i$ is a subcone of $\Sigma$ and
        \item For all $i_1, i_2 \in I$, we have $K_{i_1} \cap K_{i_2} = K_{i_3}$ for some $i_3 \in I$.
    \end{itemize}
\end{dfn}

\begin{prop} \label{prop:fan_criterion}
    Let $\Sigma$ be a toric decorated poset.
    The following are equivalent:
    \begin{enumerate}
        \item $\Sigma$ is a stacky fan.
        \item $\Sigma$ admits a cover consisting of subcones, and $|\Sigma|$ admits all finite meets.
        \item The cover $\{ \Sigma_{\leq \sigma} \}_{\sigma \in \Sigma}$ of $\Sigma$ satisfies the fan condition.
        \item There exists a cover of $\Sigma$ satisfying the fan condition.
    \end{enumerate}
\end{prop}

\begin{proof}
    ($1 \Rightarrow 2$).
    Suppose $\Sigma$ is a stacky fan.
    It is clear that $\Sigma$ admits a cover consisting of subcones.
    If $\{ \sigma_i \}_{i \in I}$ is a (necessarily finite) collection of elements of $\Sigma$, the meet $\lim_i \sigma_i$ is the unique maximal element of $\cap_i \Sigma_{\leq \sigma_i}$, if it exists.
    We know all $\Sigma_{\leq \sigma_i}$ are stacky cones, so the intersection $\cap_i \Sigma_{\leq \sigma_i}$ must be a stacky cone.
    This stacky cone must have a unique maximal element, so $\Sigma$ must have meets.

    ($2 \Rightarrow 3$).
    Suppose $|\Sigma|$ admits all finite meets, and let $\{ K_i \}_{i \in I}$ be a cover of $\Sigma$ consisting of subcones.
    To show that $\{ \Sigma_{\leq \sigma} \}_{\sigma \in \Sigma}$ satisfies the fan condition, first let $\sigma \in \Sigma$.
    Then $\sigma \in K_i$ for some $i \in I$, and \zcref{prop:subcones_of_cone}(1) shows that the subset $\Sigma_{\leq \sigma} \subset K_i$ is a subcone of $K_i$, hence of $\Sigma$.

    To check that the condition on intersections holds, let $\sigma_1, \sigma_2 \in \Sigma$.
    By hypothesis, the meet $\sigma_{i_1} \wedge \sigma_{i_2}$ exists, so 
    \[
        \Sigma_{\leq \sigma_{i_1}} \cap \Sigma_{\leq \sigma_{i_2}} = \Sigma_{\leq \sigma_1 \wedge \sigma_2}
    \]
    is also an element of our cover, as needed.

    ($3 \Rightarrow 4$).
    Obvious.

    ($4 \Rightarrow 1$).
    Suppose $\{ K_i \}_{i \in I}$ is a cover of $\Sigma$ which satisfies the fan condition.
    Let $\sigma \in \Sigma$.
    Then $\sigma \in K_i$ for some $i$.
    Because $K_i$ is downward closed, we have $\Sigma_{\leq \sigma} \subset K_i$.
    By \zcref{prop:subcones_of_cone}(1), we see that $\Sigma_{\leq \sigma}$ is itself a stacky cone, hence is a subcone of $\Sigma$.

    We still need to show that the intersection of any finite collection of subcones of $\Sigma$ is itself a subcone.
    By induction, it suffices to show that if $K'_1$ and $K'_2$ are subcones of $\Sigma$, then $K'_1 \cap K'_2$ is a subcone of $\Sigma$.
    Note that $K'_1$ has a unique maximal element $\kappa'_1$, and $\kappa'_1 \in K_{i_1}$ for some $i_1 \in I$.
    Because $K_{i_1}$ is downward closed, we must have $K'_1 \subset K_{i_1}$.
    Likewise we have $K'_2 \subset K_{i_2}$ for some $i_2 \in I$.
    By the fan condition, we have $K_{i_1} \cap K_{i_2} = K_{i_3}$ for some $i_3 \in I$.
    In particular we have $K'_1 \cap K'_2 \subset K_{i_3}$, so 
    \[
        K'_1 \cap K'_2 = (K'_1 \cap K_{i_3}) \cap (K'_2 \cap K_{i_3}).
    \]
    Note that $K'_1 \cap K_{i_3}$ is a subcone of $K_{i_1}$ by \zcref{prop:subcones_of_cone}(2).
    Likewise $K'_2 \cap K_{i_3}$ is a subcone of $K_{i_2}$.
    Both of these are actually subcones of $K_{i_3}$, so $(K'_1 \cap K_{i_3}) \cap (K'_2 \cap K_{i_3})$ is a subcone of $K_{i_3}$ by \zcref{prop:subcones_of_cone}(2) again.
    Thus $K'_1 \cap K'_2$ is a subcone of $K_{i_3}$ (hence also of $\Sigma$), as needed.
\end{proof}

Here are a few key examples to keep in mind.
More will be discussed in \zcref{sub:comparison_of_stacky_fans}:

\begin{ex}
    Any stacky cone $K$ is itself a stacky fan.
    Indeed, the cover $\{ K \}$ of $K$ satisfies the fan condition.
\end{ex}

\begin{ex} \label{ex:classical_toric_fan}
    Let $N$ be a finitely generated free abelian group, and let $\ul{\Sigma} \subset N \otimes_\ZZ \RR$ be a fan (as in \cite[\S 3.1]{cls}).
    Then we may define a stacky fan $\Sigma = (N, |\Sigma|, P)$ as follows:
    \begin{itemize}
        \item $|\Sigma|$ is the poset of cones in $\ul{\Sigma}$.
        \item For all cones $\ul{\sigma} \in \ul{\Sigma}$, we have $P_{\ul{\sigma}} = (\ul{\sigma} \cap N)$.
        (Inclusions of cones are sent to the obvious morphisms.)
    \end{itemize}
    We may view $\Sigma$ as a union of the cones $K_{\ul{\sigma}}$ (see \zcref{ex:classical_toric_cone}) for $\ul{\sigma} \in \ul{\Sigma}$.
    Any two subcones $K_{\ul{\sigma_1}}, K_{\ul{\sigma_2}} \subset \Sigma$ intersect in the subcone $K_{\ul{\sigma_1} \cap \ul{\sigma_2}}$, so $\Sigma$ is in fact a stacky fan.
    This allows us to view the fans of classical toric geometry as stacky fans.
\end{ex}

\begin{ex} \label{ex:P1_fan}
    As a sub-example of \zcref{ex:classical_toric_fan}, let $\ul{\Sigma}$ be the fan of $\PP^1$, consisting of $\big\{ \{ 0 \}, \RR_{\geq 0}, \RR_{\leq 0} \big\} \subset \RR$.
    The corresponding stacky fan $\Sigma$ satisfies:
    \begin{itemize}
        \item $\Nc = \ZZ$.
        \item $|\Sigma|$ is the poset
        \[
            \sigma_{\leq 0} \leftarrow \sigma_0 \rightarrow \sigma_{\geq 0}.
        \]
        \item The functor $P$ sends the above poset to
        \[
            \{-\NN \subset \ZZ\} \leftarrow \{ 0 \subset \ZZ \} \rightarrow \{ \NN \subset \ZZ \}.
        \]
    \end{itemize}
    Note that $\Sigma$ is the union of two copies of the stacky cone for $\AA^1$ (\zcref{ex:A1_cone}) along a common copy of the stacky cone for $\GG_m$ (\zcref{ex:Gm_cone}), as expected.
\end{ex}

\begin{ex}
    We can generalize the preceding example to give stacky fans for weighted projective lines $\PP(d_0, d_1)$, where $d_0$ and $d_1$ are positive natural numbers.
    Here:
    \begin{itemize}
        \item $\Nc = N = \coker(\ZZ \xrightarrow{\alpha} \ZZ^2)$, where $\alpha$ is given by the matrix
        \[
            \begin{bmatrix}
                d_0 \\
                d_1
            \end{bmatrix}.
        \]
        Choosing a splitting of the natural map $N \to N / N_\tors$, we can write $N \simeq \ZZ \oplus \ZZ / \gcd(d_0, d_1)$.
        \item $|\Sigma|$ is the poset
        \[
            \sigma_{\leq 0} \leftarrow \sigma_0 \rightarrow \sigma_{\geq 0}.
        \]
        \item The functor $S$ sends the above poset to
        \[
            \{\NN \xrightarrow{1 \mapsto (-d_1, 0)} N\} \leftarrow \{ 0 \mapsto \Nc \} \rightarrow \{\NN \xrightarrow{1 \mapsto (d_0, 0)} N\}.
        \]
    \end{itemize}
    If $d_0$ and $d_1$ are coprime, then $N \simeq \ZZ$, and we can view $\Sigma$ as a union of the stacky fans of $\AA^1 / \mu_{d_0}$ and $\AA^1 / \mu_{d_1}$ (see \zcref{ex:A1_mod_mu_d_cone}).
    When this is not the case, the presence of torsion in $N$ indicates that the weighted projective line $\PP(d_0, d_1)$ has nontrivial generic stabilizer.
    We will discuss this more generally in \zcref{cor:generic_stabilizer_criteria} below.
\end{ex}

\subsection{Comparison with other definitions of stacky fans} \label{sub:comparison_of_stacky_fans}

There are many definitions of ``stacky fan'' in the literature.
In this subsection we will discuss the definitions we know of and show that they all give stacky fans in our sense.

\begin{ex} \label{ex:BCS_fan}
    Let $(N, \ul{\Sigma}, \beta: \ZZ^n \to N)$ be a stacky fan in the sense of \cite{borisov2005orbifold},\footnote{The stacky fans of \cite{iwanari2009category} comprise the special case where $N$ is free.} i.e.\
    \begin{itemize}
        \item $N$ is a finitely generated abelian group,
        \item $\ul{\Sigma}$ is a (classical) simplicial fan in $N \otimes \RR$ with rays $\rho_1, \dots, \rho_n$, and
        \item $\beta: \ZZ^n \to N$ is a homomorphism such that, for all $i$, we have $\rho_i = \RR_{\geq 0} \cdot \beta(\ebf_i)$.
    \end{itemize}
    From this data, we can define a stacky fan $\Sigma$ (in our sense) with:
    \begin{itemize}
        \item $\Nc_\Sigma = N$,
        \item $|\Sigma|$ given by the set of cones in $\ul{\Sigma}$, and
        \item For $\ul{\sigma} \in \ul{\Sigma}$ with rays $\rho_{i_1}, \dots, \rho_{i_k}$, $P_{\ul{\sigma}} = \NN \ebf_{i_1} \oplus \NN \ebf_{i_k}$, mapping to $N$ via $\beta$.
    \end{itemize}
\end{ex}

\begin{ex} \label{ex:KM_fan}
    Let $(N, \{ P_\sigma \}_{\sigma \in |\Sigma|})$ be a KM fan in the sense of \cite[Definition 2.3.1]{gillam2015theory},\footnote{See the discussion immediately following \cite[Definition 2.3.1]{gillam2015theory} for a discussion of how to pass between this definition of KM fan and an alternative definition that generalizes the stacky fans of \cite{tyomkin2012tropical}.} i.e.:
    \begin{itemize}
        \item $N$ is a finitely generated abelian group and
        \item $\{ P_\sigma \}_{\sigma \in |\Sigma|}$ is a \emph{set} of sharp toric submonoids of $N$,\footnote{By \zcref{prop:sharp_toric_characterization}, $P_\sigma$ is sharp and toric if and only if $P_\sigma$ is sharp and fs, which is what \cite[Definition 2.3.1]{gillam2015theory} requires.}
    \end{itemize}
    such that:
    \begin{itemize}
        \item Every face of each $P_\sigma$ is contained in $\{ P_\sigma \}$, and
        \item Each pair $P_{\sigma_1}$ and $P_{\sigma_2}$ intersects in a common face.
    \end{itemize}
    (If $N$ is free, we say that this KM fan is a ``lattice KM fan.'')
    
    From this data, we can define a stacky fan $\Sigma = (N, |\Sigma|, P)$ (in our sense) where $P$ sends $\sigma \in |\Sigma|$ to the inclusion map $P_\sigma \hookrightarrow N$.
\end{ex}

\begin{ex} \label{ex:GS_fan}
    Let $(L, \ul{\Sigma}, \beta)$ be a non-strict stacky fan in the sense of \cite[Definition 2.15]{geraschenko2015toric_1}, i.e.\
    \begin{itemize}
        \item $L$ is a finitely generated free abelian group,
        \item $N$ is a finitely generated abelian group,
        \item $\ul{\Sigma}$ is a classical fan in $L \otimes \RR$, and
        \item $\beta: L \to N$ is a homomorphism.
    \end{itemize}
    From this data, we can define a stacky fan $\Sigma$ (in our sense) where:
    \begin{itemize}
        \item $\Nc_\Sigma = \Omega^\infty \cofib(\RHom(H^0_\beta, \ZZ) \oplus \RHom(H^1_\beta, \ZZ) \xrightarrow{(0, - \circ f)} L)$ where:
        \begin{itemize}
            \item For $i = 0, 1$, $H^i_\beta = \Rbf^i \Hom(\cofib(\beta), \ZZ)$ and
            \item $f$ is the natural map $\RHom(H^1_\beta, \ZZ) = \Hom(H^1_\beta, \ZZ) \to \Hom(\Hom(L, \ZZ), \ZZ) = L$,
        \end{itemize}
        \item $|\Sigma|$ is the set of cones in $\ul{\Sigma}$, and
        \item For $\ul{\sigma} \in \ul{\Sigma}$, $P_{\ul{\sigma}} = \ul{\sigma} \cap L$, mapping to $\Nc_\Sigma$ via the composite of the natural maps $\ul{\sigma} \cap L \to L \to \Nc_\Sigma$.
    \end{itemize}
\end{ex}

\subsection{Topology of stacky fans}

Let $\Sigma$ be a stacky fan.
We can think of the underlying finite poset $|\Sigma|$ as a topological space, where the open subsets are the downward closed subsets.
In this section we will discuss the topology and combinatorics of $|\Sigma|$.

First observe that $|\Sigma|$ has a unique minimal element.
Indeed, this is true for any subcone $K$ of $\Sigma$ (since the face poset of any monoid $M$ has a unique minimal element $M^\pm$).
Furthermore, since subcones of $\Sigma$ intersect in subcones, the minimal element of any two subcones must also be the minimal element of their intersection, giving uniqueness.
The analogous statement for maximal elements is false: as we saw in \zcref{ex:P1_fan}, a typical stacky fan has many maximal elements.

By \zcref{prop:subcones_of_cone}, the subcones of a stacky cone $K$ are in bijective correspondence with the points of $K$.
We can extend this to stacky fans:

\begin{prop} \label{prop:pts_and_cones_stacky_fan}
    Let $\Sigma$ be a stacky fan.
    Then the assignment $\sigma \mapsto \Sigma_{\leq \sigma}$ gives a bijection between points of $\Sigma$ and subcones of $\Sigma$.
    In particular, $\Sigma$ is a stacky cone if and only if $|\Sigma|$ has a unique maximal element.
\end{prop}

\begin{proof}
    Let $\sigma \in \Sigma$.
    Then $\Sigma_{\leq \sigma}$ is a subcone of $\Sigma$ by definition, and the assumption that $|\Sigma|$ is a poset implies the assignment $\sigma \mapsto \Sigma_{\leq \sigma}$ is injective.
    To see that this assignment is surjective, let $K$ be a subcone of $\Sigma$.
    Then $K$ has a maximal element $\kappa$, and $K = \Sigma_{\leq \kappa}$.

    For the ``in particular'' claim, we already noted the ``only if'' direction in \zcref{prop:subcones_of_cone}.
    For the ``if'' direction, suppose $\sigma \in |\Sigma|$ is the unique maximal element.
    Then $\Sigma = \Sigma_{\leq \sigma}$, hence is a stacky cone.
\end{proof}

Thus subcones of a stacky fan are the downward-closed subsets with unique maximal elements.
More generally, any downward closed subset of a stacky fan is itself a stacky fan:

\begin{prop} \label{prop:downward_closed_is_subfan}
    Let $\Sigma$ be a fan.
    Then any downward-closed subset $\Sigma' \subset \Sigma$ admits a natural stacky fan structure, and the inclusion map $\Sigma \hookrightarrow \Sigma'$ is a morphism of stacky fans.
\end{prop}

\begin{proof}
    We view $\Sigma'$ as a stacky fan with $\Nc_{\Sigma'} = \Nc_\Sigma$ and $P_{\Sigma'} = P_\Sigma|_{\Sigma'}$.
    To see that this is a stacky fan, we apply \zcref{prop:fan_criterion} to the cover $\{ \Sigma_{\leq \sigma} \}_{\sigma \in \Sigma'}$ of $\Sigma'$.
    We need to check:
    \begin{itemize}
        \item For all $\sigma \in \Sigma'$, $\Sigma_{\leq \sigma}$ is a subcone of $\Sigma'$.
        Because $\Sigma$ is a stacky fan, we see that $\Sigma_{\leq \sigma}$ is a subcone of $\Sigma$ which is contained in $\Sigma'$, so this claim is clear.
        \item For all $\sigma_1, \sigma_2 \in \Sigma'$, the intersection $\Sigma_{\leq \sigma_1} \cap \Sigma_{\leq \sigma_2}$ is a subcone of $\Sigma'$.
        This again follows from the assumption that $\Sigma$ is a stacky fan. \qedhere
    \end{itemize}
\end{proof}

\begin{dfn}
    A \emph{subfan} of a stacky fan is a downward closed subset, viewed as a stacky fan with the induced structure.
    An \emph{open immersion of stacky fans} is a morphism of stacky fans which is isomorphic to the inclusion of a subfan.
\end{dfn}

\begin{rmk} \label{rmk:Zariski_on_StkFan}
    We can define a \emph{Zariski topology} on (the homotopy category of) $\StkFan$ by saying that covers of a stacky fan $\Sigma$ are families of morphisms $\{ j_i: \Sigma_i \to \Sigma \}_{i \in I}$ such that each $j_i$ is an open immersion and such that $\Sigma = \cup_{i \in I} j_i(\Sigma_i)$.
\end{rmk}

In the classical theory of fans, it is often useful to consider the collection of all subcones of some specified dimension in a given fan.
The analogue in the theory of stacky fans is straightforward:

\begin{dfn}
    For a stacky fan $\Sigma$ and $n \in \NN$, let 
    \[
        \Sigma(n) = \bset{\sigma \in |\Sigma|}{\rk P_\sigma^\gp = n}.
    \]
    Also let $\Sigma(\leq n) = \cup_{i \leq n} \Sigma(i)$, so $\Sigma(\leq n)$ is a subfan of $\Sigma$. 
\end{dfn}

Note that $\Sigma(0) = \Sigma(\leq 0)$ is the subcone consisting of the unique minimal element of $|\Sigma|$.
The cone $\Sigma(0)$ encodes the dense torus of the toric stack corresponding to $\Sigma$, as we shall see in \zcref{prop:geometrization_of_cone}.

\subsection{Morphisms of stacky fans}

Recall that morphisms of schemes are essentially obtained by gluing together morphisms of affine schemes.
An analogous claim holds for stacky fans:

\begin{prop} \label{prop:lim_colim_maps_in_StkFan}
    Let $\Sigma$ and $\Sigma'$ be stacky fans.
    Then 
    \[
        \Map_\StkFan(\Sigma, \Sigma') = \lim_{K \subset \Sigma} \colim_{K' \subset \Sigma'} \Map_\StkCone(K, K'),
    \]
    where the limit is taken over subcones of $\Sigma$ and $\Sigma'$.
\end{prop}

\begin{proof}
    To define a morphism of stacky fans $\Sigma \to \Sigma'$, it suffices to define morphisms of stacky fans $K \to \Sigma'$ for all subcones $K \subset \Sigma'$ (in fact, for all subcones in a cover of $\Sigma$) such that these morphisms agree up to natural isomorphism on intersections.
    This is the same as saying
    \[
        \Map_\StkFan(\Sigma, \Sigma') = \lim_{K \subset \Sigma} \Map_\StkCone(K, \Sigma').
    \]
    
    We thus reduce to studying the case where $\Sigma = K$ is a stacky cone.
    Let $\kappa_{\max}$ be the maximal element of $K$.
    Then, for any $f: \Sigma \to \Sigma'$, we must have $f(K) \subset \Sigma'_{\leq f(\kappa_{\max})}$.
    Every $\Sigma'_{\leq f(\kappa_{\max})}$ is a subcone of $\Sigma'$, so $\Map_\StkFan(K, \Sigma')$ is a union of the spaces of maps from $K$ to subcones of $\Sigma'$.
    In other words,
    \[
        \Map_\StkFan(K, \Sigma') = \colim_{K' \subset \Sigma'} \Map_\StkCone(K, K').
    \]
    Combining this with the results of the previous paragraph gives the claim.
\end{proof}

\section{Stacky tori and toric stacks} \label{sec:defining_toric_stacks}

We now seek to describe the 2-category $\TStk$ of toric stacks.
These will be algebraic stacks defined on $\Aff_k$, the \'etale site of affine $k$-schemes.\footnote{Standard descent arguments allow us to also view these as algebraic stacks on the other standard sites over $\Spec k$ arising in algebraic geometry, e.g.\ the \'etale site of all $k$-schemes.}

\subsection{Shifted Cartier duality and stacky tori}

For a sheaf of abelian groups $\Asc$ on $\Aff_k$, the \emph{Cartier dual} of $\Asc$ is the abelian group scheme $\CDual(\Asc) = \ul{\Hom}(\Asc, \GG_m)$, where the multiplication on $\CDual(\Asc)$ arises from the multiplication on $\GG_m$ and the diagonal map on $\Asc$.
Because $k$ is algebraically closed, Cartier duality gives an antiequivalence between the category of diagonalizable groups over $k$ and the category of finitely generated abelian groups (viewed as constant sheaves on $\Aff_k$).
More precisely, for any diagonalizable $G$, we have $\CDual(G) = \ul{\weight(G)}$ and $\CDual(\ul{\weight(G)}) = G$.

Let us recall how this story generalizes to the world of stacks.

\begin{dfn} \label{dfn:2ab_stk}
    An \emph{abelian 2-group stack} is a stack (for the large \'etale site of $\Spec k$) valued in symmetric monoidal groupoids.
    These assemble into a 2-category $\tAb(\Stk)$, defined as the full subcategory of $\Fun(\Aff_k\op, \tAb)$ spanned by objects whose underlying functor $\Aff_k\op \to \Gpd$ is a stack.
    
    One can define actions of abelian 2-group stacks on other stacks as in \cite[Definition B.12]{fantechi2010smooth}.
\end{dfn}

\begin{rmk}
    As with abelian 2-groups, one should keep in mind that homomorphisms of abelian 2-group stacks are not the same as homomorphisms of the underlying ``2-group stacks:'' not every monoidal functor is symmetric monoidal.
\end{rmk}

\begin{dfn}[{\cite[Definition 3.1]{brochard2021duality}}]
    Let $\Ac$ be an abelian 2-group stack.
    The \emph{shifted Cartier dual} of $\Ac$ is the abelian 2-group stack $\SCDual(\Ac)$ with $R$-points
    \[
        \SCDual(\Ac)(R) = \Hom_{\tAb(\Stk)}(\Ac, B\GG_m).
    \]
    This defines a 2-functor $\SCDual: \tAb(\Stk)\op \to \tAb(\Stk)$. 
\end{dfn}

\begin{ex} \label{ex:shifted_Cartier_dual_of_A}
    Let $A$ be a finitely generated abelian group.
    Then \cite[Example 3.7]{brochard2021duality} gives $\SCDual(\ul{A}) = B\CDual(\ul{A})$, the classifying stack of the Cartier dual of $\ul{A}$.
\end{ex}

\begin{ex} \label{ex:shifted_Cartier_dual_of_BA}
    Let $A$ be a finitely generated abelian group, and note that $\ul{\Bsf A} = B\ul{A}$.\footnote{That is, the stackification of the constant prestack $R \mapsto \Bsf A$ is the classifying stack of the constant group sheaf $\ul{A}$.}
    Then \cite[Corollary 3.5]{brochard2021duality} gives $\SCDual(B\ul{A}) = \CDual(\ul{A})$.
\end{ex}

\begin{ex} \label{ex:shifted_Cartier_dual_of_2term}
    We can combine the preceding examples as follows.
    Let $\phi: A \to A'$ be a homomorphism of finitely generated abelian groups.
    Then
    \begin{align*}
        \SCDual(\ul{\cofib(\phi)}) &= \Hom(\ul{\cofib(\phi)}, B\GG_m) \\
        &= \fib(\SCDual(\ul{A'}) \to \SCDual(\ul{A})) \\
        &= \fib(B\CDual(\ul{A'}) \to B\CDual(\ul{A})) \\
        &= \CDual(\ul{A}) / \CDual(\ul{A'}).
    \end{align*}
\end{ex}

\begin{dfn} \label{dfn:stacky_torus}
    A \emph{stacky torus} is an abelian 2-group stack of the form $\SCDual(\ul{\Ac})$ where $\Ac$ is a toric groupoid.
\end{dfn}

There are of course more explicit (but less canonical) perspectives on the definition of stacky tori:
\begin{itemize}
    \item By \zcref{rmk:2term_vs_DZ} and \zcref{prop:deligne_classification}, every toric groupoid splits (non-canonically) as a direct sum $A_0 \oplus \Bsf A_1$, where $A_0$ is finitely generated and $A_1$ is finitely generated free.
    Thus every stacky torus is of the form $\SCDual(\ul{A_0}) \times \SCDual(B \ul{A_1}) = B\CDual(\ul{A_0}) \times \CDual(\ul{A_1})$ with $A_0$ finitely generated and $A_1$ finitely generated free.
    In other words, every stacky torus splits (non-canonically) as a product $BG \times T$ with $G$ diagonalizable and $T$ a torus.
    \item The preceding claim implies that every stacky torus can be written as a quotient $[T / G]$ with $T$ a torus, $G$ diagonalizable, and $G \to T$ a homomorphism.
    Conversely, every such quotient $[T / G]$ is a stacky torus.
\end{itemize}

We define stacky tori using toric groupoids largely because of the following result:

\begin{prop}
    The 2-functor $\SCDual(\ul{(-)}): \TGpd\op \to \tAb(\Stk)$ defines an antiequivalence between $\TGpd$ and the full subcategory of $\tAb(\Stk)$ spanned by the stacky tori.
\end{prop}

\begin{proof}
    Let $\Ac \in \TGpd$.
    By \cite[Theorem 4.11]{brochard2021duality}, we have $(\SCDual)^2(\ul{\Ac}) = \ul{\Ac}$.
    The functor sending a stacky torus $\Tfr$ to $\Gamma(\Spec k, \SCDual(\Tfr))$ thus defines an inverse to $\SCDual(\ul{(-)})$, since 
    \[
        \Gamma\big(\Spec k, (\SCDual)^2(\ul{\Ac})\big) = \Gamma(\Spec k, \ul{\Ac}) = \Ac. \qedhere
    \]
\end{proof}

\subsection{Definition and examples of toric stacks}

\begin{dfn} \label{dfn:toric_stack}
    A \emph{toric stack} (over $\Spec k$) is an Artin stack $\Xfr$ of finite type over $\Spec k$ together with:
    \begin{itemize}
        \item An action of a stacky torus $\Tfr$ on $\Xfr$ and
        \item A dense $\Tfr$-equivariant open immersion $j: \Tfr \hookrightarrow \Xfr$
    \end{itemize}
    such that:
    \begin{itemize}
        \item $\Xfr$ is normal,
        \item $\Xfr$ has affine diagonal, and
        \item All stabilizer groups of $\Xfr$ at geometric points are diagonalizable.
    \end{itemize}
    We refer to $\Tfr$ as the \emph{dense stacky torus} of $\Xfr$.
\end{dfn}

As in \cite{geraschenko2015toric_2}, the ``affine diagonal'' condition here serves as a stacky replacement for the usual separation assumption on toric varieties (which cannot hold for general Artin stacks).
A systematic discussion of the affine diagonal condition for schemes can be found in \cite{alonso2008derived}.
The following remarks will hopefully clarify some aspects of this condition.

\begin{rmk} \label{rmk:affine_diagonal_reasonable}
    The property of having affine diagonal is a ``reasonable property'' of morphisms in the sense of \cite[\S 8.1]{vakil2025rising},\footnote{Here we replace all references to Zariski covers in \textit{op.\ cit.\ } by references to smooth covers, since we are dealing with stacks.} i.e.\ the class of morphisms with affine diagonal is stable under composition, stable under base change, and smooth-local on the target (see e.g.\ \cite[Theorem 11.1.2]{vakil2025rising} for the standard argument).
    This implies that the class of morphisms with affine diagonal has many useful properties, e.g.\ there is a cancellation law for such morphisms as discussed in \cite[Theorem 11.1.1]{vakil2025rising}.
\end{rmk}

The next result gives a useful alternative condition for checking that an algebraic stack has affine diagonal.
The proof is essentially that of \cite[Proposition 2.4]{alonso2008derived} (with ``Zariski cover'' replaced by ``smooth cover''), so we omit the details.

\begin{prop} \label{prop:affine_diagonal_intersection_criterion}
    Let $\Xfr$ be an algebraic stack, and let $f: \Xfr \to \Spec R$ be a morphism to an affine scheme.
    The following are equivalent:
    \begin{enumerate}
        \item $f$ has affine diagonal.
        \item For all smooth covers $\{ U_i \to \Xfr \}_{i \in I}$ with all $U_i$ affine, all fiber products $U_i \times_{\Xfr} U_j$ are affine.
        \item There exists a smooth cover $\{ U_i \to \Xfr \}_{i \in I}$ with all $U_i$ affine and such that all fiber products $U_i \times_{\Xfr} U_j$ are affine.
    \end{enumerate}
\end{prop}

We now provide some key examples of toric stacks:

\begin{ex}
    Any toric variety is a toric stack.
\end{ex}

\begin{ex} \label{ex:quotient_of_toric_variety_is_toric_stack}
    Let $X$ be a toric variety with torus $T$, let $G$ be a diagonalizable group, and suppose we are given a homomorphism $G \to T$.
    Then $[X / G]$ is a toric stack with dense stacky torus $\Tfr = [T / G]$.
    The toric stacks that arise in this way are precisely the ``non-strict toric stacks'' considered in \cite{geraschenko2015toric_1}.
\end{ex}

\begin{ex}
    The affine line with doubled origin is a toric stack with dense torus $\GG_m$.
    However, the affine plane with doubled origin is not a toric stack, since it does not have affine diagonal.
\end{ex}

\subsection{A 2-category of toric stacks}

Defining a 2-category $\TStk$ of toric stacks is a bit involved.
Let $\Xfr$ and $\Xfr'$ be toric stacks with dense stacky tori $j: \Tfr \hookrightarrow \Xfr$ and $j': \Tfr' \hookrightarrow \Xfr'$.
For every homomorphism $\phi: \Tfr \to \Tfr'$, we get an action of $\Tfr$ on $\Xfr$, and we can thus make sense of the groupoid of $\Tfr$-equivariant maps $\Map^\Tfr(\Xfr, \Xfr') = \Map^{\Tfr,\phi}(\Xfr, \Xfr')$.
This defines a functor
\begin{align*}
    M_{\Xfr,\Xfr'}: \Hom_{\tAb(\Stk)}(\Tfr, \Tfr') &\to \Gpd \\
    \phi &\mapsto \Map^{\Tfr,\phi}(\Xfr, \Xfr').
\end{align*}
Applying the Grothendieck construction to $M_{\Xfr,\Xfr'}$ produces a groupoid $\Map_\TStk^\pre(\Xfr, \Xfr')$ with objects given by pairs $(\phi: \Tfr \to \Tfr', f: \Xfr \to \Xfr')$ where $f$ is $\Tfr$ equivariant for the action induced by $\phi$.
There is a natural map $\Map_\TStk^\pre(\Xfr, \Xfr') \to \Hom_{\tAb(\Stk)}(\Tfr, \Tfr')$ sending $(\phi, f)$ to $\phi$.

To get a reasonable definition of morphism of toric stack, we still need to impose compatibility with the dense open immersions $j, j'$.
Note that precomposition with $j$ defines a natural transformation $M_{\Xfr,\Xfr'} \to M_{\Tfr,\Xfr'}$, and postcomposition with $j'$ defines a natural transformation $M_{\Tfr,\Tfr'} \to M_{\Tfr,\Xfr'}$.
Applying the Grothendieck construction gives a diagram of groupoids
\[
\begin{tikzcd}
    & \Map_\TStk^\pre(\Tfr, \Tfr') \dar["j' \circ -"] \\
    \Map_\TStk^\pre(\Xfr, \Xfr') \rar["- \circ j"] & \Map_\TStk^\pre(\Tfr, \Xfr')
\end{tikzcd}
\]
over $\Hom_{\tAb(\Stk)}(\Tfr, \Tfr')$.
Every homomorphism of abelian 2-group stacks induces a natural equivariant map (sending the unit to the unit), so there is a natural functor $\Hom_{\tAb(\Stk)}(\Tfr, \Tfr') \to \Map_\TStk^\pre(\Tfr, \Tfr')$.
Precomposing the vertical arrow in our above diagram with this functor gives a diagram
\[
\begin{tikzcd}
    & \Hom_{\tAb(\Stk)}(\Tfr, \Tfr') \dar["j' \circ -"] \\
    \Map_\TStk^\pre(\Xfr, \Xfr') \rar["- \circ j"] & \Map_\TStk^\pre(\Tfr, \Xfr').
\end{tikzcd}
\]

\begin{dfn}
    Let $\Xfr$ and $\Xfr'$ be toric stacks as above.
    We define the \emph{category of morphisms of toric stacks from $\Xfr$ to $\Xfr'$} as
    \[
        \Map_{\TStk}(\Xfr, \Xfr') = \Map_\TStk^\pre(\Xfr, \Xfr') \times_{\Map_\TStk^\pre(\Tfr, \Xfr')} \Hom_{\tAb(\Stk)}(\Tfr, \Tfr').
    \]
\end{dfn}

Objects of $\Map_{\TStk}(\Xfr, \Xfr')$ (i.e.\ 1-morphisms in $\TStk$) are triples
\[
    (\phi: \Tfr \to \Tfr', f: \Xfr \to \Xfr', \alpha: f \circ j \simto j' \circ \phi)
\]
where:
\begin{itemize}
    \item $\phi$ is a homomorphism of abelian 2-group stacks,
    \item $f$ is a $\Tfr$-equivariant morphism of stacks, and
    \item $\alpha$ is an isomorphism of $\Tfr$-equivariant morphisms $\Tfr \to \Xfr'$.
\end{itemize}
Morphisms $(\phi_1, f_1, \alpha_1) \to (\phi_2, f_2, \alpha_2)$ in $\Map_{\TStk}(\Xfr, \Xfr')$ are pairs of isomorphisms $(\phi_1 \simto \phi_2, f_1 \simto f_2)$ satisfying certain coherence conditions which we will not write out explicitly.

We may now define a composition operation on morphisms of toric stacks: for
\[
    (\phi: \Tfr \to \Tfr', f: \Xfr \to \Xfr', \alpha: f \circ j \simto j' \circ \phi) \in \Map_{\TStk}(\Xfr, \Xfr')
\]
and
\[
    (\phi': \Tfr' \to \Tfr'', f': \Xfr' \to \Xfr'', \alpha': f' \circ j' \simto j'' \circ \phi') \in \Map_{\TStk}(\Xfr', \Xfr''),
\]
the composite is
\[
    (\phi' \circ \phi, f' \circ f, (\alpha' \star \phi) \circ (\id_{f'} \star \alpha): f' \circ f \circ j \simto j'' \circ \phi' \circ \phi) \in \Map_{\TStk}(\Xfr, \Xfr'')
\]
where $\star$ is horizontal composition in the 2-category of stacks with $\Tfr$-action.
Composition extends to a functor $\Map_{\TStk}(\Xfr, \Xfr') \times \Map_{\TStk}(\Xfr', \Xfr'') \to \Map_{\TStk}(\Xfr, \Xfr'')$, so we may define:

\begin{dfn}
    The \emph{2-category of toric stacks} $\TStk$ has:
    \begin{itemize}
        \item Objects given by toric stacks $\Xfr$,
        \item Categories of morphisms given by $\Map_\TStk(\Xfr, \Xfr')$ as defined above, and
        \item Composition as defined above.
    \end{itemize}
\end{dfn}

We will often abuse notation and denote a morphism of toric stacks by $f: \Xfr \to \Xfr'$, writing $f|_\Tfr$ for the corresponding homomorphism $\Tfr \to \Tfr'$ (and leaving the isomorphism $f \circ j \simto j' \circ f|_\Tfr$ implicit).

\subsection{Basic toric stacks}

\begin{dfn} \label{dfn:basic}
    A toric stack $\Ufr$ is \emph{basic} if it can be written as a quotient stack $[U / G]$ where $U$ is an affine toric variety with dense torus $T$ and $G$ is a diagonalizable group acting on $U$ via a homomorphism $G \to T$.
    Write $\TStk^\basic$ for the full subcategory of $\TStk$ with objects the basic toric stacks.
\end{dfn}

\begin{notn}
    If $\Xfr$ is a toric stack and $\Ufr \subset \Xfr$ is an toric open substack such that $\Ufr$ is basic, we say $\Ufr$ is a \emph{basic open} of $\Xfr$.
\end{notn}

\begin{rmk}
    Recall (\cite[Definition 3.1]{alper2013good}) that a quasicompact algebraic stack $\Xfr$ over $\Spec k$ is \emph{cohomologically affine} if its global sections functor $\Gamma(\Xfr, -): \QCoh(\Xfr) \to \Mod_k$ is exact.
    In \cite{geraschenko2015toric_1}, Geraschenko and Satriano use the term ``cohomologically affine'' to refer to basic toric stacks, since the notions agree for toric stacks in their sense (\cite[Remark B.6]{geraschenko2015toric_1}).
    However, it is not clear \emph{a priori} that every cohomologically affine toric stack in our sense is basic, though we will eventually prove this in \zcref{prop:basic_coh_aff}.
    Thus we generally prefer the term ``basic'' to ``cohomologically affine.''
\end{rmk}

Note that the quotient presentation of a basic toric stack $\Ufr$ is non-unique.
For example, we have $\Spec k = \GG_m / \GG_m = \GG_m^2 / \GG_m^2$.
We do not take the choice of a quotient presentation as part of the data of a basic toric stack; instead, we merely require that a quotient presentation exists.
We shall see in \zcref{prop:geometrization_of_cone} that basic toric stacks also admit a ``coordinate-free'' definition: they are precisely the toric stacks that arise as ``geometrizations'' of stacky cones.

We would like to show that the intersection of basic opens in an arbitrary toric stack is still basic.
We will first need the following:

\begin{lem} \label{lem:inclusion_of_basic_is_affine}
    Let $\Xfr$ be a toric stack, let $\Ufr$ be a basic toric stack and let $f: \Ufr \to \Xfr$ be a representable morphism (e.g.\ an open immersion).
    Then $f$ is affine.
\end{lem}

\begin{proof}
    Choose a smooth cover $\{ p_i: U_i \to \Xfr \}$ with all $U_i$ affine.
    By descent it suffices to show all of the projection morphisms $f_i: \Ufr \times_\Xfr U_i \to U_i$ are affine.
    
    Write $\Ufr = U / G$ for some affine toric variety $U$ and diagonalizable group $G$.
    Then, for all $i$, we have $\Ufr \times_\Xfr U_i = (U \times_\Xfr U_i) / G$.
    Since $\Xfr$ has affine diagonal, each $\Ufr \times_\Xfr U_i$ is affine.
    Thus each $\Ufr \times_\Xfr U_i$ is a quotient of an affine scheme $U \times_\Xfr U_i$ by a linearly reductive group, hence is cohomologically affine by \cite[Example 12.9]{alper2013good}.
    Because $f$ is representable, $\Ufr \times_\Xfr U_i$ is an algebraic space and Serre's criterion (e.g.\ \stacks{07V6}) shows each $\Ufr \times_\Xfr U_i$ is an affine scheme.
    Thus all $f_i$ must be affine.
\end{proof}

\begin{prop} \label{prop:intersection_of_basic}
    Let $\Xfr$ be a toric stack, and let $\Ufr \subset \Xfr$ and $\Ufr' \subset \Xfr$ be basic opens of $\Xfr$.
    Then $\Ufr \cap \Ufr'$ is a basic open of $\Xfr$.
\end{prop}

\begin{proof}
    Write $\Ufr = [U / G]$.
    By \zcref{lem:inclusion_of_basic_is_affine}, the inclusion $\Ufr' \hookrightarrow \Xfr$ must be affine.
    Then the base change $U \times_\Xfr \Ufr' \hookrightarrow U$ is an affine torus-equivariant open immersion, so $U \times_\Xfr \Ufr'$ is an affine toric variety.
    Since $\Ufr \cap \Ufr' = (U \times_\Xfr \Ufr') / G$, we see $\Ufr \cap \Ufr'$ is basic.
\end{proof}

\section{Toric stacks via stacky fans} \label{sec:defining_geometrization}

So far, we have introduced the 2-categories $\StkFan$ and $\TStk$.
We now seek to relate these two categories via a ``geometrization'' functor $\Xfr_{-}: \StkFan \to \TStk$.

\subsection{Geometrization of stacky cones}

We begin by discussing how to construct a basic toric stack from a stacky cone.
Our construction here is closely related to that in \cite[\S 4]{ito2025geometric}.

\begin{dfn}
    Let $K$ be a stacky cone.
    The \emph{geometrization} of $K$ is the stack\footnote{This is a stack by faithfully flat descent for quasicoherent sheaves.} $\Xfr_K$ defined by
    \[
        \Xfr_K(R) = \iota_0 \Fun^\otimes\big(\GDual(\Oc(K)), \Mod_R\big)
    \]
    for $\Spec R \in \Aff_k$.
\end{dfn}

\begin{rmk}
    By descent, if $S$ is any test scheme, we have
    \[
        \Xfr_K(S) = \iota_0 \Fun^\otimes\big(\GDual(\Oc(K)), \QCoh(S)\big).
    \]
\end{rmk}

\begin{prop} \label{prop:geometrization_of_cone}
    Let $K$ be a stacky cone.
    Then:
    \begin{enumerate}
        \item $\Xfr_K$ is a basic toric stack, and there is a natural symmetric monoidal equivalence 
        \[
            \big(\QCoh(\Xfr_K), \otimes_{\Osc}\big) \simeq \big(\PSh_k(\GDual(\Oc(K))), \star_{\GDual(\Oc(K))}\big).
        \]
        \item The dense stacky torus $\Tfr_K$ of $\Xfr_K$ satisfies 
        \[
            \Tfr_K = \Xfr_{K(0)} = \SCDual(\ul{\GDual(\Nc_K)}).
        \]
        In particular, $\dim \Xfr_K = \dim \Xfr_{K(0)} = \rk \pi_0(\Nc_K) - \rk \Aut_{\Nc_K}(0)$.
        \item Every basic toric stack arises as $\Xfr_K$ for some stacky cone $K$.
    \end{enumerate}
\end{prop}

\begin{proof}
    (1). Let $\Cc = \GDual(\Oc(K))$, so that $\Xfr_K(S) = \iota_0 \Fun^\otimes(\Cc, \Mod_R)$ for all $\Spec R \in \Aff_k$.
    By the universal property of presheaf categories, we have $\Xfr_K(S) = \iota_0 \Fun^{\otimes,L}(\PSh_k(\Cc), \Mod_S)$ for all $S$.
    Use \zcref{cor:skew_quotient_of_discrete} to write $\Cc \simeq A \vec{/}_\phi M$ for some toric monoid $M$, some discrete abelian group $A$, and some $\phi: M \to A$.
    \zcref{prop:presheaves_as_graded_modules} gives a symmetric monoidal equivalence $(\PSh_k(\Cc), \star_\Cc) \simeq (\GrMod_{k\angles{M}}^A, \otimes_{k\angles{M}})$.
    If $U = \Spec k\angles{M}$ and $G = k\angles{A}$, then the dense torus $T$ of $U$ is $\Spec k\angles{M^\gp}$, and $G$ acts on $U$ via the homomorphism $G \to T$ induced by $M^\gp \to A$.
    It is easy to see that
    \[
        (\GrMod_{k\angles{M}}^A, \otimes_{k\angles{M}}) = (\QCoh(U)^G, \otimes_{\Osc_U}) = (\QCoh(U / G], \otimes_{\Osc_{U/G}}).
    \]
    Thus we have a collection of equivalences
    \[
        \Xfr_K(S) \simeq \iota_0 \Fun^{\otimes,L}(\QCoh(U / G), \QCoh(S))
    \]
    which are natural in $S$.
    By the Tannakian formalism (see e.g.\ \cite[Theorem 2.9]{schappi2018abelian}), we have 
    \[
        \iota_0 \Fun^{\otimes,L}(\QCoh(U / G), \QCoh(S)) \simeq (U/G)(S),
    \]
    for all test schemes $S$, so $\Xfr_K \simeq U/G$.
    By \zcref{ex:quotient_of_toric_variety_is_toric_stack}, we see that $\Xfr_K$ is a toric stack.
    
    The claimed symmetric monoidal equivalence is the composite equivalence
    \[
        (\PSh_k(\Cc), \star_\Cc) \simeq (\GrMod_{k\angles{M}}^A, \otimes_{k\angles{M}}) = (\QCoh([U / G]), \otimes_{\Osc_{U/G}}) \simeq (\QCoh(\Xfr_K), \otimes_{\Osc}).
    \]
    Tracing through the definition of this equivalence, we see that it is actually induced by the natural symmetric monoidal functor $\GDual(\Oc(K)) = \Cc \to \QCoh(\Xfr_K)$, hence is natural in $K$.

    (2). With the notation used in the proof of (1), we have $\Tfr_K \simeq T / G$.
    By the argument of (1), we have
    \[
        \big(\QCoh(T / G), \otimes_{\Osc_{T/G}}\big) = \big(\GrMod_{k\angles{M^\gp}}^A, \otimes_{k\angles{M^\gp}}\big) = \big(\PSh_k(A \vec{/} M^\gp), \star_{A \vec{/} M^\gp}\big)
    \]
    and thus $\Tfr_K \simeq T / G \simeq \Xfr_{K'}$ where $K' = \Cone(\GDual(A \vec{/} M^\gp))$.
    The composite equivalence $\Tfr_K \simto \Xfr_{K'}$ is natural.

    We now need to show that $K' = K(0)$.
    By \zcref{prop:Cgpd_calculation}, we have $A \vec{/} M^\gp = \Cc^\gpd$ and thus $K' = \Cone(\GDual(\Cc^\gpd))$.
    Furthermore, $\GDual(\Cc^\gpd) = \iota_0 \GDual(\Cc) = \Nc_K$ by \zcref{prop:gale_duality_localization}.
    Thus $K' = \Cone(\Nc_K) = K(0)$.

    To see the last equivalence in (2), we note that
    \begin{align*}
        \SCDual(\ul{\GDual(\Nc_K)}) &= \SCDual(\ul{A \vec{/} M^\gp}) \\
        &= \SCDual(\ul{\cofib(M^\gp \to A)}) \\
        &= \CDual(\ul{M^\gp}) / \CDual(\ul{A}) \textrm{ by \zcref{ex:shifted_Cartier_dual_of_2term}} \\
        &= T / G = \Tfr_K.
    \end{align*}
    
    (3). Given a basic toric stack $U / G$, we may write $U = \Spec k\angles{M}$ and $G = k\angles{A}$ for some toric monoid $M$ and finitely generated free abelian group $A$.
    Assuming that the action of $G$ on $U$ arises via a homomorphism $G \to T$ (where $T$ is the dense torus of $U$), we can pass to groups of characters and obtain a homomorphism $M^\gp \to A$.
    This induces a homomorphism $\phi: M \to A$, and running the argument of (1) in reverse shows that $U / G = \Xfr_K$ for $K = \Cone(\GDual(A \vec{/}_{\phi} M))$.
\end{proof}

\begin{notn}
    As in \zcref{prop:geometrization_of_cone}, if $K$ is a stacky cone, we will write $\Tfr_K$ for the dense stacky torus of $\Xfr_K$.
\end{notn}

Here are a few key examples:

\begin{ex}
    Let $X$ be an affine toric variety.
    By classical toric geometry, $X$ is classified by a sharp toric saturated submonoid $P \subset N$, where $N$ is a finitely generated free abelian group.\footnote{In the notation of \cite{cls}, if $X = U_\sigma$ for $\sigma \subset N \otimes \RR$, then $P = \sigma \cap N$.}
    We have $X = \Xfr_{\Cone(N \vec{/} P)}$.
    This can be shown directly, though we will not give a proof since it can be deduced as a special case of \zcref{prop:KM_geometrization} below.
\end{ex}

\begin{ex} \label{ex:BT_as_geometrization}
    Let $T$ be an algebraic torus.
    By \zcref{ex:Gale_dual_of_lattice}, we have $\GDual(\Bsf \coweight(T)) = \weight(T)$, so $\Xfr_{\Cone(\Bsf \coweight(T))}$ is the stack given by
    \[
        R \mapsto \iota_0 \Fun^\otimes(\weight(T), \Mod_R) = \iota_0 \Fun^{\otimes,L}(\QCoh(BT), \Mod_R) = BT(R).
    \]
    Thus $\Xfr_{\Cone(\Bsf \coweight(T))} = BT$.
\end{ex}

Given a stacky cone $K$, the basic toric stack $\Xfr_K$ is naturally a commutative 2-monoid stack with multiplication defined as follows.
The symmetric monoidal product on $\Osc(K)$ corresponds to the diagonal morphism $\Delta_{\GDual(\Oc(K))}: \GDual(\Oc(K)) \to \GDual(\Oc(K)) \times \GDual(\Oc(K))$.
For $S \in \Sch_k$, we can use this diagonal to obtain a multiplication map (for all $S \in \Sch_{k}$):
\begin{align*}
    \Xfr_K(S) \times \Xfr_K(S) &= \Fun^\otimes\big(\GDual(\Oc(K)), \QCoh(S)\big) \times \Fun^\otimes\big(\GDual(\Oc(K)), \QCoh(S)\big) \\
    &\xrightarrow{(F_1, F_2) \mapsto F_1 \otimes_{\Osc} F_2} \Fun^\otimes\big(\GDual(\Oc(K)) \times \GDual(\Oc(K)), \QCoh(S)\big) \\
    &\xrightarrow{- \circ \Delta_{\GDual(\Oc(K))}} \Fun^\otimes\big(\GDual(\Oc(K)), \QCoh(S)\big)
    &= \Xfr_K(S).
\end{align*}
The symmetric monoidal structure morphisms on $\Osc(K)$ make each $\Xfr_K(S)$ into a commutative 2-monoid.
This construction is natural in $S$, so we are justified in saying that $\Xfr_K$ is a commutative 2-monoid stack.

In particular, this construction gives the stacky torus structure on $\Tfr_K = \Xfr_{K(0)}$.
The inclusion $K(0) \subset K$ gives a homomorphism $\Tfr_K \to \Xfr_K$ which gives rise to the natural action $\Tfr_K \curvearrowright \Xfr_K$.

Equipped with this understanding of the stacky torus action on $\Xfr_K$, we can show that $\Xfr_{-}$ actually defines a functor:

\begin{prop} \label{prop:geometrization_of_cone_functorial}
    Geometrization defines a functor $\Xfr_{-}: \StkCone \to \TStk$.
\end{prop}

\begin{proof}
    To see that this construction is functorial in $K$, note that $\Xfr_{-}$ may be understood at the level of underlying stacks as the composite
    \[
        \StkCone \xrightarrow{\Oc} \TCat \xrightarrow{\GDual} \TCat\op \xrightarrow{\iota_0 \Fun^\otimes(-, \QCoh)} \Stk_{k}.
    \]
    Thus $\Xfr_{-}$ defines a functor $\StkCone \to \Stk_k$.
    In fact, the discussion preceding the proposition statement shows that $\Xfr_{-}$ sends morphisms of stacky cones to morphisms of commutative 2-monoid stacks.
    
    To see that $\Xfr_{-}$ actually induces morphisms of toric stacks, suppose we are given some morphism of stacky cones $f: K \to K'$.
    This $f$ induces a morphism $f(0): K(0) \to K'(0)$.
    The commutative diagram
    \[
    \begin{tikzcd}
        K(0) \rar["f(0)"] \dar[hook] & K'(0) \dar[hook] \\
        K \rar["f"] & K'
    \end{tikzcd}
    \]
    induces a commutative diagram of commutative 2-monoid stacks
    \[
    \begin{tikzcd}
        \Tfr_K \rar \dar[hook] & \Tfr_{K'} \dar[hook] \\
        \Xfr_K \rar["\Xfr_f"] & \Xfr_{K'}.
    \end{tikzcd}
    \]
    In particular, the morphism $f$ naturally induces:
    \begin{itemize}
        \item A homomorphism of stacky tori $\Tfr_K \to \Tfr_{K'}$ and
        \item A $\Tfr_K$-equivariant structure on $\Xfr_f$.
    \end{itemize}
    This is exactly what we need to get a morphism in $\TStk$.
    We omit the (straightforward) discussion of 2-morphisms and verification of the 2-functor axioms.
\end{proof}

Before we can move on to discussing the geometrization of stacky fans, we need the following result, which we will generalize in \zcref{prop:open_imm_biconditional}.

\begin{prop} \label{prop:geometrization_cone_immersion}
    Let $K$ be a stacky cone, and let $j: K' \hookrightarrow K$ be a an open immersion of stacky cones.
    Then $\Xfr_j: \Xfr_{K'} \to \Xfr_K$ is an affine open immersion.
\end{prop}

\begin{proof}
    By \zcref{cor:skew_quotient_of_discrete}, we may write $\GDual(\Oc(K)) \simeq A \vec{/} M$ for some toric monoid $M$ and finitely generated abelian group $A$.
    The proof of \zcref{prop:geometrization_of_cone} shows that $\Xfr_K \simeq [U / G]$, where $U = \Spec k\angles{M}$ and $G = \Spec k\angles{A}$.
    
    If $K'$ is a subcone of $K$, the natural map $\Oc(K') \to \Oc(K)$ corresponds to the inclusion of a face of $\PO \Oc(K)$.
    \zcref{prop:gale_duality_localization} then implies the natural map $\GDual(\Oc(K)) \to \GDual(\Oc(K'))$ is a localization.
    By \zcref{prop:localization_properties}, we see that $\GDual(\Oc(K')) = A \vec{/} M[-S]$ for some face $S \subset M$.
    This allows us to write $\Xfr_{K'} \simeq [U' / G]$, where $U' = \Spec k\angles{M[-S]} = \Spec k\angles{M}[S\inv]$ is an open subvariety of $U$, as needed.
\end{proof}

\subsection{Geometrization of stacky fans}

So far we have discussed how to geometrize stacky cones to obtain basic toric stacks.
We can obtain more general toric stacks by geometrizing stacky fans:

\begin{dfn}
    Let $\Sigma$ be a stacky fan.
    The \emph{geometrization of $\Sigma$} is the stack $\Xfr_\Sigma = \bigcup_{K \subset \Sigma \textrm{ a subcone}} \Xfr_K$.
\end{dfn}

Here the ``union'' is really a colimit in $\Stk$ taken over open immersions $\Xfr_{K'} \subset \Xfr_K$ for subcones $K' \subset K \subset \Sigma$.
This is a genuine union of groupoids at the level of geometric points, though for more general $S$-points we need to stackify.

We think of $\Xfr_{\Sigma}$ as a ``union of basic opens.''
The poset indexing this colimit may be identified with $|\Sigma|$, so the basic opens of $\Xfr_{\Sigma}$ correspond to points of $|\Sigma|$, just like how affine toric open subvarieties of a toric variety correspond to cones in the fan.

\begin{prop} \label{prop:geometrization_of_fan}
    If $\Sigma$ is a stacky fan, then $\Xfr_\Sigma$ is a toric stack with dense stacky torus 
    \[
        \Tfr_\Sigma = \Xfr_{\Sigma(0)} = \SCDual(\ul{\GDual(\Nc_\Sigma)}).
    \]
    In particular, $\dim \Xfr_\Sigma = \dim \Xfr_{\Sigma(0)} = \rk \pi_0(\Nc_\Sigma) - \rk \Aut_{\Nc_\Sigma}(0)$.
\end{prop}

\begin{proof}
    Let $\Sigma$ be a stacky fan.
    For every subcone $K \subset \Sigma$, we have $\Sigma(0) = K(0) \subset \Sigma$, so $\Tfr_\Sigma$ acts on $\Xfr_K$ and embeds as a dense open of $\Xfr_K$.
    The transition maps $\Xfr_{K'} \to \Xfr_K$ are all $\Tfr_\Sigma$-equivariant and respect these embeddings by \zcref{prop:geometrization_of_cone_functorial}, so we get a natural action of $\Tfr_\Sigma$ on $\Xfr_\Sigma$ and a dense open embedding $\Tfr_\Sigma \hookrightarrow \Xfr_\Sigma$.
    The stack $\Xfr_\Sigma$ is normal and has diagonalizable stabilizers at geometric points (since these properties can be checked locally on $\Xfr_\Sigma$, this follows from \zcref{prop:geometrization_of_cone}).

    We still need to check that $\Xfr_\Sigma$ has affine diagonal.
    For each subcone $K \subset \Sigma$, choose an affine toric variety $U_K$ with a smooth toric surjection $U_K \twoheadrightarrow \Xfr_K$.
    Let $p_K: U_K \to \Xfr_\Sigma$ be the induced morphism.
    It suffices to show that the smooth cover $\{ p_K: U_K \to \Xfr_\Sigma \}_{K \subset \Sigma}$ satisfies the intersection condition of \zcref{prop:affine_diagonal_intersection_criterion}.
    
    Given two subcones $K_1, K_2 \subset \Sigma$, we have $\Xfr_{K_1} \times_{\Xfr_\Sigma} \Xfr_{K_2} = \Xfr_{K_1 \cap K_2}$, so the images of $p_{K_1}$ and $p_{K_2}$ meet in $\Xfr_{K_1 \cap K_2}$.
    Thus
    \begin{equation} \label{eq:fiber_product_checking_fan_affine_diagonal}
        U_{K_1} \times_{\Xfr_\Sigma} U_{K_2} = p_{K_1}\inv(\Xfr_{K_1 \cap K_2}) \times_{\Xfr_{K_1 \cap K_2}} p_{K_2}\inv(\Xfr_{K_1 \cap K_2}).
    \end{equation}
    Now $p_{K_1}\inv(\Xfr_{K_1 \cap K_2})$ fits into a Cartesian square
    \[
    \begin{tikzcd}
        p_{K_1}\inv(\Xfr_{K_1 \cap K_2}) \rar \dar & U_{K_1} \dar \\
        \Xfr_{K_1 \cap K_2} \rar[hook] & \Xfr_{K_1}.
    \end{tikzcd}
    \]
    The bottom map here is affine by \zcref{prop:geometrization_cone_immersion}, so since $U_{K_1}$ is an affine scheme we see that $p_{K_1}\inv(\Xfr_{K_1 \cap K_2})$ is also an affine scheme.
    Similarly, $p_{K_2}\inv(\Xfr_{K_1 \cap K_2})$ is an affine scheme.
    By \zcref{prop:geometrization_of_cone}, $\Xfr_{K_1 \cap K_2}$ has affine diagonal.
    Thus the fiber product in \zcref{eq:fiber_product_checking_fan_affine_diagonal} must be affine by \zcref{prop:affine_diagonal_intersection_criterion}, and we win.
\end{proof}

It is clear that geometrization defines a functor $\Xfr_{-}: \StkFan \to \TStk$.

\begin{rmk}
    Those in the know may think of $\Xfr_{-}: \StkFan \to \TStk$ as the left Kan extension of the corresponding functor $\Xfr_{-}: \StkCone \to \TStk$ along the inclusion $\StkCone \hookrightarrow \StkFan$.
    We will not have any explicit need of this perspective.
\end{rmk}

\subsection{Geometrization and other formalisms}

Here we establish that our geometrization construction generalizes the constructions of \cite{gillam2015theory} and \cite{geraschenko2015toric_1}.
The proofs here are rather unenlightening exercises in matching up notation, and we will not fault any reader who chooses to skip them.

\begin{prop} \label{prop:KM_geometrization}
    Let $(N, \{P_\sigma\}_{\sigma \in |\Sigma|})$ be a KM fan, and let $\Sigma$ be the corresponding stacky fan constructed via \zcref{ex:KM_fan}.
    Then the geometrization $\Xfr_\Sigma$ agrees with the toric stack constructed from $(N, \{P_\sigma\}_{\sigma \in |\Sigma|})$ in \cite[\S 2.5-2.6]{gillam2015theory}.
\end{prop}

\begin{proof}
    Let us first consider the case where $(N, \{P_\sigma\}_{\sigma \in |\Sigma|})$ is a single cone (with KM structure) together with all of its faces.
    In other words, the data of $(N, \{P_\sigma\}_{\sigma \in |\Sigma|})$ consists of a sharp toric submonoid $P$ of some finitely generated abelian group $N$.

    To construct the corresponding toric stack via the approach of \cite{gillam2015theory}, we perform the following steps:
    \begin{enumerate}
        \item Choose a free subgroup $L \subset N$ containing $P^\gp$ and such that $N / L$ is finite.
        \item Apply $\RHom(-, \ZZ)$ to the short exact sequence
        \[
        \begin{tikzcd}
            0 \rar & L \rar & N \rar & N / L \rar & 0.
        \end{tikzcd}
        \]
        to get an exact sequence
        \[
        \begin{tikzcd}
            0 \rar & \Hom(N, \ZZ) \rar & \Hom(L, \ZZ) \rar & \Ext^1(N / L, \ZZ) \rar & \Ext^1(N, \ZZ) \rar & 0.
        \end{tikzcd}
        \]
        \item Let $S = \bset{\phi: L \to \ZZ}{\phi(P) \geq 0}$.
        The inclusion $S \hookrightarrow \Hom(L, \ZZ)$ induces a homomorphism $S \to \Ext^1(N / L, \ZZ)$.
        \item Take the quotient
        \[
            (\Spec k\angles{S}) / (\Spec k\angles{\Ext^1(N / L, \ZZ)})
        \]
        where the action is induced by the homomorphism described in the previous step.
    \end{enumerate}
    By the construction in \zcref{prop:geometrization_of_cone}, this quotient stack is the same as the geometrization of $\Cone(\GDual(\Ext^1(N / L, \ZZ) \vec{/} S))$.
    Thus it suffices to show that the stacky cone $K$ corresponding to $(N, \{P_\sigma\}_{\sigma \in |\Sigma|})$ via \zcref{ex:KM_fan} is 
    \[
        \Cone(\GDual(\Ext^1(N / L, \ZZ) \vec{/} S)).
    \]
    Since $K = \Cone(N \vec{/} P)$ by construction, we actually just need to show $\GDual(N \vec{/} P) = \Ext^1(N / L, \ZZ) \vec{/} S$, or equivalently
    \[
        \GDual(\Ext^1(N / L, \ZZ) \vec{/} S) = N \vec{/} P.
    \]

    To accomplish this we will use our explicit description of Gale duality (\zcref{cons:gale_dual}).
    First note that $S = \Hom(L, \ZZ) \times_{\Hom(P, \ZZ)} \Hom(P, \NN)$, so
    \begin{align*}
        D_\CMon(S) &= D_\CMon(\Hom(L, \ZZ)) \oplus_{D_\CMon(\Hom(P, \ZZ))} D_\CMon(\Hom(P, \ZZ)) \\
        &= 0 \oplus_0 D_\CMon^2(P) \textrm{ because $D_\CMon$ sends any abelian group to $0$} \\
        &= 0 \oplus_0 P \\
        &= P.
    \end{align*}
    Thus $\PO \GDual(\Ext^1(N / L, \ZZ) \vec{/} S) = P$.
    To compute $\iota_0 \GDual(\Ext^1(N / L, \ZZ) \vec{/} S)$, observe that $S^\gp = \Hom(L, \ZZ)$ (this is precisely the standard computation that shows that $L$ is the character lattice of the dense torus in $\Spec k\angles{S}$).
    Thus
    \[
        \fib(S^\gp \to \Ext^1(N / L, \ZZ)) = \fib(\Hom(L, \ZZ) \to \Hom(N / L, \ZZ)) = \RHom(N, \ZZ),
    \]
    as we can see using the exact sequence of step (2) above.
    We may now compute
    \[
        \iota_0 \GDual(\Ext^1(N / L, \ZZ) \vec{/} S) = \RHom(\RHom(N, \ZZ), \ZZ) = N
    \]
    so that $\GDual(\Ext^1(N / L, \ZZ) \vec{/} S) = N \vec{/} P$ as desired.

    We see that Gillam--Molcho's construction of toric stacks agrees with our own when the input is a single KM cone.
    For the case of more general KM fans, we note that both our construction and Gillam--Molcho's are based on gluing these locally defined toric stacks via the same gluing maps, so the two constructions agree in general.
\end{proof}

\begin{prop}
    Let $(L, \ul{\Sigma}, \beta: L \to N)$ be a non-strict stacky fan in the sense of Geraschenko--Satriano, and let $\Sigma$ be the corresponding stacky fan (see \zcref{ex:GS_fan}).
    Then the geometrization $\Xfr_\Sigma$ agrees with the toric stack constructed in \cite[\S 2.2]{geraschenko2015toric_1}.
\end{prop}

\begin{proof}
    As in the proof of \zcref{prop:KM_geometrization}, it suffices to show that the result holds when the Geraschenko--Satriano stacky fan has a single maximal cone $\ul{\sigma}$ (since it is easy to see the result is compatible with gluing).
    So suppose this is the case.
    Then the stacky fan (in our sense) corresponding to the GS stacky fan in question is the stacky cone
    \[
        K = \Cone\big(\Omega^\infty\cofib\big(\RHom(H^0_\beta, \ZZ) \oplus \Hom(H^1_\beta, \ZZ) \to L\big) \vec{/} \Hom(\ul{\sigma} \cap L, \ZZ)\big).
    \]
    If we let $A = H^0_\beta \oplus H^1_\beta$ and $M = \bset{\phi: L \to \ZZ}{\phi|_{\sigma \cap L} \geq 0}$, then a direct computation (along the lines of \zcref{ex:Gale_dual_classical_toric}) shows that
    \[
        K = \Cone(\GDual(A \vec{/} M)).
    \]
    The construction of \zcref{prop:geometrization_of_cone} then shows
    \[
        \Xfr_K = (\Spec k\angles{M}) / (\Spec k\angles{A}),
    \]
    which exactly agrees with the quotient stack constructed in \cite[\S 2.2]{geraschenko2015toric_1}.
\end{proof}

\part{Geometrization as an equivalence of 2-categories} \label{part:geometrization_is_equivalence}

This part is dedicated to proving that the geometrization functor $\Xfr_{-}: \StkFan \to \TStk$ is an equivalence of 2-categories.
We refer to \zcref{subsec:organization} for a brief outline of the proof.

\section{Equivariantization} \label{sec:equivariantization}

One of the key reasons for working with stacks is that doing so allows us to translate equivariant questions about $G$-actions on a scheme $X$ into intrinsic questions about the moduli stack $X / G$ (and vice versa).
We'd like to be able to perform a similar translation in the theory of toric stacks.
More precisely, if $\Xfr$ is a toric stack with dense stacky torus $\Tfr$, then we'd like to construct a ``quotient stack'' $\Xfr / \Tfr$, so equivariant questions about $\Tfr \hookrightarrow \Xfr$ (which can be daunting) translate into intrinsic questions about $\Xfr / \Tfr$, which has trivial torus.

\subsection{Rigidifying toric stacks}

To construct a quotient stack $\Xfr / \Tfr$, we will first need to understand how to remove the generic stabilizer of $\Xfr$.
We introduce the following notation:

\begin{dfn}
    If $\Tfr$ is a stacky torus with identity morphism $e: \Spec k \to \Tfr$, let $\Tfr_\stab := \ul{\Aut}(e)$ be the \emph{stabilizer group} of $\Tfr$.
    This is a diagonalizable algebraic group.
\end{dfn}

For any stacky torus $\Tfr$, there is a natural homomorphism of stacky tori $B \Tfr_\stab \to \Tfr$.
In particular, if $\Xfr$ is a toric stack with dense torus $\Tfr$, then $B \Tfr_\stab$ acts on $\Xfr$.

The ``quotient'' $\Xfr / B \Tfr_\stab$ can be made mathematically precise: it is the ``rigidification'' $\Xfr \fslash \Tfr_\stab$ discussed in \cite[\S 5]{romagny2005group} and \cite[\S 5.1]{abramovich2003twisted}.\footnote{As noted in \cite[paragraph after Theorem 5.1.5]{abramovich2003twisted}, we can also make sense of $\Xfr / B \Tfr_\stab$ directly using the theory of higher stacks, though this requires a level of abstraction we will not need in this paper.}
There is a natural morphism $\Xfr \to \Xfr \fslash \Tfr_\stab$ making $\Xfr$ into a $\Tfr_\stab$-gerbe\footnote{Readers unaccustomed to gerbes should mentally substitute all references to ``$D$-gerbes'' (for $D$ diagonalizable) with references to (\'etale-locally trivial) ``$BD$-torsors.'' We will not lose anything important by this.} over $\Xfr \fslash \Tfr_\stab$, and this morphism is the universal morphism annihilating the action of $\Tfr_\stab$ on $\Xfr$.

To see that the construction of the rigidification in \cite[\S 5.1]{abramovich2003twisted} applies in this situation, we will need the following result:

\begin{lem} \label{lem:stabilizer_action_faithful}
    Let $\Xfr$ be a toric stack.
    Then, for all $k$-schemes $S$ and all $x \in \Xfr(S)$, the induced homomorphism $\phi_x: \Tfr_\stab \times S \to \ul{\Aut}_S(x)$ (of group schemes over $S$) has trivial kernel.
\end{lem}

\begin{proof}
    First consider the case where $x$ is a smooth morphism $p: U \to \Xfr$ with $U$ a scheme.
    Because $\Tfr$ is dense in $\Xfr$, the preimage $p\inv(\Tfr) = \Tfr \times_{\Xfr} U$ is dense in $U$.
    The restricted homomorphism $\phi_p|_{p\inv(\Tfr)}: \Tfr_\stab \times p\inv(\Tfr) \to \ul{\Aut}_{p\inv(\Tfr)}(p|_{p\inv(\Tfr)})$ is an isomorphism (fiberwise, it is $\Tfr_\stab \simto \Tfr_\stab$).
    Thus $\ker \phi_p$ is generically trivial over $U$.
    By \cite[Th\'eor\`eme 6.8]{grothendieck1970groupes}, $\ker \phi_p$ is of multiplicative type, i.e.\ is diagonalizable (fpqc-)locally on $S$.
    Because $\ker \phi_p$ is generically trivial, this means that $\ker \phi_p$ is locally trivial.
    By faithfully flat descent, $\ker \phi_p$ must itself be trivial.

    Now consider the case of an arbitrary morphism $x: S \to \Xfr$.
    Let $p: U \to \Xfr$ be a smooth surjection with $U$ a scheme, and let $\pi_1: S \times_\Xfr U \to S$ and $\pi_2: S \times_\Xfr U \to U$ be the projection maps.
    Consider the Cartesian square
    \[
    \begin{tikzcd}
        \ker \phi_{p \circ \pi_1} \rar \dar & \ker \phi_p \dar \\
        S \times_\Xfr U \rar["\pi_1"] & S.
    \end{tikzcd}
    \]
    Because $p$ is a smooth surjection, $\pi_1$ is a smooth surjection, so faithfully flat descent tells us that $\ker \phi_p$ is trivial if and only if $\ker \phi_{x \circ \pi_1}$ is trivial.
    We can write $x \circ \pi_1 = p \circ \pi_2$, so
    \[
        \ker \phi_{x \circ \pi_1} = \ker \phi_{p \circ \pi_2} = (\ker \phi_p) \times_U (S \times_\Xfr U).
    \]
    By the result of the first paragraph above, $\ker \phi_p$ is trivial.
    Thus $\ker \phi_{x \circ \pi_1}$ and $\ker \phi_x$ are also trivial.
\end{proof}

We will also need one direction of the following result, allowing us to ``bootstrap'' the affine diagonal condition to or from a gerbe.
The converse direction will be useful later on, so we also prove it here.

\begin{lem} \label{lem:gerbe_affine_diagonal}
    Let $D$ be a diagonalizable group over $\Spec k$, let $\Yfr$ be an Artin stack of finite type over $\Spec k$, and let $p: \Xfr \to \Yfr$ be a $D$-gerbe.
    Then:
    \begin{enumerate}
        \item $p$ has affine diagonal.
        \item The stack $\Yfr$ has affine diagonal if and only if $\Xfr$ has affine diagonal.
    \end{enumerate}
\end{lem}

\begin{proof}
    Recall (\zcref{rmk:affine_diagonal_reasonable}) that having affine diagonal is a ``reasonable property'' of morphisms.
    
    (1). Because $p$ is locally trivial, we may choose a smooth surjection $\Ufr \to \Yfr$ such that $\Xfr \times_\Yfr \Ufr = \Ufr \times BD$ and such that this equivalence matches the natural projection maps from both sides to $\Ufr$.
    That is, we have a Cartesian square
    \[
    \begin{tikzcd}
        \Ufr \times BD \rar["\pi_1"] \dar & \Ufr \dar \\
        \Xfr \rar["p"] & \Yfr.
    \end{tikzcd}
    \]
    The forgetful map $BD \to \Spec k$ has affine diagonal, so the projection map $\pi_1$ also has affine diagonal.
    Because having affine diagonal is smooth-local on the target, $p$ must have affine diagonal.

    (2). The ``only if'' direction follows from (1), since $\Xfr \to \Spec k$ is the composite of $p$ and the structure morphism $\Yfr \to \Spec k$, both of which have affine diagonal.

    For the ``if'' direction, by \zcref{prop:affine_diagonal_intersection_criterion}, it suffices to find a smooth cover $\{ q_i: U_i \to \Yfr \}_{i \in I}$ with all $U_i$ affine such that all fiber products $U_i \times_{\Yfr} U_j$ are affine.
    Recall (see e.g.\ \cite[Appendix C.1]{abramovich2008gromov}) that the rigidification $\Xfr \fslash D$ exists and $p$ can be identified with the natural map $\Xfr \to \Xfr \fslash D$.
    Let $\{ q_i: U_i \to \Xfr \fslash D \}_{i \in I}$ be a smooth cover of $\Xfr \fslash D$ by affine schemes such that the gerbe $\Xfr \to \Xfr \fslash D$ is trivial on each $U_i$.
    We claim any such cover satisfies the desired condition.

    Note that we may factor each $q_i$ through some $q_i': U_i \to \Xfr$.
    By the discussion in \cite[\S 5.1.7]{abramovich2003twisted}, each fiber product $U_i \times_{\Xfr \fslash D} U_j$ is a quotient of the algebraic space $\ul{\Isom}_{U_i \times U_j}(q_i' \circ \pi_1, q_j' \circ pi_2)$ by a free action of the diagonalizable group scheme $D \times (U_i \times U_j)$ over $U_i \times U_j$.
    The $\ul{\Isom}$ sheaf here is an affine scheme because $\Xfr$ has affine diagonal, so \cite[Theorem 13.2]{alper2013good} implies that the good moduli space of $U_i \times_{\Xfr \fslash D} U_j$ is an affine scheme.
    But $U_i \times_{\Xfr \fslash D} U_j$ is already an algebraic space (since the group scheme action is free), so uniqueness of good moduli spaces then implies $U_i \times_{\Xfr \fslash D} U_j$ is itself an affine scheme, as needed.
\end{proof}

We are now able to show that we can rigidify a toric stack $\Xfr$ and that the result $\Xfr \fslash \Tfr_\stab$ is still a toric stack.

\begin{prop} \label{prop:rigidification_is_toric}
    Let $\Xfr$ be a toric stack with dense stacky torus $\Tfr$.
    Then the rigidification $\Xfr \fslash \Tfr_\stab$ exists and is an algebraic stack.
    The natural morphism $\Xfr \to \Xfr \fslash \Tfr_\stab$ is a $\Tfr_\stab$-gerbe.
    Furthermore, $\Xfr \fslash \Tfr_\stab$ is a toric stack with dense torus $\Tfr_\coarse$ (and thus trivial generic stabilizer).
\end{prop}

\begin{proof}
    The first two statements follow from \cite[Theorem 5.1.5]{abramovich2003twisted}, using \zcref{lem:stabilizer_action_faithful} to verify that the hypotheses of the cited theorem hold.
    
    To see that $\Xfr \fslash \Tfr_\stab$ is toric, we first note by the naturality of the rigidification construction that the action $\Tfr \curvearrowright \Xfr$ descends to an action $\Tfr_\coarse = \Tfr \fslash \Tfr_\stab \curvearrowright \Xfr \fslash \Tfr_\stab$.
    Similarly, the $\Tfr$-equivariant open immersion $\Tfr \hookrightarrow \Xfr$ descends to a $\Tfr$-equivariant open immersion $\Tfr_\coarse \hookrightarrow \Xfr \fslash \Tfr_\stab$.
    
    The stack $\Xfr \fslash \Tfr_\stab$ is normal and of finite type because these are smooth-local properties.
    \zcref{lem:gerbe_affine_diagonal} applied to the natural map $\Xfr \to \Xfr \fslash \Tfr_\stab$ implies that $\Xfr \fslash \Tfr_\stab$ has affine diagonal.
    Finally, the stabilizer groups of $\Xfr \fslash \Tfr_\stab$ at geometric points are quotients of the stabilizer groups of $\Xfr$ at the corresponding geometric points, hence are diagonalizable.
    Thus $\Xfr \fslash \Tfr_\stab$ is toric.
\end{proof}

Note in particular that when $\Xfr = \Tfr$, we see that the rigidification $\Tfr \fslash \Tfr_\stab$ exists.

\begin{dfn}
    Let $\Tfr$ be a stacky torus.
    The \emph{coarse group} of $\Tfr$ is $\Tfr_\coarse := \Tfr \fslash \Tfr_\stab$.
    This is an algebraic torus.
\end{dfn}

\begin{rmk}
    Suppose $\Tfr = T / G$ for some algebraic torus $T$, diagonalizable $G$, and some $\phi: G \to T$.
    Then $\Tfr_\stab = \ker(\phi)$ and $\Tfr_\coarse = \coker(\phi)$.
    Unlike $T$ and $G$ however, the groups $\Tfr_\stab$ and $\Tfr_\coarse$ are defined canonically in terms of $\Tfr$, with no auxiliary data needed.
\end{rmk}

\subsection{Equivariantization of toric stacks}

In the previous subsection we discussed how to annihilate the stabilizer group of the dense stacky torus of a toric stack $\Xfr$.
Now we would like to annihilate the rest of the dense torus, producing a quotient toric stack $\Xfr / \Tfr$.

\begin{prop} \label{prop:equivariantization_is_toric}
    Let $\Xfr$ be a toric stack with dense torus $\Tfr$.
    Then the quotient stack 
    \[
        \Xfr / \Tfr := (\Xfr \fslash \Tfr_\stab) / \Tfr_\coarse
    \]
    exists and is a toric stack with trivial dense torus.
\end{prop}

\begin{proof}
    By \zcref{prop:rigidification_is_toric}, the stack $(\Xfr \fslash \Tfr_\stab)$ is toric with dense stacky torus $\Tfr \fslash \Tfr_\stab = \Tfr_\coarse$.
    Thus, replacing $\Xfr$ by $\Xfr \fslash \Tfr_\stab$, we may assume without loss of generality that $\Tfr = \Tfr_\coarse$.
    For conceptual clarity we will write $\Tfr = T$ since this group is a scheme.
    
    The quotient stack $\Xfr / T$ is algebraic by \cite[Theorem 4.1]{romagny2005group}.
    Because the natural morphism $\Xfr \to \Xfr / T$ is smooth, we see that $\Xfr / T$ is normal and of finite type over $\Spec k$.
    To see that $\Xfr / T$ has affine diagonal, consider the Cartesian square
    \[
        \begin{tikzcd}
            \Xfr \rar \dar & \Spec k \dar \\
            {\Xfr / T} \rar & BT.
        \end{tikzcd}
    \]
    Because $\Xfr \to \Spec k$ has affine diagonal and this property is local on the base, $\Xfr / T \to BT$ has affine diagonal.
    Composing this morphism with $BT \to \Spec k$, we see that $\Xfr / T \to \Spec k$ has affine diagonal.
    
    Finally, let $x \in \Xfr(k')$ be a geometric point (for some algebraically closed $k' \supset k$), giving rise to a geometric point $\ol{x} \in \Xfr / T(k')$.
    The stabilizer group $\ul{\Aut}_{k'}(\ol{x})$ is an extension of the algebraic torus $T$ by the diagonalizable group $\ul{\Aut}_{k'}(x)$.
    This is necessarily diagonalizable by the classification of such extensions (see e.g.\ \cite[Proposition 15.28]{milne2017algebraic}).
    Thus the stabilizer groups of all geometric points of $\Xfr / T$ are reductive.
\end{proof}

\begin{dfn}
    We refer to the toric stack $\Xfr / \Tfr := (\Xfr \fslash \Tfr_\stab) / \Tfr_\coarse$ as the \emph{equivariantization} of $\Xfr$.
\end{dfn}

\begin{rmk} \label{rmk:equivariantization_switch_order}
    In the setting of \zcref{prop:equivariantization_is_toric}, if we choose a splitting of the natural map $\Tfr \to \Tfr_\coarse$, then we get an action of $\Tfr_\coarse$ on $\Xfr$.
    By \cite[Theorem 5.1]{romagny2005group}, this induces an equivalence $\Xfr / \Tfr \simeq (\Xfr / \Tfr_\coarse) \fslash \Tfr_\stab$.
    In other words, quotienting by the coarse group commutes with rigidifying the stabilizer group.
\end{rmk}

\subsection{A weak structure theorem for toric stacks}

The equivariantization construction lets us decompose questions about general toric stacks into:
\begin{itemize}
    \item Questions about toric stacks with trivial torus,
    \item Questions about torsors over toric stacks with trivial torus, and
    \item Questions about gerbes over toric stacks with trivial torus. 
\end{itemize}
The following results make this principle mathematically precise:

\begin{prop} \label{prop:eqvar_cartesian_square}
    Let $\Xfr$ be a toric stack and choose a splitting of the natural map $\Tfr \to \Tfr_\coarse$, so $\Tfr_\coarse$ acts on $\Xfr$.
    Then there is a Cartesian square (in $\TStk$)
    \[
    \begin{tikzcd}
        \Xfr \rar["p"] \dar & \Xfr \fslash \Tfr_\stab \dar \\
        {\Xfr / \Tfr_\coarse} \rar["\ol{p}"] & {\Xfr / \Tfr}.
    \end{tikzcd}
    \]
\end{prop}

\begin{proof}
    Let $\Yfr = \Xfr \fslash \Tfr_\stab$, so $\Xfr / \Tfr = \Yfr / \Tfr_\coarse$.
    The square in question becomes
    \[
    \begin{tikzcd}
        \Xfr \rar \dar & \Yfr \dar \\
        \Xfr / \Tfr_\coarse \rar & \Yfr / \Tfr_\coarse.
    \end{tikzcd}
    \]
    which is the usual Cartesian square of quotient stacks associated with a $\Tfr_\coarse$-equivariant morphism.
    It is clear that all morphisms involved here are toric and that $\Xfr$ satisfies the universal property of the Cartesian product in the category of toric stacks.
\end{proof}

\begin{cor} \label{cor:toric_stack_bottom_up}
    Let $\Xfr$ be a toric stack.
    Then there is a Cartesian square (in $\TStk$)
    \[
    \begin{tikzcd}
        \Xfr \rar \dar & \Xfr_1 \dar["p_1"] \\
        \Xfr_2 \rar["p_2"] & \Xfr_0
    \end{tikzcd}
    \]
    where:
    \begin{itemize}
        \item $\Xfr_0$ is a toric stack with trivial torus,
        \item $p_1: \Xfr_1 \to \Xfr_0$ is the total space of a $T$-torsor for some algebraic torus $T$, and
        \item $p_2: \Xfr_2 \to \Xfr_0$ is a $D$-gerbe for some diagonalizable group $D$.
    \end{itemize}
\end{cor}

\begin{proof}
    Take $\Xfr_0 = \Xfr / \Tfr$, $\Xfr_1 = \Xfr \fslash \Tfr_\stab$, and $\Xfr_2 = \Xfr / \Tfr_\coarse$.
\end{proof}

\subsection{The stacky Sumihiro theorem and some consequences}

By passing to the equivariantization of a toric stack, we are able to reduce many problems to the easier case where the toric stack has trivial torus.
This case is much easier to analyze using the key local result of \cite{geraschenko2015toric_2}:\footnote{The errors in \cite{geraschenko2015toric_2} lie in the global analysis. \zcref{thm:stacky_sumihiro} is, to the best of the author's knowledge, entirely correct (and quite useful).}

\begin{thm}[{``Stacky Sumihiro theorem,'' \cite[Theorem 4.5]{geraschenko2015toric_2}}] \label{thm:stacky_sumihiro}
    Let $\Xfr$ be a toric stack with trivial torus, and let $x \in \Xfr(k)$.
    Then there exists a basic toric stack $\Ufr$ with trivial torus and a Zariski open immersion $j: \Ufr \hookrightarrow \Xfr$ of toric stacks such that $x \in j(\Ufr(k))$.
\end{thm}

\begin{rmk}
    The hypotheses used in \cite[Theorem 4.5]{geraschenko2015toric_2} are slightly different from those we make in \zcref{thm:stacky_sumihiro}, so we should explain why the differences do not present any issues:
    \begin{itemize}
        \item \cite[Theorem 4.5]{geraschenko2015toric_2} assumes that the stabilizers at geometric points of $\Xfr$ are linearly reductive but not \emph{a priori} diagonalizable.
        Of course, it follows from this result that all such stabilizers are in fact diagonalizable.
        We have decided to instead assume from the start that the stabilizers in question are diagonalizable, since this simplifies our analysis.
        Furthermore, diagonalizability seems like a natural assumption to make, since toric stacks should be ``as commutative as possible.''
        \item \cite[Theorem 4.5]{geraschenko2015toric_2} includes a technical hypothesis that $\Xfr$ is of ``global type.''
        As mentioned in \cite[Remark 4.4]{geraschenko2015toric_2}, the main theorem of \cite{alper2020luna} (which was still in preparation at the time of publication of \cite{geraschenko2015toric_2}) implies that every toric stack is automatically of ``global type,'' so we have removed this hypothesis.
    \end{itemize}
\end{rmk}

\zcref{thm:stacky_sumihiro} allows us to relate the topology of a toric stack with trivial torus to its set of $k$-points.

\begin{lem} \label{lem:pts_and_nbhds_trivial_torus_basic_case}
    Let $\Ufr$ be a basic toric stack with trivial torus, and write $\Ufr = \Xfr_K$ for some stacky cone $K$.
    Then:
    \begin{enumerate}
        \item $\Ufr(k) = |K|$.
        \item For every $x \in \Xfr(k)$, there is a unique smallest basic open $\Vfr_x \subset \Xfr$ containing $x$.
        \item The assignment $x \mapsto \Vfr_x$ gives a order-preserving bijection between $\Xfr(k)$ and the poset of basic opens of $\Xfr$ ordered by inclusion.
    \end{enumerate}
\end{lem}

\begin{proof}
    Write $\Ufr = U / T$, where $U$ is an affine toric variety and $T$ is its dense torus.
    Under this identification, elements of $\Ufr(k)$ correspond to $T$-orbits of closed points in $U$, and basic opens of $\Ufr$ correspond to $T$-equivariant open affines in $U$.
    Thus the result follows from the standard orbit-cone correspondence for toric varieties (\cite[Theorem 3.2.6]{cls}).
\end{proof}

\begin{prop} \label{prop:pts_and_nbhds_trivial_torus}
    Let $\Xfr$ be a toric stack with trivial torus.
    Then:
    \begin{enumerate}
        \item $\Xfr(k)$ is finite.
        \item For every $x \in \Xfr(k)$, there is a unique smallest basic open $\Ufr_x \subset \Xfr$ containing $x$.
        \item The assignment $x \mapsto \Ufr_x$ gives a bijection between $\Xfr(k)$ and the set of basic opens of $\Xfr$.
        \item $\Xfr = \cup_{x \in \Xfr(k)} \Ufr_x$.
    \end{enumerate}
\end{prop}

\begin{proof}
    (1). \zcref{thm:stacky_sumihiro} implies every $k$-point of $x \in \Xfr$ is contained in a basic open $\Ufr'_x \subset \Xfr$.
    Because $\Xfr$ is of finite type over the algebraically closed field $k$, the $k$-points of $\Xfr$ are dense in $\Xfr$.
    Thus we may write $\Xfr = \cup_{x \in \Xfr(k)} \Ufr'_x$.
    Because $\Xfr$ is quasicompact, we may choose some finite subset $\{ x_i \}_{i \in I} \subset \Xfr(k)$ such that $\Xfr = \cup_{i \in I} \Ufr'_{x_i}$.
    Then $\Xfr(k) = \cup_{i \in I} \Ufr'_{x_i}(k)$, and \zcref{lem:pts_and_nbhds_trivial_torus_basic_case}(1) implies $\Xfr(k)$ is finite.

    (2). Suppose we are given $x \in \Xfr(k)$.
    Let $\Ufr_x$ be the intersection of all basic opens of $\Xfr$ containing $x$.
    All basic opens of $\Xfr$ are determined by their $k$-points, and $\Xfr(k)$ is finite, so in fact $\Ufr_x$ is a finite intersection of basic opens of $\Xfr$.
    \zcref{prop:intersection_of_basic} then implies $\Ufr_x$ is itself a basic open of $\Xfr$.
    
    (3). The assignment $x \mapsto \Ufr_x$ is injective: if $\Ufr_x = \Ufr_{x'}$, then we have $x, x' \in \Ufr_x$ and so \zcref{lem:pts_and_nbhds_trivial_torus_basic_case}(3) (applied to $\Vfr = \Ufr_x$) implies $x = x'$.
    To see that $x \mapsto \Ufr_x$ is surjective, let $\Vfr \subset \Xfr$ be any basic open.
    Again by \zcref{lem:pts_and_nbhds_trivial_torus_basic_case}(3), we can find a point $x_\Vfr \in \Vfr(k)$ such that $\Vfr$ is the only basic open of $\Vfr$ containing $x_\Vfr$.
    Then we must have $\Vfr = \Ufr_{x_\Vfr}$, so $x \mapsto \Ufr_x$ is surjective and hence bijective.

    (4). This was essentially proved in the course of proving (1).
\end{proof}

If $\Xfr$ is a toric stack with trivial torus, then \zcref{prop:pts_and_nbhds_trivial_torus} gives a natural partial order on $\Xfr(k)$: 

\begin{dfn}
    Let $\Xfr$ be a toric stack with trivial torus.
    The \emph{specialization partial order} on $\Xfr(k)$ is defined by $x \leq x'$ if and only if $\Ufr_{x} \subset \Ufr_{x'}$ in the notation of \zcref{prop:pts_and_nbhds_trivial_torus}.
\end{dfn}

\section{The local description of toric stacks} \label{sec:local_description}

Let $\Xfr$ be a toric stack with dense stacky torus $\Tfr$.
\zcref{cor:toric_stack_bottom_up} allows us to understand $\Xfr$ in terms of its equivariantization $\Xfr / \Tfr$, torsors, and gerbes.
The stacky Sumihiro theorem (\zcref{thm:stacky_sumihiro}) implies that $\Xfr / \Tfr$ is Zariski-locally basic.
Our goal in this section is to lift this result to $\Xfr$:

\begin{prop} \label{prop:toric_stack_locally_basic}
    Let $\Xfr$ be a toric stack.
    Then $\Xfr$ is a union of basic opens.
\end{prop}

We will also discuss analogous results for morphisms of toric stacks, showing that all morphisms of toric stacks are locally controlled by morphisms of basic opens.

\subsection{Gerbes on a basic toric stack with trivial torus}

Our proof of \zcref{prop:toric_stack_locally_basic} has a fairly straightforward outline: given a toric stack $\Xfr$, we consider the equivariantization morphism $q: \Xfr \to \Xfr / \Tfr$ and show that a cover of $\Xfr / \Tfr$ by basic opens (which exists by \zcref{thm:stacky_sumihiro}) pulls back to a cover of $\Xfr$ by basic opens.
The key step is to show that, if $\Ufr \subset \Xfr / \Tfr$ is a basic open (necessarily with trivial torus), then $q\inv(\Ufr)$ is still basic.
By \zcref{cor:toric_stack_bottom_up}, we may view $q\inv(\Ufr)$ as a fiber product of gerbes and torsors over $\Ufr$.
Torsors over $\Ufr$ are fairly easy to understand and classify, so we will focus for now on understanding gerbes over $\Ufr$.

Recall (see e.g.\ \cite[Definition 7.19]{geraschenko2015toric_1}) that a gerbe $p: \Yfr \to \Xfr$ is \emph{essentially trivial} if there is an equivalence of gerbes
\[
    \Yfr \simeq \sqrt[n_1]{\Lsc_1 / \Xfr} \times_\Xfr \sqrt[n_2]{\Lsc_2 / \Xfr} \times_\Xfr \dots \times_\Xfr \sqrt[n_r]{\Lsc_r / \Xfr} \times B\GG_m^{s}
\]
for some line bundles $\Lsc_i$ on $\Xfr$, root stacks $\sqrt[n_i]{\Lsc_i / \Xfr}$, and some $s \geq 0$.
Our goal in this section is to show:

\begin{prop} \label{prop:gerbe_is_ess_triv}
    Let $\Ufr$ be a basic toric stack with trivial torus, and let $D$ be a diagonalizable algebraic group.
    Then any $D$-gerbe over $\Ufr$ is essentially trivial.
\end{prop}

\begin{rmk}
    \zcref{prop:gerbe_is_ess_triv} does not extend to the case where the dense stacky torus is nontrivial.
    Indeed, the group of $\GG_m$-gerbes on a nontrivial algebraic torus $T_0$ is quite large by \cite[Corollary 7]{magid1978brauer}, so we do not have much control over $\GG_m$-gerbes even in this seemingly simple case.
\end{rmk}

Gerbes are classified by second cohomology groups, so we will begin by performing some \'etale cohomology calculations.
We refer to \cite{lopez2026brauer} for a good explanation of how to perform similar calculations.

\begin{lem} \label{lem:UmodT_Gm_cohomology}
    Let $U$ be an affine toric variety with no torus factors, and let $T \subset U$ be the dense torus.
    Then 
    \[
    H^i(U/T, \GG_m) = \begin{cases}
        k^\times & i = 0 \\
        \weight(T) & i = 1 \\
        0 & i = 2.
    \end{cases}
    \]
\end{lem}

\begin{proof}
    Consider the standard simplicial cover $U_\bullet \to U/T$, with $U_i = U \times T^i$ for all $i$.
    This gives rise to a \v{C}ech-to-cohomology spectral sequence (see \stacks{079Z})
    \[
        E_1^{p,q} = H^q(U_p, \GG_m) \Rightarrow H^{p+q}(U/T, \GG_m).
    \]
    By \cite[Corollary 10]{demeyer1992cohomological}\footnote{For those viewing the arXiv version of \cite{demeyer1992cohomological}, we recommend the arXiv's (currently experimental) HTML feature; the PDF which is currently available on the arXiv is difficult to read due to compilation errors in the \LaTeX.}, we have $H^q(U_p, \GG_m) = H^q(T^p, \GG_m)$ for all $p, q \geq 0$.
    In particular:
    \begin{itemize}
        \item For $q = 0$ and any $p$, we have $H^0(U_p, \GG_m) = H^0(T^p, \GG_m) = k^\times \oplus (\weight(T))^{\oplus p}$.
        \item For $q = 1$ and any $p$, we have $H^1(U_p, \GG_m) = H^1(T^p, \GG_m) = \Pic(T^p) = 0$.
        \item For $q = 2$ and $p = 0$, we have $H^2(U_0, \GG_m) = H^2(\Spec k, \GG_m) = 0$.
    \end{itemize}
    The low degree terms of the $E_1$ page of this spectral sequence look like
    \[
    \begin{tikzcd}
        0 \rar & H^2(T, \GG_m) \rar & H^2(T^2, \GG_m) \rar & \dots \\
        0 \rar & 0 \rar & 0 \rar & \dots \\
        k^\times \rar["d_1^{0,0}"] & k^\times \oplus \weight(T) \rar["d_1^{1,0}"] & k^\times \oplus \weight(T)^2 \rar["d_1^{2,0}"] & \dots
    \end{tikzcd}
    \]
    To compute the $E_2$ page in low degree, we note that $d_1^{0,p}$ is given by the alternating sum of the maps $k^\times \oplus \weight(T)^p \to k^\times \oplus \weight(T)^{p+1}$ induced by pullback along the projection and multiplication maps $T^{p+1} \to T^p$.
    Thus:
    \begin{itemize}
        \item For $p = 0$ and $q = 0$, the projection and multiplication maps $T \to \Spec k$ are the same, so the alternating sum of the corresponding pullback maps vanishes.
        Thus $d_1^{0,0} = 0$ and $E_2^{0,0} = \ker d_1^{0,0} = k^\times$.
        \item For $p = 1$ and $q = 0$, the projection maps $p_1, p_2: T^2 \to T$ and the multiplication map $m: T^2 \to T$ give
        \begin{align*}
            d_1^{1,0}(c, \chi_1) &= p_2^* (c, \chi_1) - m^*(c, \chi_1) + p_1^*(c, \chi_1) \\
            &= (c, 0, \chi_1) - (c, \chi_1, \chi_1) + (c, \chi_1, 0) \\
            &= (c, 0, 0).
        \end{align*}
        Thus $\ker d_1^{1,0} = \bset{(0, \chi_1)}{\chi_1 \in \weight(T)} = \weight(T)$.
        Because $\im d_1^{0,0} = 0$, we get $E_2^{1,0} = \weight(T)$.
        \item For $p = 2$ and $q = 0$, a similar computation shows that we have
        \begin{align*}
            d_1^{2,0}(c, \chi_1, \chi_2) &= (c, 0, \chi_1, \chi_2) - (c, \chi_1, \chi_1, \chi_2) + (c, \chi_1, \chi_2, \chi_2) - (c, \chi_1, \chi_2, 0)\\
            &= (0, -\chi_1, 0, \chi_2).
        \end{align*}
        Thus $\ker d_1^{2,0} = \bset{(c,0,0)}{c \in k^\times} = \im d_1^{2,0}$, and we get $E_2^{2,0} = 0$.
    \end{itemize}
    We see that our spectral sequence has $E_2$ page
    \[
        \begin{tikzcd}
            0 \ar[drr] & E_2^{1,2} \ar[drr] & E_2^{2,2} & \dots \\
            0 \ar[drr] & 0 \ar[drr] & 0 & \dots \\
            k^\times & \weight(T) & 0 & \dots
        \end{tikzcd}
    \] 
    In low degrees (specifically for $p + q \leq 2$), there are no nontrivial differentials in the $E_2$ page or later pages.
    Thus our spectral sequence collapses at the $E_2$ page in these degrees.
    Since the spectral sequence converges to $H^{p+q}(U/T, \GG_m)$, we win.
\end{proof}

We now use \zcref{lem:UmodT_Gm_cohomology} to classify gerbes over a basic toric stack with trivial torus.

\begin{lem} \label{lem:gerbe_over_basic_classification}
    Let $\Ufr$ be a basic toric stack with trivial torus.
    Then:
    \begin{enumerate}
        \item Every $\GG_m$-gerbe over $\Ufr$ is trivial.
        \item Every $\mu_n$-gerbe over $\Ufr$ (for $n > 0$) is a root stack $\sqrt[n]{\Lsc / \Ufr}$ for some $\Lsc \in \Pic(\Ufr)$.
        Furthermore, we have $\sqrt[n]{\Lsc / \Ufr} \simeq \sqrt[n]{\Lsc' / \Ufr}$ as gerbes if and only if $\Lsc \otimes (\Lsc')\inv$ has an $n$th root in $\Pic(\Ufr)$.
    \end{enumerate}
\end{lem}

\begin{proof}
    We can write $\Ufr = U / T$ where $U$ is an affine toric variety with no torus factors and $T \subset U$ is the dense torus.
    Here $\Pic(\Ufr) = \weight(T)$.
    Claim (1) now follows from \zcref{lem:UmodT_Gm_cohomology} and the fact that $H^2(-, \GG_m)$ classifies $\GG_m$-gerbes.

    For (2), fix $n > 0$ and consider the long exact sequence in $H^\bullet(\Ufr, -)$ arising from the Kummer sequence
    \[
        \begin{tikzcd}
            1 \rar & \mu_n \rar & \GG_m \rar["z \mapsto z^d"] & \GG_m \rar & 1.
        \end{tikzcd}
    \]
    By \zcref{lem:UmodT_Gm_cohomology}, near the $H^2(\Ufr, \mu_n)$ term, this long exact sequence is
    \[
    \begin{tikzcd}
        \dots \rar & \Pic(\Ufr) \rar["(-)^{\otimes n}"] \rar & \Pic(\Ufr) \rar & H^2(\Ufr, \mu_n) \rar & 0.
    \end{tikzcd}
    \]
    Thus $H^2(\Ufr, \mu_n) = \Pic(\Ufr) / n \Pic(\Ufr)$.
    If we think of $H^2(\Ufr, \mu_n)$ as classifying $\mu_n$-gerbes, then the natural map $\Pic(\Ufr) \to H^2(\Ufr, \mu_n)$ is given by $\Lsc \mapsto \sqrt[n]{\Lsc / \Ufr}$ (see e.g.\ \cite[Appendix B.1]{abramovich2008gromov}).
    Since this map is surjective and has kernel $n \Pic(\Ufr)$, claim (2) follows.
\end{proof}

\begin{proof}[Proof of \zcref{prop:gerbe_is_ess_triv}]
    Choose a decomposition $D \cong \mu_{n_1} \times \dots \times \mu_{n_r} \times \GG_m^s$ for some $n_i$ and $s$.
    If $p: \Ufr' \to \Ufr$ is a $D$-gerbe, then this decomposition splits $\Ufr'$ into a fiber product of $\mu_{n_i}$-gerbes and $\GG_m$-gerbes over $\Ufr$.
    In both cases, these gerbes are essentially trivial by \zcref{lem:gerbe_over_basic_classification}.
    Thus $p: \Ufr' \to \Ufr$ is also essentially trivial.
\end{proof}

\subsection{Bootstrapping basicness} \label{sub:bootstrapping_basicness}

We want to use \zcref{prop:gerbe_is_ess_triv} (and the easier classification of torsors on a basic toric stack) to show the following.

\begin{prop} \label{prop:basic_quotient_implies_basic}
    Let $\Ufr$ be a toric stack with dense stacky torus $\Tfr$.
    Then $\Ufr$ is basic if and only if $\Ufr / \Tfr$ is basic.
\end{prop}

For the proof, we will first require two lemmas allowing us to view torsors and essentially trivial gerbes as basic toric stacks.

\begin{lem} \label{lem:torsor_is_basic}
    Let $\Ufr$ be a basic toric stack with trivial generic stabilizer, and let $T$ be an algebraic torus.
    Then (the total space of) any $T$-torsor $\Ufr'$ over $\Ufr$ can be equipped with the structure of a basic toric stack such that the projection map $p: \Ufr' \to \Ufr$ is a map of toric stacks.
    Furthermore, up to the choice of a $k$-point in the fiber of $p$ over the unit of $\Ufr$, this is the unique toric stack structure on $\Ufr'$ such that $p$ is toric.
\end{lem}

\begin{proof}
    For existence, write $\Ufr = U / G$ as in \zcref{dfn:basic}.
    Arguing inductively on the rank of $T$, we reduce to the case where $T$ has rank one, i.e.\ $T \cong \GG_m$.
    In this case we note that $\Pic(U) = 0$ by \cite[Proposition 4.2.2]{cls}, so every $G$-equivariant line bundle on $U$ has trivial underlying line bundle.
    Thus every $G$-equivariant line bundle on $U$ is of the form $\Osc_U(\chi)$ for some (possibly non-unique) $\chi \in \weight(G)$.\footnote{For an example of non-uniqueness, note that every character of $\GG_m$ gives rise to the same $\GG_m$-equivariant line bundle on $\GG_m$.}
    In terms of $\GG_m$-torsors, this means we can write $\Ufr' = (U \times \GG_m) / G$ where $G$ acts on $U$ as it does in our presentation of $\Ufr$ and $G$ acts on $\GG_m$ via $\chi$.
    Thus $\Ufr$ is a basic toric stack.

    For uniqueness, let $e \in \Ufr(k)$ be the unit of the dense torus of $\Ufr$.
    A toric stack structure on $\Ufr'$ compatible with $p$ is determined by an algebraic torus structure on the scheme $p\inv(e)$ and an action of $p\inv(e)$ on $\Ufr'$ together with data controlling compatibility of this structure with $p$.
    A choice of $k$-point in $p\inv(e)$ defines an isomorphism $p\inv(e) \cong T$ and thus an algebraic torus structure on $p\inv(e)$.
    The action of $p\inv(e)$ must then be defined on the fibers of $p$ by the corresponding $T$-action. 
    Since the global action of $p\inv(e)$ on $\Ufr'$ is determined by its action on fibers, we see that the toric stack structure on $\Ufr'$ is unique given the choice of $k$-point.
\end{proof}

\begin{lem} \label{lem:ess_triv_gerbe_is_basic}
    Let $\Ufr$ be a basic toric stack and $D$ a diagonalizable group.
    Then any essentially trivial $D$-gerbe $\Ufr'$ over $\Ufr$ can be equipped with the structure of a basic toric stack such that the projection map $p: \Ufr' \to \Ufr$ is a map of toric stacks.
    Furthermore, this is the unique toric stack structure on $\Ufr'$ such that $p$ is toric.
\end{lem}

\begin{proof}
    Write $\Ufr = U / G$ as in \zcref{dfn:basic}.
    For the existence claim, by essential triviality of $\Ufr'$ and induction on the number of factors in a decomposition 
    \[
        D \cong \mu_{n_1} \times \dots \mu_{n_r} \times \GG_m \times \dots \times \GG_m,
    \]
    we can reduce to the following two cases:
    \begin{itemize}
        \item $D = \mu_n$ for some $n$ and $\Ufr' = \sqrt[n]{\Lsc / \Xfr}$ for some line bundle $\Lsc$ on $\Ufr$.
        In this case the result follows from \cite[Lemma 7.1(2)]{fantechi2010smooth} (the Deligne--Mumford hypothesis in the result stated there is not necessary for the argument to work).
        \item $D = \GG_m$ and $\Ufr' = \Ufr \times B\GG_m$.
        In this case $\Ufr' = U / (G \times \GG_m)$ where $\GG_m$ acts trivially on $U$, so the claim is obvious.
    \end{itemize}
    The uniqueness claim follows as in \cite[Theorem 6.25(1)]{fantechi2010smooth}.
\end{proof}

\begin{proof}[{Proof of \zcref{prop:basic_quotient_implies_basic}}]
    ($\Rightarrow$) This is easy: if $\Ufr = U / G$ and $T$ is the dense torus of $U$, then $\Tfr = T / G$ and so 
    \[
        \Ufr / \Tfr = (U / G) / (T / G) = U / T.
    \]

    ($\Leftarrow$).
    Consider the Cartesian square of \zcref{cor:toric_stack_bottom_up}:
    \[
    \begin{tikzcd}
        \Ufr \rar["p_2'"] \dar["p_1'"] & \Ufr_1 \dar["p_1"] \\
        \Ufr_2 \rar["p_2"] & \Ufr / \Tfr.
    \end{tikzcd}
    \]
    Here $p_1: \Ufr_1 \to \Ufr / \Tfr$ is a $T$-gerbe for some algebraic torus $T$ and $p_2: \Ufr_2 \to \Ufr / \Tfr$ is a $D$-gerbe for some diagonalizable algebraic group $D$.

    The unit of $\Tfr$ maps to a $k$-point of $\Ufr_1$ in the fiber over the unit of the dense stacky torus of $\Ufr / \Tfr$.
    By the existence part of \zcref{lem:torsor_is_basic}, this $k$-point in $\Ufr_1$ allows us to view the underlying stack of $\Ufr_1$ as a basic toric stack over $\Ufr / \Tfr$.
    The uniqueness part of \zcref{lem:torsor_is_basic} implies that this structure agrees with the original toric stack structure on $\Ufr_1$.
    Thus $\Ufr_1$ is basic.

    The $D$-gerbe $p_2: \Ufr_2 \to \Ufr / \Tfr$ is essentially trivial by \zcref{prop:gerbe_is_ess_triv}.
    Thus the pullback gerbe $p_2': \Ufr \to \Ufr_1$ must also be essentially trivial.
    By \zcref{lem:ess_triv_gerbe_is_basic}, this implies the underlying stack of $\Ufr$ can be viewed as a basic toric stack in such a way that $p_2'$ is toric.
    The uniqueness part of \zcref{lem:ess_triv_gerbe_is_basic} implies that this toric stack structure agrees with the original toric stack structure on $\Ufr$.
    Thus $\Ufr$ is basic.
\end{proof}

\subsection{Proof of \zcref{prop:toric_stack_locally_basic}}

With the material of the preceding subsections, we are now ready to prove:

\begin{proof}[Proof of \zcref{prop:toric_stack_locally_basic}]
    Let $\Tfr$ be the dense stacky torus of $\Xfr$, and let $q: \Xfr \to \Xfr / \Tfr$ be the equivariantization morphism.
    By \zcref{prop:pts_and_nbhds_trivial_torus}, $\Xfr / \Tfr$ is a union of basic opens $\{\ol{\Ufr}_{[x]}\}_{[x] \in (\Xfr / \Tfr)(k)}$.
    Let $\Ufr_{[x]} = q\inv(\ol{\Ufr}_{[x]}) \subset \Xfr$ for all $[x] \in (\Xfr / \Tfr)(k)$.
    Then $\Xfr$ is the union of all $\Ufr_{[x]}$, so it suffices to show each $\Ufr_{[x]}$ is basic.
    For all $[x]$, we have $\ol{\Ufr}_{[x]} = \Ufr_{[x]} / \Tfr$ because equivariantization is a local construction for $\Tfr$-equivariant charts.
    Thus $\Ufr_{[x]} / \Tfr$ is basic, so $\Ufr_{[x]}$ is basic by \zcref{prop:basic_quotient_implies_basic}.
\end{proof}

\subsection{The local description of morphisms}

Now that we know every toric stack is a union of basic opens, it is natural to wonder whether toric morphisms can be described in terms of morphisms of basic opens.
This is indeed true:

\begin{prop} \label{prop:basic_maps_to_basic}
    Let $\Xfr$ be a toric stack, let $\Ufr'$ be a basic toric stack, and let $f: \Ufr' \to \Xfr$ be a morphism.
    Then there exists a basic open $\Ufr \subset \Xfr$ such that $f(\Ufr') \subset \Ufr$.
\end{prop}

\begin{proof}
    Let:
    \begin{itemize}
        \item $\Tfr'$ be the dense stacky torus of $\Ufr$,
        \item $\Tfr$ be the dense stacky torus of $\Xfr$,
        \item $q': \Ufr' \to \Ufr' / \Tfr$ be the quotient morphism of $\Ufr'$, and
        \item $q: \Xfr \to \Xfr / \Tfr$ be the quotient morphism of $\Xfr$.
    \end{itemize}
    Passing to quotients gives a commutative diagram
    \[
    \begin{tikzcd}
        \Ufr' \rar["f"] \dar["q'"] & \Xfr \dar["q"] \\
        {\Ufr' / \Tfr'} \rar["\ol{f}"] & {\Xfr / \Tfr}.
    \end{tikzcd}
    \]
    Let $[u'] \in (\Ufr' / \Tfr')(k)$ be the maximal point with respect to the specialization partial order.
    By \zcref{thm:stacky_sumihiro}, $\ol{f}([u'])$ is contained in a basic open $\ol{\Ufr} \subset \Xfr / \Tfr$.
    Maximality of $[u']$ implies that $\ol{f}(\Ufr' / \Tfr') \subset \ol{\Ufr}$.
    Let $\Ufr = q\inv(\ol{\Ufr}) \subset \Xfr$, so $\ol{\Ufr} = \Ufr / \Tfr$.
    Then $f(\Ufr') \subset \Ufr$, and $\Ufr$ is basic by \zcref{prop:basic_quotient_implies_basic}.
\end{proof}

\begin{cor} \label{cor:toric_stack_maps_lim_colim}
    Let $\Xfr$ and $\Xfr'$ be toric stacks.
    Then
    \[
        \Map_\TStk(\Xfr, \Xfr') = \lim_{\Ufr \subset \Xfr} \colim_{\Ufr' \subset \Xfr'} \Map_{\TStk}(\Ufr, \Ufr')
    \]
    where the limit and colimit are taken over basic opens in $\Xfr$ and $\Xfr'$, respectively.
\end{cor}

\begin{proof}
    Morphisms of toric stacks are defined locally, so by \zcref{prop:toric_stack_locally_basic} we can write
    \[
        \Map_\TStk(\Xfr, \Xfr') = \lim_{\Ufr \subset \Xfr} \Map_{\TStk}(\Ufr, \Xfr').
    \]
    Thus it suffices to prove that if $\Ufr$ is basic, then 
    \[
        \Map_{\TStk}(\Ufr, \Xfr') = \colim_{\Ufr' \subset \Xfr'} \Map_{\TStk}(\Ufr, \Ufr').
    \]
    But this follows from \zcref{prop:basic_maps_to_basic}.
\end{proof}

Another consequence of \zcref{prop:basic_maps_to_basic} is the following ``covering'' characterization of basic toric stacks.

\begin{cor}
    Let $\Xfr$ be a toric stack.
    Then the following are equivalent:
    \begin{enumerate}
        \item $\Xfr$ is basic.
        \item There exists a basic toric stack $\Ufr$ and a surjective morphism of toric stacks $\Ufr \to \Xfr$.
        \item There exists an affine toric variety $U$ and a surjective morphism of toric stacks $U \to \Xfr$.
    \end{enumerate}
\end{cor}

\begin{proof}
    ($1 \Rightarrow 2$).
    Obvious.
    
    ($2 \Rightarrow 3$).
    Obvious.

    ($3 \Rightarrow 1$).
    If $f: U \to \Xfr$ is a morphism as in (3), then \zcref{prop:basic_maps_to_basic} implies $f(U)$ is contained in a basic open $\Ufr \subset \Xfr$.
    Surjectivity of $f$ then forces $\Ufr = \Xfr$, so $\Xfr$ is basic.
\end{proof}

\section{Geometrization is an equivalence} \label{sec:proof_of_main_thm}

In this section, we will combine the results of the previous sections to show that geometrization gives an equivalence between $\StkFan$ and $\TStk$.

\begin{thm} \label{thm:classification}
    The geometrization functor $\Xfr_{-}: \StkFan \to \TStk$ is an equivalence of 2-categories.
\end{thm}

\begin{proof}
    Fully faithfulness is \zcref{prop:geometrization_fully_faithful} below, and essential surjectivity is \zcref{prop:geometrization_ess_surj} below.
\end{proof}

\subsection{Fully faithfulness} \label{sub:fully_faithful}

To prove that geometrization is a fully faithful functor, first consider the basic case.
Given a stacky cone $K$, we seek to recover $K$ from the toric stack structure on $\Xfr_K$.
Actually, because we defined $\Xfr_K$ in terms of $D_\Gale(\Oc(K))$, it will be easier for us to focus on recovering $D_\Gale(\Oc(K))$ from $\Xfr_K$.
The following construction will allow us to do so.

\begin{dfn}
    Let $\Xfr$ be a toric stack.
    Define a symmetric monoidal category $\CLine(\Xfr)$, the \emph{combinatorial line category} of $\Xfr$, by\footnote{From the discussion in \zcref{sub:reconstruction}, we see that the convolution monoidal structure on $\QCoh(\Tfr)$ induces comonoid structures on the hom-sets in $\Line(\QCoh(\Tfr))$ and thus a notion of ``setlike morphism'' in this category.}
    \[
        \CLine(\Xfr) = \Line(\QCoh(\Xfr)) \times_{\Line(\QCoh(\Tfr))} \Line(\QCoh(\Tfr))^\setlike.
    \]
    This is contravariantly functorial in $\Xfr$.
\end{dfn}

Morally, $\CLine(\Xfr)$ is the non-full subcategory of $\QCoh(\Xfr)$ spanned by the invertible sheaves with only monomials as morphisms.

\begin{rmk}
    We will show later (in \zcref{prop:CLine_vs_DGale}) that, for any toric stack $\Xfr$, $\CLine(\Xfr)$ is always a toric category.
    Throughout \zcref{sec:picard_groups}, we will discuss more practical ways to compute $\CLine(\Xfr)$ and variants thereof.
    For now, we will focus on understanding $\CLine(\Ufr)$ when $\Ufr$ is a basic toric stack.
\end{rmk}

Using the $\CLine$ construction and \zcref{prop:linear_reconstruction_algebraic}, we obtain a way to reconstruct a stacky cone from its basic toric stack.

\begin{lem} \label{lem:linear_reconstruction_geometric}
    Let $K$ be a stacky cone.
    Then $\CLine(\Xfr_K) = \GDual(\Oc(K))$.
\end{lem}

\begin{proof}
    Applying \zcref{prop:linear_reconstruction_algebraic} to $\GDual(\Oc(K))$ shows
    \begin{equation} \label{eq:linear_reconstruction_intermediate}
        \GDual(\Oc(K)) = \Line\big(\PSh_k(\GDual(\Oc(K)))\big) \times_{\Line(\PSh_k(\GDual(\Oc(K))^\gpd))} \Line\big(\PSh_k(\GDual(\Oc(K))^\gpd))\big)^\setlike.
    \end{equation}
    By \zcref{prop:geometrization_of_cone}(1), we have
    \[
        \Big(\PSh_k\big(\GDual(\Oc(K))\big), \star_{\GDual(\Oc(K))}\Big) = \big(\QCoh(\Xfr_K), \otimes_{\Osc}\big).
    \]
    Furthermore, $\GDual(\Oc(K))^\gpd = \GDual(\iota_0 \Oc(K)) = \GDual\big(\Oc(K(0))\big)$ by \zcref{prop:gale_duality_localization}.
    Because the dense stacky torus $\Tfr_K$ is $\Xfr_{K(0)}$, \zcref{prop:geometrization_of_cone}(1) gives
    \[
        \Big(\PSh_k\big(\GDual(\Oc(K)^\gpd)\big), \star_{\GDual(\Oc(K)^\gpd)}\Big) = \big(\QCoh(\Tfr_K), \otimes_{\Osc}\big).
    \]
    Plugging this into \eqref{eq:linear_reconstruction_intermediate} and comparing with the definition of $\CLine(\Xfr_K)$ gives the claim.
\end{proof}

\begin{prop} \label{prop:geometrization_fully_faithful_cone}
    Geometrization induces an equivalence $\Xfr_{-}: \StkCone \simto \TStk^\basic$ with quasi-inverse $\Ufr \mapsto \Cone(\GDual(\CLine(\Ufr)))$.
    In particular, $\Xfr_{-}$ is fully faithful when viewed as a functor on stacky cones.
\end{prop}

\begin{proof}
    Write $F: \StkCone \to \TStk^\basic$ for the functor $\Ufr \mapsto \Cone(\GDual(\CLine(\Ufr)))$.
    We need to show $F$ defines a quasi-inverse to the geometrization functor $\Xfr_{-}$.

    For one direction, suppose $K \in \StkCone$.
    Then 
    \begin{align*}
        F(\Xfr_K) &= \Cone(\GDual(\CLine(\Xfr_K))) \\
        &= \Cone\big((\GDual)^2(\Oc(K))\big) \textrm{ by \zcref{lem:linear_reconstruction_geometric}} \\
        &= \Cone(\Oc(K)) \textrm{ by \zcref{thm:gale_duality}} \\
        &= K \textrm{ by \zcref{cor:stacky_cone_equivalence}.}
    \end{align*}
    
    For the reverse direction, suppose we are given a basic toric stack $\Ufr$.
    Then $\Ufr = \Xfr_K$ for some stacky cone $K$ by \zcref{prop:geometrization_of_cone}(3), so $\Xfr_{F(\Ufr)} = \Xfr_{F(\Xfr_K)} = \Xfr_K = \Ufr$
    as needed.
\end{proof}

\begin{rmk}
    The functor $\CLine$ is \emph{not} fully faithful when considered as a functor on arbitrary toric stacks.
    For example, one can compute that
    \[
        \CLine(\PP^1) = \ZZ \vec{/}_{[1,1]} \NN^2 = \CLine([\AA^2 / \GG_m]),
    \]
    so that
    \[
        \Cone(\GDual(\CLine(\PP^1))) = \Cone(\GDual(\CLine([\AA^2 / \GG_m]))).
    \]
    This is unsurprising if one thinks of basic toric stacks as analogous to affine schemes: the functor $\Osc(-)$ is not fully faithful on the category of all schemes, but it is fully faithful on the category of affine schemes.
\end{rmk}

Despite this, we can extend the claim about fully faithfulness of $\Xfr_{-}$ in \zcref{prop:geometrization_fully_faithful_cone} to arbitrary stacky fans:

\begin{prop} \label{prop:geometrization_fully_faithful}
    The geometrization functor $\Xfr_{-}: \StkFan \to \TStk$ is fully faithful.
\end{prop}

\begin{proof}
    Let $\Sigma$ and $\Sigma'$ be stacky fans.
    By \zcref{cor:toric_stack_maps_lim_colim}, we can write
    \[
        \Map_\TStk(\Xfr_K, \Xfr_{\Sigma'}) = \colim_{K' \subset \Sigma} \Map_\TStk(\Xfr_K, \Xfr_{K'}). 
    \]
    Now observe that
    \begin{align*}
        \Map_\TStk(\Xfr_\Sigma, \Xfr_{\Sigma'}) &= \lim_{K \subset \Sigma} \colim_{K' \subset \Sigma} \Map_\TStk(\Xfr_K, \Xfr_{K'}) \textrm{ by the above}. \\
        &= \lim_{K \subset \Sigma} \colim_{K' \subset \Sigma} \Map_{\StkCone}(K, K') \textrm{ by \zcref{prop:geometrization_fully_faithful_cone}} \\
        &= \Map_{\StkFan}(\Sigma, \Sigma') \textrm{ by \zcref{prop:lim_colim_maps_in_StkFan}.} \qedhere
    \end{align*}
\end{proof}

\subsection{Cocharacter categories}

Combining the equivalences of \zcref{prop:geometrization_fully_faithful_cone} and \zcref{cor:stacky_cone_equivalence} yields an equivalence $\TStk^\basic \simto \TCat$.
Our construction of this equivalence is somewhat inexplicit.
We can remedy this by introducing the following construction:

\begin{dfn}
    Let $\Ufr$ be a basic toric stack.
    The \emph{cocharacter category} of $\Ufr$ is the toric category 
    \[
        \coweightcat(\Ufr) := \Map_\TStk(\GG_m, \Ufr) \vec{/}_{j_{\AA^1}^*} \Map_\TStk(\AA^1, \Ufr),
    \]
    where the homomorphism $j_{\AA^1}^*$ is restriction along the inclusion $j_{\AA^1}: \GG_m \hookrightarrow \AA^1$.
\end{dfn}

We will see that $\coweightcat$ is well-defined in the course of proving:

\begin{prop} \label{prop:cocharacter_equivalence}
    Let $K$ be a stacky cone.
    Then:
    \begin{enumerate}
        \item $\Map_\TStk(\GG_m, \Xfr_K) = \iota_0 \Oc(K)$,
        \item $\Map_\TStk(\AA^1, \Xfr_K) = \PO \Oc(K)$, and
        \item $\coweightcat(\Xfr_K) = \Oc(K)$.
    \end{enumerate}
    Thus $\coweightcat: \TStk^\basic \to \TCat$ is the inverse of the equivalence $\Xfr_{\Cone(-)}: \TCat \simto \TStk^\basic$. 
\end{prop}

\begin{proof}
    (1).
    This is a direct computation:
    \begin{align*}
        \Map_\TStk(\GG_m, \Xfr_K) &= \Map_\TStk(\Xfr_{\Cone(\ZZ)}, \Xfr_K) \\
        &= \Map_\StkCone(\Cone(\ZZ), K) \textrm{by \zcref{prop:geometrization_fully_faithful_cone}} \\
        &= \iota_0 \Fun^\oplus(\ZZ, \Oc(K)) \textrm{ by \zcref{cor:stacky_cone_equivalence}} \\
        &= \iota_0 \Oc(K) \textrm{ by \zcref{ex:i0_corepresented}.}
    \end{align*}

    (2).
    This is essentially the same computation as in (1), except we use $\AA^1 = \Xfr_{\Cone(\ZZ \vec{/} \NN)}$ and replace the reference to \zcref{ex:i0_corepresented} by a reference to \zcref{ex:P0_corepresented}.

    (3).
    The natural inclusion $j_{\AA^1}: \GG_m \hookrightarrow \AA^1$ is induced by the skew quotient map $\ZZ \to \ZZ \vec{/} \NN$, so by parts (1) and (2) we see that $\coweightcat(\Xfr_K)$ is the skew quotient of the usual codomain map $\PO \Oc(K) \to \iota_0 \Oc(K)$.
    The result then follows from \zcref{thm:TCat_structure_thm}.
\end{proof}

\begin{rmk} \label{rmk:explicit_mult_on_cocharacter_cat}
    Here is a more explicit description of the multiplication on $\Map_\TStk(\AA^1, \Ufr)$ when $\Ufr$ is a basic toric stack.
    Combining \zcref{prop:finite_product_coproduct_TCat} and \zcref{prop:geometrization_fully_faithful_cone}, we see that finite products and finite coproducts agree in $\TStk^\basic$.
    Thus, given $f_1, f_2 \in \Map_\TStk(\AA^1, \Ufr)$, we obtain a coproduct morphism\footnote{This relies crucially on the assumption that $\Ufr$ is basic: the morphisms $\AA^1 \to \PP^1$ given by $z \mapsto [1 : z]$ and $z \mapsto [z : 1]$ cannot be combined into a reasonable morphism $\AA^2 \to \PP^1$.} $f: \AA^2 \to \Ufr$ satisfying $f|_{\AA^1 \times \{1\}} = f_1$ and $f|_{\{1\} \times \AA^1} = f_2$.
    Restricting this along the diagonal $\Delta_{\AA^1}: \AA^1 \to \AA^2$ gives a morphism $f_1 \cdot f_2 \in \Map_\TStk(\AA^1, \Ufr)$.
    
    If $\Ufr = \Xfr_K$ for some stacky cone $K$, one can check that this construction agrees with the standard monoid operation on $\PO \Oc(K)$, and also with the multiplication induced by the commutative 2-monoid structure on $\Ufr$ outlined before the proof of \zcref{prop:geometrization_of_cone_functorial}.
    We omit the proof, which is largely an exercise in tracing through the definitions of the equivalences involved.
    
    We can also provide a similar description for the multiplication on $\Map_\TStk(\GG_m, \Ufr)$.
    In fact, unlike the description of the multiplication on $\Map_\TStk(\AA^1, \Ufr)$, this does not require $\Ufr$ to be basic: since any morphism from toric $\GG_m$ to a toric stack lands in the dense stacky torus of the codomain, the product morphisms $f_1 \cdot f_2$ are always well-defined.
\end{rmk}

Cocharacter categories are useful for directly relating properties of a basic toric stack $\Ufr$ to properties of the corresponding toric category.
As an example application of this, we prove the following result relating open immersions of stacky cones to open immersions of stacky fans.
We will generalize this in \zcref{prop:open_imm_biconditional}.

\begin{lem} \label{lem:geometrization_cone_immersion_converse}
    Let $j: \Ufr' \to \Ufr$ be an open immersion of basic toric stacks.
    Then $j$ is the geometrization of an open immersion of stacky cones.
\end{lem}

\begin{proof}
    By \zcref{prop:cocharacter_equivalence}, it suffices to show that $\coweightcat(j): \coweightcat(\Ufr') \to \coweightcat(\Ufr)$ induces:
    \begin{itemize}
        \item An isomorphism $\Hom_{\TStk}(\GG_m, \Ufr') \simto \Hom_{\TStk}(\GG_m, \Ufr)$, and
        \item A face inclusion $\Hom_{\TStk}(\AA^1, \Ufr') \to \Hom_{\TStk}(\AA^1, \Ufr)$.
    \end{itemize}
    The former is obvious: if $\Tfr$ is the dense stacky torus of $\Ufr'$ and $\Ufr$, then
    \[
        \Hom_{\TStk}(\GG_m, \Ufr') = \Hom_{\TStk}(\GG_m, \Tfr) = \Hom_{\TStk}(\GG_m, \Ufr).
    \]
    
    For the latter, it is clear that $\Hom_{\TStk}(\AA^1, \Ufr')$ is a submonoid of $\Hom_{\TStk}(\AA^1, \Ufr)$.
    To see that it is in fact a face, suppose $f_1, f_2 \in \Hom_{\TStk}(\AA^1, \Ufr)$ satisfy $f_1 \cdot f_2 \in \Hom_{\TStk}(\AA^1, \Ufr')$.
    As in \zcref{rmk:explicit_mult_on_cocharacter_cat}, this implies that the coproduct morphism $f: \AA^2 \to \Ufr$ sends $(0, 0) \in \AA^2$ to a point of $\Ufr'$.

    We claim that the morphism $f$ lands in $\Ufr' \subset \Ufr$.
    Indeed, for all $x \in \AA^2(k)$, the torus orbit closure $\ol{\GG_m^2 \cdot x}$ contains $(0, 0)$.
    Thus $f(0, 0)$ must be contained in the smallest $\Tfr$-equivariant reduced closed substack of $\Ufr$ containing $f(x)$.
    If $f(x) \not\in \Ufr'(k)$, then $f(0, 0) \in (\Ufr \setminus \Ufr')(k)$, a contradiction.
    Thus we must have $f(x) \in \Ufr'(k)$.
    Because every point of $\AA^2$ specializes to a $k$-point, it follows that $f$ maps $\AA^2$ to $\Ufr'$.

    Restricting $f$ to $\AA^1 \times \{1\}$ and $\{1\} \times \AA^1$, we see that $f_1$ and $f_2$ take values in $\Ufr'$.
    Thus $\Hom_{\TStk}(\AA^1, \Ufr')$ is a face of $\Hom_{\TStk}(\AA^1, \Ufr)$, as needed.
\end{proof}

\subsection{Essential surjectivity}

We can prove that the geometrization functor $\Xfr_{-}: \StkFan \to \TStk$ is essentially surjective by combining several of our previous results:

\begin{prop} \label{prop:geometrization_ess_surj}
    Let $\Xfr$ be a toric stack.
    Then $\Xfr = \Xfr_\Sigma$ for some stacky fan $\Sigma$.
\end{prop}

\begin{proof}
    Let $\Tfr$ be the dense stacky torus of $\Xfr$, and consider the equivariantization morphism $p: \Xfr \to \Xfr / \Tfr$.
    Let $|\Sigma| = (\Xfr / \Tfr)(k)$ with the specialization partial order.
    By \zcref{prop:pts_and_nbhds_trivial_torus}, $\Xfr / \Tfr$ is a union of basic opens $\ol{\Ufr}_\sigma$ labeled by $\sigma \in |\Sigma|$.
    \zcref{prop:basic_quotient_implies_basic} implies that all preimages $\Ufr_\sigma = p\inv(\ol{\Ufr}_\sigma)$ are basic because each equivariantization $\Ufr_\sigma / \Tfr = \ol{\Ufr}_\sigma$ is basic.
    In particular, for all $\sigma \in \Sigma$, there is a stacky cone $\Sigma_{\leq \sigma}$ such that $\Ufr_\sigma = \Xfr_{\Sigma_{\leq \sigma}}$.
    
    We have written $\Xfr$ as a union / colimit of a diagram of open immersions of basic toric stacks $\Xfr_{\Sigma_{\leq \sigma}}$.
    By \zcref{prop:geometrization_fully_faithful_cone} and \zcref{lem:geometrization_cone_immersion_converse}, this is the same as saying that $\Xfr$ is obtained by gluing together the basic toric stacks $\Xfr_{\Sigma_{\leq \sigma}}$ along open immersions $\Sigma_{\leq \sigma} \hookrightarrow \Sigma_{\leq \sigma'}$.
    Letting $\Sigma = \cup_{\sigma \in |\Sigma|} \Sigma_{\leq \sigma}$, we see that $\Xfr = \Xfr_\Sigma$, as desired.
\end{proof}

\subsection{Remarks on geometric Gale duality}

By combining several of the equivalences we have constructed thus far, we may produce an interesting autoequivalence of $\TStk^\basic$:

\begin{thm}[Geometric Gale duality] \label{thm:geometric_gale_duality}
    There exists a unique antiequivalence $\GeomDual: (\TStk^\basic)\op \simto \TStk$ satisfying:
    \begin{enumerate}
        \item $(\GeomDual)^2 = \id$,
        \item $\GeomDual(\GG_m) = B\GG_m$,
        \item $\GeomDual(\AA^1) = \AA^1 / \GG_m$, and
        \item $\GeomDual$ sends the inclusion map $\GG_m \hookrightarrow \AA^1$ to the natural map $\AA^1 / \GG_m \to B\GG_m$.
    \end{enumerate}
    Furthermore, $\GeomDual$ interchanges combinatorial line categories and cocharacter categories: if $\Ufr$ is a basic toric stack, then
    \[
        \coweightcat(\GeomDual(\Ufr)) = \CLine(\Ufr) \textrm{ and } \CLine(\GeomDual(\Ufr)) = \coweightcat(\Ufr).
    \]
\end{thm}

\begin{proof}
    The antiequivalence $\GeomDual$ is defined by $\GeomDual(\Xfr_{\Cone(\Cc)}) = \Xfr_{\Cone(\GDual(\Cc))}$.
    The axioms for $\GeomDual$ are obtained by geometrizing those listed in \zcref{thm:gale_duality}.
    To see that $\GeomDual$ interchanges combinatorial line categories and cocharacter categories, suppose $\Ufr$ is a basic toric stack.
    Write $\Ufr = \Xfr_{\Cone(\Cc)}$ for some toric category $\Cc$.
    Then we can compute
    \begin{align*}
        \coweightcat(\GeomDual(\Ufr)) &= \coweightcat(\Xfr_{\Cone(\GDual(\Cc))}) \\
        &= \GDual(\Cc) \textrm{ by \zcref{prop:cocharacter_equivalence}} \\
        &= \CLine(\Xfr_{\Cone(\Cc)}) \textrm{ by \zcref{lem:linear_reconstruction_geometric}} \\
        &= \CLine(\Ufr).
    \end{align*}
    The reverse equivalence follows by symmetry.
\end{proof}

\begin{rmk}
    A theory of Gale duality for smooth toric stacks $\AA^n / G$ with $G \subset \GG_m^n$ is already well-known and appears (for instance) in the study of abelian 3d mirror symmetry \cite{gammage2025perverse, gammage2026hypertoric}.
    The author does not yet know which of the results of this theory still hold for Gale dual pairs of singular toric stacks.
\end{rmk}

\begin{rmk}
    Suppose $X$ is a smooth projective toric variety.
    Then $X$ can naturally be viewed as a GIT quotient, i.e.\ a toric open substack of some basic toric stack $\Ufr$.
    The stacky fan of $\GeomDual(\Ufr)$ is closely related to the \emph{GKZ fan} (a.k.a.\ \emph{secondary fan}) of $X$ discussed in \cite[Chapter 14]{cls}, which controls variation of GIT for this GIT problem.
    By geometrizing the GKZ fan, we obtain a \emph{secondary toric variety}, discussed in e.g.\ \cite{kite2022discriminants}.
    It would be interesting to understand the geometry of the relationship between the toric stack $\GeomDual(\Ufr)$ and the secondary toric variety.
\end{rmk}

\part{Applications} \label{part:applications}

\section{The stacky orbit-cone correspondence and some consequences} \label{sec:orbit_cone}

Before we can establish further applications of \zcref{thm:classification}, we will find it useful to develop a better understanding of how the topology and geometry of a toric stack relate to the combinatorics of its stacky fan.

\subsection{The stacky fan perspective on equivariantization}

When $\Xfr = \Xfr_\Sigma$ is the geometrization of a stacky fan, we can translate our rigidification and equivariantization constructions into the language of stacky fans.

\begin{dfn}
    Let $\Sigma$ be a stacky fan.
    We define a stacky fan $\Sigma \fslash \stab$, the \emph{rigidification} of $\Sigma$, as follows:
    \begin{itemize}
        \item $\Nc_{\Sigma \fslash \stab} = \pi_0(\Nc_\Sigma) / \pi_0(\Nc_\Sigma)_\tors$,
        \item $|\Sigma \fslash \stab| = |\Sigma|$, and
        \item For all $\sigma \in |\Sigma \fslash \stab|$, $P_{\Sigma \fslash \stab, \sigma} = P_{\Sigma,\sigma}$ with the induced map to $\Nc_{\Sigma \fslash \stab}$.
    \end{itemize}
    We also define a stacky fan $\Sigma / \Sigma(0)$, the \emph{equivariantization} of $\Sigma$, as follows:
    \begin{itemize}
        \item $\Nc_{\Sigma / \Sigma(0)} = 0$,
        \item $|\Sigma / \Sigma(0)| = |\Sigma|$, and
        \item For all $\sigma \in |\Sigma / \Sigma(0)|$, $P_{\Sigma / \Sigma(0), \sigma} = P_{\Sigma,\sigma}$ with the unique map to $0$.
    \end{itemize}
\end{dfn}

\begin{prop} \label{prop:rigid_eqvar_fan}
    Let $\Sigma$ be a stacky fan.
    Then:
    \begin{enumerate}
        \item $\Xfr_\Sigma \fslash (\Tfr_\Sigma)_\stab = \Xfr_{\Sigma \fslash \stab}$ and
        \item $\Xfr_\Sigma / \Tfr_\Sigma = \Xfr_{\Sigma / \Sigma(0)}$.
    \end{enumerate}
\end{prop}

\begin{proof}
    (1).
    The rigidification construction is local for $\Tfr_\Sigma$-equivariant charts on $\Xfr_\Sigma$, so we can reduce to the case where $\Sigma = K$ is a stacky cone.
    In this case, $\Xfr_K$ is a commutative 2-monoid stack, $B(\Tfr_K)_\stab$ acts on $\Xfr_K$ via a homomorphism $B(\Tfr_K)_\stab \to \Xfr_K$, and $\Xfr_K \fslash (\Tfr_K)_\stab = \cofib(B(\Tfr_K)_\stab \to \Xfr_K)$.
    Using the equivalence $\coweightcat: \TStk^\basic \simto \TCat$ of \zcref{prop:cocharacter_equivalence}, we see that it suffices to identify $\cofib(\coweightcat(B(\Tfr_K)_\stab) \to \coweightcat(\Xfr_K))$ with $\Oc(K \fslash \stab)$.
    
    Note that $\coweightcat(B(\Tfr_K)_\stab) = (\Nc_K)_\tors$, the full subcategory of $\Nc_K$ consisting of objects which are torsion in $\pi_0(\Nc_K)$.
    Thus, by the colimit computations of \zcref{prop:Tcat_finite_lim_colim}, we obtain
    \begin{align*}
        \iota_0 \cofib(\coweightcat(B(\Tfr_K)_\stab) \to \coweightcat(\Xfr_K)) &= \cofib(\iota_0\coweightcat(B(\Tfr_K)_\stab) \to \iota_0\coweightcat(\Xfr_K)) \\
        &= \cofib((\Nc_K)_\tors \to \Nc_K) \\
        &= \pi_0(\Nc_K) / \pi_0(\Nc_K)_\tors \\
        &= \iota_0 \Oc(K \fslash \stab)
    \end{align*}
    and
    \[
        \PO \cofib(\coweightcat(B(\Tfr_K)_\stab) \to \coweightcat(\Xfr_K)) = \coker(0 \to \PO \Oc(K)) = \PO \Oc(K) = \PO \Oc(K \fslash \stab).
    \]
    Thus \zcref{thm:TCat_structure_thm} implies
    \[
        \cofib(\coweightcat(B(\Tfr_K)_\stab) \to \coweightcat(\Xfr_K)) = \Oc(K \fslash \stab)
    \]
    as needed.

    (2).
    This is much the same as (1), except instead of using $\coweightcat(B(\Tfr_K)_\stab) = (\Nc_K)_\tors$, we use $\coweightcat(\Tfr_K) = \Nc_K$.
\end{proof}

\subsection{Orbits and cones}

We can now gain a better understanding of the relationship between the combinatorial data of a stacky fan $\Sigma$ and the orbits in the corresponding toric stack $\Xfr_\Sigma$.
We saw the following result in the course of the proof of \zcref{prop:geometrization_ess_surj}.

\begin{prop} \label{prop:underlying_poset_of_eqvar}
    Let $\Sigma$ be a stacky fan.
    Equip $(\Xfr_\Sigma / \Tfr_\Sigma)(k)$ with the specialization partial order.
    Then $(\Xfr_\Sigma / \Tfr_\Sigma)(k) = |\Sigma|$ as posets.
\end{prop}

The $\Tfr_\Sigma$-orbits of $k$-points in $\Xfr_\Sigma$ are precisely the preimages of $k$-points of $\Xfr_\Sigma / \Tfr_\Sigma$ in $\Xfr_\Sigma$.
Thus \zcref{prop:underlying_poset_of_eqvar} can be understood as a stacky ``orbit-cone correspondence,'' identifying $\Tfr_\Sigma$-orbits of $k$-points in $\Xfr_\Sigma$ with elements of $\Sigma$.

\begin{dfn}
    Let $\Sigma$ be a stacky fan, and let $\sigma \in \Sigma$.
    We write $\Ofr(\sigma)$ for the $\Tfr_\Sigma$-orbit in $\Xfr_\Sigma$ corresponding to $\sigma$.
\end{dfn}

We can describe these orbits $\Ofr(\sigma)$ explicitly:

\begin{prop} \label{prop:orbit_as_SCDual}
    Let $\Sigma$ be a stacky fan, and let $\sigma \in \Sigma$.
    Then there is a natural equivalence 
    \[
        \Ofr(\sigma) = \SCDual(\ul{\GDual(\cofib(P_\sigma^\gp \to \Nc_\Sigma))}).
    \]
\end{prop}

\begin{proof}
    This claim is $\Tfr_\Sigma$-equivariantly local on $\Xfr_\Sigma$, so we can reduce to the case where $\Sigma = K$ is a stacky cone with maximal point $\sigma = \kappa$.
    In this case we have
    \begin{align*}
        \GDual(\cofib(P_\kappa^\gp \to \Nc_K)) &= \GDual(\Oc(K)^\gpd) \textrm{ by \zcref{prop:Cgpd_calculation}} \\
        &= \iota_0 \GDual(\Oc(K)) \textrm{ by \zcref{prop:gale_duality_localization},}
    \end{align*}
    so it suffices to show $\Ofr(\kappa) = \SCDual(\ul{\iota_0 \GDual(\Oc(K))})$.

    Use \zcref{cor:skew_quotient_of_discrete} to write $\GDual(\Oc(K)) = A \vec{/} M$ for some discrete abelian group $A$ and some toric monoid $M$.
    Then $\iota_0 \GDual(\Oc(K)) = \cofib(M^\pm \to A)$.
    The proof of \zcref{prop:geometrization_of_cone} shows that $\Xfr_K = [U / G]$, where $U = \Spec k\angles{M}$ and $G = \Spec k\angles{A}$.
    Let $T$ be the dense torus of $U$.
    The minimal torus orbit of $U$ is precisely the locus on which all characters of $T$ which are not invertible on $U$ vanish.
    That is, the minimal torus orbit is naturally isomorphic to the torus 
    \[
        T' = \Spec k\angles{M^\pm} = \Spec k\angles{M} / (M \setminus M^\pm).
    \]
    Therefore
    \begin{align*}
        \Ofr(\kappa) &= T' / G \\
        &= \Spec k\angles{M^\pm} / \Spec k\angles{A} \\
        &= \SCDual(\cofib(M^\pm \to A)) \textrm{ by \zcref{ex:shifted_Cartier_dual_of_2term}} \\
        &= \SCDual(\iota_0 \GDual(\Oc(K))).
    \end{align*}
    One can check that this equivalence is independent of the choice of skew quotient presentation of $\GDual(\Oc(K))$.
\end{proof}

\begin{cor} \label{cor:orbit_codimension}
    Let $\Sigma$ be a stacky fan, and let $\sigma \in \Sigma(n)$ for some $n$.
    Then $\codim_{\Xfr_\Sigma} \Ofr(\sigma) = n$.
\end{cor}

\begin{proof}
    The equivariantization map $p: \Xfr_\Sigma \to \Xfr_{\Sigma / \Sigma(0)}$ is smooth and $\Ofr(\sigma)$ is the preimage $p\inv(\sigma)$, so
    \[
        \codim_{\Xfr_\Sigma} \Ofr(\sigma) = \codim_{\Xfr_{\Sigma / \Sigma(0)}} \sigma.
    \]
    Thus we may assume without loss of generality that $\Sigma = \Sigma / \Sigma(0)$, i.e.\ $\Nc_\Sigma = 0$.
    
    We may compute
    \begin{align*}
        \Ofr(\sigma) &= \SCDual\big(\ul{\GDual(\cofib(P_\sigma^\gp \to 0))}\big) \textrm{ by \zcref{prop:orbit_as_SCDual}} \\
        &= \SCDual(\ul{\GDual}(\Bsf P_\sigma^\gp)) \\
        &\simeq \SCDual(\ul{\GDual}(\Bsf \ZZ^n)) \textrm{ because $\sigma \in \Sigma(n)$} \\
        &= B\GG_m^n,
    \end{align*}
    so $\dim \Ofr(\sigma) = -n$.
    Because 
    $\dim \Xfr_\Sigma = \dim \Tfr_\Sigma = 0$, we get $\codim_{\Xfr_\Sigma} \Ofr(\sigma) = 0 - (-n) = n$.
\end{proof}

\subsection{Classification of open immersions}

The stacky orbit-cone correspondence allows us to show that open immersions of toric stacks exactly correspond to open immersions of stacky fans.

\begin{prop} \label{prop:open_imm_biconditional}
    Let $f: \Sigma \to \Sigma'$ be a morphism of stacky fans.
    Then $f$ is an open immersion of stacky fans if and only if $\Xfr_f: \Xfr_\Sigma \to \Xfr_{\Sigma'}$ is an open immersion of toric stacks.
\end{prop}

\begin{proof}
    First note that, in both cases, an open immersion is a morphism which is:
    \begin{itemize}
        \item injective (on elements of the underlying poset for stacky fans and on $k$-points for toric stacks), and
        \item an open immersion Zariski-locally on the domain.
    \end{itemize}

    ($\Rightarrow$).
    Suppose $f: \Sigma \to \Sigma'$ is an open immersion.
    To see that $\Xfr_f$ is an open immersion Zariski-locally on $\Xfr_\Sigma$, let $K \subset \Sigma$ be an open subcone.
    The image $f(K)$ lands in some subcone $K' \subset \Sigma'$.
    The induced morphism $\Xfr_K \to \Xfr_{K'}$ is an open immersion by \zcref{prop:geometrization_cone_immersion}, and composing this with the open immersion $\Xfr_{K'} \to \Xfr_{\Sigma'}$ shows that $\Xfr_K \to \Xfr_{\Sigma'}$ is an open immersion.
    Because $\Xfr_\Sigma$ is a union of the $\Xfr_K$'s, we see that $\Xfr_f$ is Zariski-locally an open immersion.

    To see that $\Xfr_f$ is injective on $k$-points, let $x_1, x_2 \in \Xfr_\Sigma(k)$, and suppose $\Xfr_f(x_1) = \Xfr_f(x_2)$.
    Both $x_1$ and $x_2$ give rise to orbits $[x_1], [x_2] \in (\Xfr_\Sigma / \Tfr_\Sigma)(k)$.
    By \zcref{prop:underlying_poset_of_eqvar}, the natural morphism $(\Xfr_f / \Tfr_f)(k): (\Xfr_\Sigma / \Tfr_\Sigma)(k) \to (\Xfr_{\Sigma'} / \Tfr_{\Sigma'})(k)$ can be identified with $|f|: |\Sigma| \to |\Sigma'|$.
    Since $|f|$ is injective, we must have $[x_1] = [x_2]$ in $(\Xfr_\Sigma / \Tfr_\Sigma)(k)$.
    Thus $x_1$ and $x_2$ lie in the same orbit $\Ofr(\sigma) \subset \Xfr_\Sigma$.
    But the restriction of $\Xfr_f$ to the basic open $\Xfr_{\Sigma_{\leq \sigma}}$ containing $\Ofr(\sigma)$ is an open immersion, hence injective.
    Thus we must have $x_1 = x_2$.
    
    ($\Leftarrow$).
    Suppose $\Xfr_f$ is an open immersion of toric stacks.
    To see that $\Xfr_f$ is Zariski-locally an open immersion, note that $\Xfr_f$ maps basic opens to basic opens, so we may reduce to the case where $\Sigma$ and $\Sigma'$ are stacky cones.
    In this case the claim follows from \zcref{lem:geometrization_cone_immersion_converse}.

    To see that $f$ is (globally) injective, we again use \zcref{prop:underlying_poset_of_eqvar} to identify $(\Xfr_\Sigma / \Tfr_\Sigma)(k)$ with $|\Sigma|$.
    It suffices to show that the quotiented morphism $\Xfr_f / \Tfr_f$ is injective.
    Equivalently, for any $\Tfr_\Sigma$-orbits $\Ofr(\sigma_1), \Ofr(\sigma_2)$ in $\Xfr_\Sigma$, we need to show that $\Xfr_f(\Ofr(\sigma_1)(k)) \cap \Xfr_f(\Ofr(\sigma_2)(k)) \neq \varnothing$ implies $\Ofr(\sigma_1) = \Ofr(\sigma_2)$.
    Since $\Xfr_f$ is $\Tfr_\Sigma$-equivariant and induces an isomorphism $\Tfr_\Sigma \simto \Tfr_{\Sigma'}$, all natural maps $\Ofr(\sigma_i) \to \Ofr(f(\sigma_i))$ are quotient maps, hence are surjective on $k$-points.
    Thus, if $\Xfr_f(\Ofr(\sigma_1)(k)) \cap \Xfr_f(\Ofr(\sigma_2)(k)) \neq \varnothing$, we must in fact have $\Xfr_f(\Ofr(\sigma_1)) = \Xfr_f(\Ofr(\sigma_2))$.
    But then injectivity of $\Xfr_f$ forces $\Ofr(\sigma_1) = \Ofr(\sigma_2)$, as needed.
\end{proof}

\begin{cor}
    The geometrization functor $\Xfr_{-}: \StkFan \to \TStk$ is an equivalence of Zariski sites.
\end{cor}

\subsection{Cohomological characterization of basic toric stacks}

The stacky orbit-cone correspondence also allows us to determine whether or not a toric stack is basic entirely in terms of sheaf cohomology.

\begin{prop} \label{prop:basic_coh_aff}
    Let $\Xfr_\Sigma$ be a toric stack.
    The following are equivalent:
    \begin{enumerate}
        \item $\Sigma$ is a stacky cone.
        \item $\Xfr_\Sigma$ is basic.
        \item $\Xfr_\Sigma$ is cohomologically affine, i.e.\ the global sections functor $\Gamma(\Xfr_\Sigma, -): \QCoh(\Xfr_\Sigma) \to \Mod_k$ is exact.
    \end{enumerate}
\end{prop}

\begin{proof}
    ($1 \Leftrightarrow 2$).
    This follows from \zcref{prop:geometrization_of_cone}.

    ($2 \Rightarrow 3$).
    This is a special case of \cite[Example 12.9]{alper2013good}.
    
    ($3 \Rightarrow 1$).
    Suppose $\Xfr_\Sigma$ is cohomologically affine.
    By \zcref{prop:pts_and_cones_stacky_fan}, to see that $\Sigma$ is a stacky cone, it suffices to show that $|\Sigma|$ has a unique maximal point.
    
    Choose a splitting of $\Tfr \to \Tfr_\coarse$, so $\Tfr_\coarse$ acts on $\Xfr_\Sigma$.
    The quotient $\Xfr_\Sigma / \Tfr_\coarse$ is cohomologically affine (again by the argument of \cite[Example 12.9]{alper2013good}).
    Furthermore, $\Xfr_\Sigma / \Tfr_\coarse = \Xfr_{\Sigma'}$ for some stacky fan $\Sigma'$ with $|\Sigma'| = |\Sigma|$. 
    Thus we may replace $\Sigma$ by $\Sigma'$ and assume without loss of generality that $\Tfr_\Sigma = B(\Tfr_\Sigma)_\stab$.

    Let $\sigma_1, \dots, \sigma_n$ be the maximal points of $\Sigma$, so the corresponding orbits $\Ofr(\sigma_i)$ are all closed.
    Then $\Yfr = \cup_{i=1}^n \Ofr(\sigma_i) \subset \Xfr_\Sigma$ is a closed substack of $\Xfr_\Sigma$.
    Let $i: \Yfr \to \Xfr_\Sigma$ be the natural closed immersion, and let $\Isc$ be the ideal sheaf of $\Yfr$, so there is a short exact sequence
    \[
    \begin{tikzcd}
        0 \rar & \Isc \rar & \Osc_{\Xfr_\Sigma} \rar & i_* \Osc_\Yfr \rar & 0.
    \end{tikzcd}
    \]
    Taking global sections and using cohomological affineness gives a short exact sequence
    \[
    \begin{tikzcd}
        0 \rar & \Gamma(\Xfr_\Sigma, \Isc) \rar & \Gamma(\Xfr_\Sigma, \Osc_{\Xfr_\Sigma}) \rar & \Gamma(\Yfr, \Osc_\Yfr) \rar & 0.
    \end{tikzcd}
    \]
    Because $\Xfr_\Sigma$ is reduced, global sections of $\Osc_{\Xfr_\Sigma}$ are determined by their restriction to the dense torus.
    In other words, $\Gamma(\Xfr_\Sigma, \Osc_{\Xfr_\Sigma})$ is a subring of
    \[
        \Gamma(\Tfr_\Sigma, \Osc_{\Tfr_\Sigma}) = \Gamma(B(\Tfr_\Sigma)_\stab, \Osc_{B(\Tfr_\Sigma)_\stab}) = k.
    \]
    Since $\Gamma(\Xfr_\Sigma, \Osc_{\Xfr_\Sigma})$ is a $k$-algebra, we must actually have $\Gamma(\Xfr_\Sigma, \Osc_{\Xfr_\Sigma}) = k$.
    It follows that $k$ surjects onto $\Gamma(\Yfr, \Osc_\Yfr)$.
    Since 
    \[
        \Gamma(\Yfr, \Osc_\Yfr) = \bigoplus_{i=1}^n \Gamma(\Ofr(\sigma_i), \Osc_{\Ofr(\sigma_i)}) = k^{\oplus n},
    \]
    we must have $n = 1$.
    Thus $\sigma_1$ is the unique maximal point of $|\Sigma|$, as desired.
\end{proof}

Of course, being cohomologically affine is actually a property of morphisms of stacks.
The generalization of the above criterion to the relative setting is straightforward:

\begin{dfn}
    A morphism of stacky fans $f: \Sigma \to \Sigma'$ is \emph{conical} if, for every subcone $K' \subset \Sigma'$, the preimage $f\inv(K') \subset \Sigma$ is a subcone.
\end{dfn}

\begin{prop} \label{prop:coh_aff_relative}
    Let $f: \Sigma \to \Sigma'$ be a morphism of stacky fans.
    The following are equivalent:
    \begin{enumerate}
        \item $f$ is conical.
        \item The geometrization $\Xfr_f: \Xfr_\Sigma \to \Xfr_{\Sigma'}$ is cohomologically affine, i.e.\ the pushforward functor $(\Xfr_f)_*: \QCoh(\Xfr_\Sigma) \to \QCoh(\Xfr_{\Sigma'})$ is exact.
    \end{enumerate}
\end{prop}

\begin{proof}
    ($\Rightarrow$).
    Suppose $f$ is conical.
    By \cite[Proposition 3.10]{alper2013good}, the property of being cohomologically affine is fppf-local (and hence Zariski-local) on the base whenever the base has at least quasi-affine diagonal.
    Thus we may cover $\Xfr_{\Sigma'}$ by basic opens $\Xfr_{K'}$ and reduce to the case where $\Sigma' = K'$ is a stacky cone.
    Since $f$ is conical, $\Sigma = f\inv(K')$ must also be a stacky cone $K$.
    But then 
    \[
    \begin{tikzcd}
        \Xfr_K \ar[rr, "\Xfr_f"] \ar[dr] & & \Xfr_{K'} \ar[dl] \\
        & \Spec k & 
    \end{tikzcd}
    \]
    is a commutative diagram of toric stacks where the structure maps to $\Spec k$ are both cohomologically affine by \zcref{prop:basic_coh_aff}.
    By the cancellation property for cohomologically affine morphisms (\cite[Proposition 3.13]{alper2013good}), this implies $\Xfr_f$ is cohomologically affine.

    ($\Leftarrow$).
    Suppose $\Xfr_f$ is cohomologically affine.
    Let $K' \subset \Sigma'$ be a subcone, and let $\Ufr = \Xfr_{K'} \times_{\Xfr_{\Sigma}} \Xfr_{\Sigma'}$ be the preimage of $\Xfr_{K'}$ in $\Xfr_\Sigma$.
    Then $\Ufr$ is a toric open substack of $\Xfr_{\Sigma'}$, and the natural morphism $\Ufr \to \Xfr_{K'}$ is cohomologically affine.
    Composing this morphism with the cohomologically affine structure morphism $\Xfr_{K'} \to \Spec k$, we see that the structure morphism $\Ufr \to \Spec k$ must also be cohomologically affine.
    But then $\Ufr$ is basic by \zcref{prop:basic_coh_aff}, hence corresponds to some subcone $K \subset \Sigma$.
    Since $\Ufr$ also corresponds to the subfan $f\inv(K') \subset \Sigma$, we have $f\inv(K') = K$.
    Since $K'$ was arbitrary, $f$ must be conical.
\end{proof}

\section{Classification of toric Deligne--Mumford stacks and beyond} \label{sec:dm_stacks}

So far we have worked in the generality of toric Artin stacks with possibly infinite generic stabilizer.
Many applications of toric stacks focus on much simpler cases, assuming e.g.\ that the toric stack in question is Deligne--Mumford or that the generic stabilizer is trivial.
In this section we will give criteria in terms of stacky fans for determining when a toric stack satisfies some common conditions.
In particular, given a stacky fan $\Sigma$, we can detect when the toric stack $\Xfr_\Sigma$:
\begin{itemize}
    \item has generic stabilizer (\zcref{cor:generic_stabilizer_criteria}),
    \item is Deligne--Mumford (\zcref{prop:DM_criterion}),
    \item is a scheme (\zcref{prop:scheme_criterion}),
    \item is separated (\zcref{prop:separated_criterion}), and
    \item is complete (\zcref{prop:complete_criterion}).
    \item is smooth (\zcref{prop:smooth_criterion}).
\end{itemize}

\begin{rmk}
    \zcref{prop:separated_criterion} actually shows that a toric stack $\Xfr_\Sigma$ is separated if and only if it is a KM fan in the sense of \cite{gillam2015theory} (see \zcref{ex:KM_fan}).
    Our proof of \zcref{prop:separated_criterion} is based on a reduction to the classification of separated toric Deligne--Mumford stacks presented in \cite[Theorem 3.10.7]{gillam2015theory}.
    Our criterion for completeness is then merely a restatement of \cite[Proposition 3.5.4]{gillam2015theory}.
\end{rmk}

\begin{rmk}
    For the sake of brevity, we only consider absolute properties of stacks (rather than relative properties of morphisms) here.
    Analogous characterizations of the corresponding relative properties should be possible and interesting.
\end{rmk}

\subsection{Stabilizer computations and consequences}

We are first interested in determining when the stabilizer group of a toric stack at a $k$-point is finite or trivial.
The following lemma is key in answering this question:

\begin{lem} \label{lem:computation_of_Tstab}
    Let $\Ac$ be a toric groupoid, and let $\Tfr = \SCDual(\ul{\GDual(\Ac)})$.
    Then:
    \begin{enumerate}
        \item $\Tfr_\stab$ is finite if and only if $\Ac$ is discrete, and
        \item $\Tfr_\stab$ is trivial if and only if $\Ac$ is discrete and free.
    \end{enumerate}
\end{lem}

\begin{proof}
    By \zcref{prop:strict_trivialization} and the structure theorem of finitely generated abelian groups, we may choose a (non-canonical) isomorphism 
    \[
        \Ac \simeq \ZZ^{r_1} \times \pi_\tors \times \Bsf \ZZ^{r_2}
    \]
    where $r_1, r_2 \geq 0$ and $\pi_\tors$ is a finitely generated torsion abelian group.
    We can compute:
    \begin{itemize}
        \item $\SCDual(\ul{\GDual(\ZZ^{r_1})}) \simeq \SCDual(\ul{\Bsf \ZZ^{r_1}}) = \GG_m^{r_1}$ by \zcref{ex:shifted_Cartier_dual_of_BA}.
        \item $\SCDual(\ul{\GDual(\pi_\tors)}) \simeq \SCDual(\ul{\pi_\tors}) \simeq B\pi_\tors$ by \zcref{ex:shifted_Cartier_dual_of_A}.
        \item $\SCDual(\ul{\GDual(\Bsf \ZZ^{r_2})}) \simeq \SCDual(\ul{\ZZ^{r_2}}) = \Bsf \GG_m^{r_2}$ by \zcref{ex:shifted_Cartier_dual_of_A}.
    \end{itemize}
    Thus
    \[
        \Tfr \simeq \GG_m^{r_1} \times BA \times B \GG_m^{r_2}.
    \]
    We see that $\Tfr_\stab \simeq A \times \GG_m^{r_2}$ is finite if and only if $r_2 = 0$, which happens if and only if $\Ac$ is discrete.
    Furthermore, $\Tfr_\stab = 0$ if and only if $r_2 = 0$ and $A = 0$, or equivalently $\Ac$ is discrete and free.
\end{proof}

\begin{cor} \label{cor:generic_stabilizer_criteria}
    Let $\Xfr = \Xfr_\Sigma$ be a toric stack.
    Then:
    \begin{enumerate}
        \item $\Xfr$ has finite generic stabilizer if and only if $\Nc_\Sigma$ is discrete.
        \item $\Xfr$ has trivial generic stabilizer if and only if $\Nc_\Sigma$ is discrete and free.
    \end{enumerate}
\end{cor}

\begin{proof}
    These are really claims about the stacky torus $\Tfr_\Sigma = \SCDual(\ul{\GDual(\Nc_\Sigma)})$, so they follow from \zcref{lem:computation_of_Tstab}.
\end{proof}

By applying \zcref{lem:computation_of_Tstab} to the torus orbits $\Ofr(\sigma)$, we can determine when a toric stack $\Xfr_\Sigma$ is Deligne--Mumford.

\begin{prop} \label{prop:DM_criterion}
    Let $\Xfr = \Xfr_\Sigma$ be a toric stack.
    The following are equivalent:
    \begin{enumerate}
        \item $\Xfr$ is Deligne--Mumford,
        \item The fan $\Sigma$ satisfies:
        \begin{itemize}
            \item $\Nc_\Sigma$ is discrete and, 
            \item For all $\sigma \in \Sigma$, the natural morphism $b_\sigma: P_\sigma \to \Nc_\Sigma$ is injective.
        \end{itemize}
        \item For all $\sigma \in \Sigma$, the toric category $\Oc(\Sigma_{\leq \sigma})$ is a poset.
    \end{enumerate}
\end{prop}

\begin{proof}
    ($1 \Leftrightarrow 2$).
    If $\Xfr$ is Deligne--Mumford, then $\Xfr$ must have finite generic stabilizer.
    Thus for both directions of the claim we may assume $\Nc_\Sigma$ is discrete.
    Because $\Xfr$ is a union of the stacky torus orbits $\Ofr(\sigma)$ for $\sigma \in \Sigma$, we see that $\Xfr$ is Deligne--Mumford if and only if all $\Ofr(\sigma)$ have finite generic stabilizer.
    
    Fix $\sigma \in \Sigma$.
    By \zcref{prop:orbit_as_SCDual}, we may write
    \[
        \Ofr(\sigma) = \SCDual(\ul{\GDual(\cofib(b_\sigma^\gp: P_\sigma^\gp \to \Nc_\Sigma))}).
    \]
    Because $\Nc_\Sigma$ is discrete, the cofiber here decomposes as $\coker(b_\sigma^\gp) \times \Bsf \ker(b_\sigma^\gp)$.
    Then \zcref{lem:computation_of_Tstab} implies the generic stabilizer of $\Ofr(\sigma)$ is finite if and only if $\ker(b_\sigma^\gp) = 0$, which happens if and only if $b_\sigma$ is injective.

    ($2 \Leftrightarrow 3$).
    Write each $\Oc(\Sigma_{\leq \sigma})$ as $\Nc_\Sigma \vec{/} P_\sigma$.
    The forward direction follows from \zcref{ex:toric_poset}, and the backward direction follows from \zcref{ex:toric_poset_converse}.
\end{proof}

Before characterizing when a toric stack is a scheme, we should first note that ``there are no toric algebraic spaces:''

\begin{rmk} \label{rmk:no_toric_alg_sp}
    \zcref{thm:classification} implies that every toric algebraic space $\Xfr$ is a scheme.
    The proof is a straightforward adaptation of that in \cite[Remark 6.3]{geraschenko2015toric_2}, using the fact that $\Xfr$ admits a Zariski cover by basic opens.
\end{rmk}

Thus showing that a toric stack is a scheme reduces to showing that all stabilizer groups are trivial.
We can use this to characterize the toric schemes much like we characterized the toric Deligne--Mumford stacks:

\begin{prop} \label{prop:scheme_criterion}
    Let $\Xfr = \Xfr_\Sigma$ be a toric stack.
    The following are equivalent:
    \begin{enumerate}
        \item $\Xfr$ is a scheme.
        \item The stacky fan $\Sigma$ satisfies:
        \begin{itemize}
            \item $\Nc_\Sigma$ is discrete and free,
            \item For all $\sigma \in \Sigma$, $b_\sigma: P_\sigma \to \Nc_\Sigma$ is injective, and
            \item For all $\sigma \in \Sigma$, $\im b_\sigma$ is saturated in $\Nc_\Sigma$.
        \end{itemize}
        \item For all $\sigma \in \Sigma$, the toric category $\Oc(\Sigma_{\leq \sigma})$ is a poset satisfying the following saturation condition:
        \begin{itemize}
            \item If $n \in \Nc_\Sigma = \iota_0 \Oc(\Sigma_{\leq \sigma})$ and $0 \leq d \cdot n$ for some $d > 0$, then $0 \leq n$.
        \end{itemize}
    \end{enumerate}
\end{prop}

\begin{proof}
    ($1 \Leftrightarrow 2$). This can be deduced from \zcref{ex:Gale_dual_classical_toric}, though we give an alternate argument in the spirit of \zcref{prop:DM_criterion}. 
    
    If $\Xfr$ is a scheme, then $\Xfr$ must be Deligne--Mumford with trivial generic stabilizer.
    Thus for both directions of the claim we may assume $\Nc_\Sigma$ is discrete and free and all $b_\sigma$ are injective.
    Again writing $\Xfr$ as a union of torus orbits $\Ofr(\sigma)$, we see that $\Xfr$ is an algebraic space (hence a scheme, by \zcref{rmk:no_toric_alg_sp}) if and only if all $\Ofr(\sigma)$ have trivial generic stabilizer.

    Fix $\sigma \in \Sigma$.
    We may use \zcref{prop:orbit_as_SCDual} and injectivity of the $b_\sigma$ to write
    \[
        \Ofr(\sigma) = \SCDual(\ul{\GDual(\Nc_\Sigma / \im b_\sigma^\gp}).
    \]
    By \zcref{lem:computation_of_Tstab}, we see $\Ofr(\sigma)$ has trivial generic stabilizer if and only if $\Nc_\Sigma / \im b_\sigma^\gp$ is torsionfree.
    By \zcref{prop:saturated_submonoid_criterion}, this is equivalent to $\im b_\sigma$ being saturated in $\Nc_\Sigma$.

    ($2 \Leftrightarrow 3$).
    Let $\sigma \in \Sigma$.
    We showed in \zcref{prop:DM_criterion} that the first two conditions of (2) hold if and only if $\Oc(\Sigma_{\leq \sigma})$ is a poset.
    The third condition of (2) is equivalent to the saturation condition on the poset because, by \zcref{ex:toric_poset_converse}, saying that $n \geq 0$ is the same as saying that $n \in b_\sigma(P_{\sigma})$.
\end{proof}

\begin{rmk}
    \zcref{prop:DM_criterion} tells us how to describe the maximal Deligne--Mumford toric open substack of a given toric stack $\Xfr_\Sigma$ via stacky fans: we simply throw out all points of $\Sigma$ which do not satisfy \zcref{prop:DM_criterion}(3).
    Likewise, \zcref{prop:scheme_criterion} tells us that the maximal toric open subscheme of $\Xfr_\Sigma$ has a stacky fan obtained by throwing out all points of $\Sigma$ which do not satisfy \zcref{prop:scheme_criterion}(3),
\end{rmk}

\subsection{Separation and KM fans}

We'd now like to understand when a toric stack $\Xfr = \Xfr_\Sigma$ is separated.
The answer is closely related to the theory of KM fans presented in \cite{gillam2015theory}.
In fact, Gillam and Molcho show that separated toric stacks with trivial generic stabilizer are classified by ``lattice KM fans,'' i.e.\ the stacky fans of \cite[\S 4]{tyomkin2012tropical}.
To generalize their result, we need the following:

\begin{lem} \label{lem:rigidification_separated}
    Let $\Xfr$ be a toric stack with finite generic stabilizer.
    Then the rigidification map $p: \Xfr \to \Xfr \fslash \Tfr_\stab$ is separated.
\end{lem}

\begin{proof}
    The map $p$ exhibits $\Xfr$ as a $\Tfr_\stab$-gerbe over $\Xfr \fslash \Tfr_\stab$.
    The property of being separated is local on the base, so we may pass to a coordinate chart on $\Xfr \fslash \Tfr_\stab$ for which the gerbe in question is trivial.
    We thus reduce to showing that the projection map $B\Tfr_\stab \times S \to S$ is separated for any scheme $S$ over $\Spec k$.
    This projection map is the base change of the structure map $B\Tfr_\stab \to \Spec k$, hence is separated because $\Tfr_\stab$ is finite.
\end{proof}

\begin{prop} \label{prop:separated_criterion}
    Let $\Xfr = \Xfr_\Sigma$ be a toric stack.
    Then $\Xfr$ is separated if and only if $\Sigma$ arises from a KM fan (see \zcref{ex:KM_fan}), i.e.\
    \begin{enumerate}
        \item (Finite generic stabilizer): $\Nc_\Sigma$ is discrete,
        \item (Deligne--Mumford): All structure maps $b_\sigma: P_\sigma \to \Nc_\Sigma$ (for $\sigma \in \Sigma$) are inclusions,
        \item ($\{ b_\sigma(P_\sigma) \}_{\sigma \in |\Sigma|}$ is a set): For all $\sigma_1, \sigma_2 \in \Sigma$, if $b_{\sigma_1}(P_{\sigma_1}) = b_{\sigma_2}(P_{\sigma_2})$, then $\sigma_1 = \sigma_2$, and
        \item (Cones intersect in faces): For all $\sigma_1, \sigma_2 \in \Sigma$, the intersection $b_{\sigma_1}(P_{\sigma_1}) \cap b_{\sigma_2}(P_{\sigma_2})$ (taken in $\Nc_\Sigma$) is a face of $b_{\sigma_1}(P_{\sigma_1})$ and $b_{\sigma_2}(P_{\sigma_2})$.
    \end{enumerate}
\end{prop}

\begin{proof}
    ($\Rightarrow$).
    Suppose $\Xfr$ is separated.
    Because $\Xfr$ has affine diagonal, this implies that the diagonal morphism of $\Xfr$ is in fact finite, so all stabilizer groups of $\Xfr$ are finite.
    In particular, $\Xfr$ must be Deligne--Mumford, so the first two conditions hold by \zcref{prop:DM_criterion}.
    Write $N_\Sigma = \Nc_\Sigma$ for notational clarity, since this is now guaranteed to be discrete.
    
    Consider the rigidification $\Xfr \fslash \Tfr_\stab$.
    By \zcref{lem:rigidification_separated} and the cancellation law for separated morphisms, the stack $\Xfr \fslash \Tfr_\stab$ is separated.
    The classification result of \cite[Theorem 3.10.7]{gillam2015theory} then implies $\Xfr \fslash \Tfr_\stab$ is the geometrization of a lattice KM fan.
    By \zcref{prop:rigid_eqvar_fan}, the stacky fan corresponding to $\Xfr \fslash \Tfr_\stab$ is $\Sigma \fslash \stab$, i.e.\ it is the same as $\Sigma$ but with $N_\Sigma$ replaced by $\ol{N}_\Sigma := N_\Sigma / (N_\Sigma)_\tors$ and with the morphisms $b_\sigma$ replaced by the natural morphisms $\ol{b}_\sigma: P_\sigma \to \ol{N}_\Sigma$.
    
    Because $\Sigma \fslash \stab$ is a KM fan, $\{ \ol{b}_\sigma(P_\sigma) \}_{\sigma \in |\Sigma|}$ is a set.
    Because $b_{\sigma_1}(P_{\sigma_1}) = b_{\sigma_2}(P_{\sigma_2})$ implies $\ol{b}_{\sigma_1}(P_{\sigma_1}) = \ol{b}_{\sigma_2}(P_{\sigma_2})$ and thus $\sigma_1 = \sigma_2$, we see that $\{ b_\sigma(P_\sigma) \}_{\sigma \in |\Sigma|}$ is a set.
    
    Again because $\Sigma \fslash \stab$ is a KM fan, we know that cones intersect in faces in $\Sigma \fslash \stab$.
    This translates into the claim that, for all $\sigma_1, \sigma_2 \in \Sigma$, the intersection $\ol{b}_{\sigma_1}(P_{\sigma_1}) \cap \ol{b}_{\sigma_2}(P_{\sigma_2})$ is a face of $\ol{b}_{\sigma_1}(P_{\sigma_1})$ and $\ol{b}_{\sigma_2}(P_{\sigma_2})$.
    We can use this to show $b_{\sigma_1}(P_{\sigma_1}) \cap b_{\sigma_2}(P_{\sigma_2})$ is a face of $b_{\sigma_1}(P_{\sigma_1})$ and $b_{\sigma_2}(P_{\sigma_2})$ as follows.

    For notational convenience we write $B = b_{\sigma_1}(P_{\sigma_1}) \cap b_{\sigma_2}(P_{\sigma_2})$ and $\ol{B} = \ol{b}_{\sigma_1}(P_{\sigma_1}) \cap \ol{b}_{\sigma_2}(P_{\sigma_2})$.
    Because $\{\ol{b}_{\sigma}(P_\sigma)\}_{\sigma \in |\Sigma|}$ is closed under taking faces, we can write the face $\ol{B} \subset \ol{b}_{\sigma_1}(P_{\sigma_1})$ as $\ol{b}_{\sigma_3}(P_{\sigma_3})$ for some $\sigma_3 \leq \sigma_1$.
    Likewise we can write $\ol{B} = \ol{b}_{\sigma_3'}(P_{\sigma_3'})$ for some $\sigma_3' \leq \sigma_2$.
    Because $\{ \ol{b}_\sigma(P_\sigma) \}_{\sigma \in |\Sigma|}$ is a set, we must have $\sigma_3 = \sigma_3'$.
    Then, because $b_{\sigma_3}(P_{\sigma_3})$ is a face of $b_{\sigma_1}(P_{\sigma_1})$ and $b_{\sigma_2}(P_{\sigma_2})$, it will suffice to show $B = b_{\sigma_3}(P_{\sigma_3})$ as subsets of $N_\Sigma$.

    Let $i_1: P_{\sigma_3} \to P_{\sigma_1}$ and $i_2: P_{\sigma_3} \to P_{\sigma_2}$ be the face inclusion morphisms, so 
    $b_{\sigma_3} = b_{\sigma_1} \circ i_1 = b_{\sigma_2} \circ i_2$.
    Thus 
    \[
        b_{\sigma_3}(P_{\sigma_3}) = b_{\sigma_1}(i_1(P_{\sigma_3})) \subset b_{\sigma_1}(P_{\sigma_1}),
    \]
    and likewise $b_{\sigma_3}(P_{\sigma_3}) \subset b_{\sigma_2}(P_{\sigma_2})$.
    Thus $b_{\sigma_3}(P_{\sigma_3}) \subset B$.
    
    For the reverse inclusion $B \subset b_{\sigma_3}(P_{\sigma_3})$, note that the natural map $B \to \ol{B}$ is injective because the natural map $b(P_{\sigma_1}) \to \ol{b}(P_{\sigma_1})$ is an isomorphism.
    It therefore suffices to show $\ol{B} \subset \ol{b}_{\sigma_3}(P_{\sigma_3})$.
    But this holds by definition of $\sigma_3$.
    Thus 
    \[
    b_{\sigma_1}(P_{\sigma_1}) \cap b_{\sigma_2}(P_{\sigma_2}) = B = b_{\sigma_3}(P_{\sigma_3})
    \]
    so cones intersect in faces for $\Sigma$.

    ($\Leftarrow$).
    This is \cite[Corollary 3.2.7]{gillam2015theory}.
\end{proof}

\begin{rmk}
    The ``cones intersect in faces'' condition in \zcref{prop:separated_criterion} is superficially similar to the requirement that subcones of a stacky fan intersect in subcones, but the two conditions are distinct.
    Indeed, the former is an assertion about set-theoretic intersections within the discrete abelian group $\Nc_\Sigma$, while the latter is an assertion about meets in the poset $|\Sigma|$.
\end{rmk}

\begin{ex}
    For an explicit example of a stacky fan where all conditions of \zcref{prop:separated_criterion} other than ``cones intersect in faces'' hold, consider the stacky fan 
    \[
        \Sigma = \Cone(\ZZ^2 \vec{/}_\phi \NN^2) \cup_{\Cone(\ZZ^2)} \Cone(\ZZ^2 \vec{/}_\psi \NN)
    \]
    where
    \[
        \phi = \begin{bmatrix}
            1 & 0 \\
            0 & 1
        \end{bmatrix}
        \textrm{ and }
        \psi = \begin{bmatrix}
            1 \\
            1
        \end{bmatrix}.
    \]
    Then $\Sigma$ is a stacky fan satisfying the first three conditions listed in \zcref{prop:separated_criterion}, but the set-theoretic intersection $\im \phi \cap \im \psi$ is $\im \psi$, which is not a face of $\im \phi$.
\end{ex}

\subsection{The underlying classical fan of a KM fan}

\begin{dfn}[{\cite[Example 2.4.4]{gillam2015theory}}]
    Let $\Sigma$ be a KM fan (viewed as a stacky fan).
    The \emph{underlying classical fan} of $\Sigma$, written $\Sigma_\coarse$, is the stacky fan $(N_{\Sigma,\coarse}, |\Sigma_\coarse|, P_{\Sigma,\coarse})$ where:
    \begin{itemize}
        \item $N_{\Sigma,\coarse} = N_\Sigma / (N_\Sigma)_\tors$,
        \item $|\Sigma_\coarse| = |\Sigma|$, and
        \item For all $\sigma \in |\Sigma_\coarse| = |\Sigma|$, $P_{\Sigma,\coarse,\sigma}$ is the saturation of $P_{\Sigma,\sigma}$ in $N_{\Sigma,\coarse}$.
    \end{itemize}
    This is a classical fan as in \cite{cls}.
\end{dfn}

\begin{rmk}
    The above definition of a stacky fan $\Sigma_\coarse$ makes sense for any Deligne--Mumford fan $\Sigma$, though the result will be a classical fan if and only if $\Sigma$ is a KM fan.
\end{rmk}

\cite[\S 2.3]{gillam2015theory} outlines a method for drawing pictures of KM fans using underlying classical fans.
Essentially, one draws the usual picture of $\Sigma_\coarse$ (as in \cite{cls}), except that one only draws the lattice points contained in the sets $\im b_\sigma$.
When $N_\Sigma$ has torsion, one includes extra dimensions in the picture corresponding to generators of $(N_\Sigma)_\tors$.

\begin{ex}
    The stacky fan corresponding to $\PP(1, 2)$ can be drawn as 
    \[
    \begin{tikzcd}
        \dots & \lar \bullet & \lar \bullet & \lar \bullet \ar[rr] & & \bullet \ar[rr] & & \dots,
    \end{tikzcd}
    \]
    where the arrows on the right are longer because they skip over the lattice points 1, 3, 5, et cetera.
\end{ex}

\cite[Proposition 3.2.6]{gillam2015theory} shows that, for any KM fan $\Sigma$, the toric variety $\Xfr_{\Sigma_\coarse}$ is the coarse moduli space of $\Xfr_\Sigma$.
We can show that this is actually a good moduli space in the sense of \cite{alper2013good}:

\begin{prop}
    Let $\Sigma$ be a KM fan.
    Then $\Xfr_{\Sigma_\coarse}$ is the good moduli space of $\Xfr_\Sigma$.
\end{prop}

\begin{proof}
    Being a good moduli space is local on $\Xfr_{\Sigma_\coarse}$, so it suffices to consider the case where $\Sigma_\coarse = K_\coarse$ is a single cone (so the same also holds for $\Sigma = K$).
    By \cite[Proposition 3.2.6]{gillam2015theory}, $\Xfr_{K_\coarse}$ is the coarse moduli space of $\Xfr_K$.
    Both good moduli spaces and coarse moduli spaces satisfy the same universal property (at least in the locally noetherian case) of being the ``initial algebraic spaces receiving a map from the stack in question.''
    Thus, if $\Xfr_K$ also has a good moduli space, then this good moduli space agrees with the coarse moduli space $\Xfr_{K_\coarse}$.
    We know the natural morphism $\Xfr_K \to \Spec \Gamma(\Xfr_K, \Osc_{\Xfr_K})$ is a good moduli space morphism by \cite[Theorem 13.2]{alper2013good}.
    Thus $\Xfr_{K_\coarse} = \Spec \Gamma(\Xfr_K, \Osc_{\Xfr_K})$ is a good moduli space of $\Xfr_\Sigma$.
\end{proof}

\begin{rmk}
    General toric stacks need not admit good moduli spaces, and the construction of an ``underlying classical fan'' breaks down for general stacky fans.
    It is possible to characterize the toric stacks which admit good moduli spaces using a criterion similar to that of \cite[Theorem 6.5]{geraschenko2015toric_1}, though we shall not do so here.
\end{rmk}

\subsection{Completeness}

For the sake of completeness, we recall the following result from \cite[Proposition 3.5.4]{gillam2015theory} characterizing when the morphism from a separated toric stack to $\Spec k$ is proper.

\begin{prop}[{\cite[Proposition 3.5.4]{gillam2015theory}}] \label{prop:complete_criterion}
    Let $\Xfr = \Xfr_\Sigma$ be a toric stack.
    The following are equivalent:
    \begin{enumerate}
        \item $\Xfr$ is complete.
        \item $\Sigma$ arises from a KM fan, and for all $n \in \Nc_\Sigma$, there exists a nonzero natural number $d$ and $\sigma \in \Sigma$ such that $d \cdot n \in \im b_\sigma$.
        \item $\Sigma$ arises from a KM fan, and the underlying classical fan $\Sigma_\coarse$ is complete.
    \end{enumerate}
\end{prop}

\subsection{Smoothness and the Cox construction}

It is straightforward to check when a toric stack is smooth:

\begin{prop} \label{prop:smooth_criterion}
    Let $\Xfr = \Xfr_\Sigma$ be a toric stack.
    The following are equivalent:
    \begin{enumerate}
        \item $\Xfr_\Sigma$ is smooth.
        \item For all maximal points $\sigma \in \Sigma$, the commutative monoid $P_\sigma$ is isomorphic to $\NN^n$ for some $n$ (depending on $\sigma$).
        \item For all $\sigma \in \Sigma$, the commutative monoid $P_\sigma$ isomorphic to $\NN^n$ for some $n$ (depending on $\sigma$).
    \end{enumerate}
\end{prop}

\begin{proof}
    ($1 \Rightarrow 2$).
    Suppose $\Xfr$ is smooth.
    Let $\sigma \in \Sigma$ be maximal, and let $K = \Sigma_{\leq \sigma}$.
    We may assume without loss of generality that $\Sigma = K$, so $\Xfr$ is basic.
    Furthermore, since replacing $K$ by its equivariantization $K / K(0)$ does not change $P_\sigma$ or the smoothness of the toric stack in question, we may assume without loss of generality that $K = K(0)$, or equivalently $\Nc_K = 0$.
    Thus $\Osc(K) = \Bsf P_\sigma$.
    It follows that $\Xfr_K = [U / T]$ where $U$ is the toric variety $\Spec k\angles{P_\sigma}$ and $T$ is the dense torus in $U$.
    Smoothness is a smooth-local property, so $\Xfr_K$ is smooth if and only if $U$ is smooth.
    By \cite[Theorem 1.3.12]{cls}, $U$ is smooth if and only if $P_\sigma \cong \NN^n$ for some $n$, giving the claim.

    ($2 \Rightarrow 3$).
    Suppose (2) holds.
    For all $\sigma' \in \Sigma$, there exists a maximal point $\sigma \geq \sigma'$ such that $P_{\sigma'}$ is a face of $P_\sigma$.
    Since every face of $\NN^n$ is of the form $\NN^{n'}$ for some $n' \leq n$, we win.

    ($3 \Rightarrow 1$).
    This follows by running the argument of (1) in reverse for all subcones $K \subset \Sigma$.
\end{proof}

We will call stacky fans satisfying the equivalent conditions of \zcref{prop:smooth_criterion} \emph{smooth stacky fans}.
Using our formalism, it is easy to reinvent the Cox construction:

\begin{prop} \label{prop:cox_construction_fan}
    Let $\Sigma$ be a smooth stacky fan.
    Then there exists a universal stacky cone $K_\Sigma$ receiving an open immersion $\Sigma \hookrightarrow K_{\Sigma}$.
    Furthermore, we can write $K_\Sigma = \Cone(\Nc_\Sigma \vec{/} \NN^{\Sigma(1)})$, where $\NN^{\Sigma(1)} \to \Nc_\Sigma$ is induced by the maps $\{ b_\rho: \NN \to \Nc_\Sigma \}_{\rho \in \Sigma(1)}$.
\end{prop}

\begin{proof}
    Let $K_\Sigma = \Oc(\Sigma)$, so that $K_\Sigma$ is the universal stacky cone receiving a map $j: \Sigma \to K_\Sigma$ by \zcref{prop:Cone_Gamma_adjunction}.
    Let $\kappa$ be the maximal point of $K_\Sigma$.
    Then $P_{K_\Sigma,\kappa} = \colim_{\sigma \in \Sigma} P_{\Sigma,\sigma}$.
    Every $P_{\Sigma,\sigma}$ is the direct sum of the one-dimensional faces $P_{\Sigma,\rho} \cong \NN$ for $\rho \in \Sigma(1) \cap \Sigma_{\leq \sigma}$.
    It follows that $P_{K_\Sigma,\kappa} = \NN^{\Sigma(1)}$ since both satisfy the same universal property.
    The natural map $P_{K_\Sigma,\kappa} \to \Nc_\Sigma$ is induced by $\{ b_\rho: \NN \to \Nc_\Sigma \}_{\rho \in \Sigma(1)}$.
    
    For $\sigma \in \Sigma$, the morphism $j$ sends $\sigma$ to the face $\NN^{\Sigma(1) \cap \Sigma_{\leq \sigma}} \subset \NN^{\Sigma(1)}$.
    Thus:
    \begin{itemize}
        \item $|j|: |\Sigma| \to |K_\Sigma|$ is an injective morphism of posets.
        Indeed, if $j(\sigma_1) = j(\sigma_2)$, then $\Sigma(1) \cap \Sigma_{\leq \sigma_1} = \Sigma(1) \cap \Sigma_{\leq \sigma_2}$.
        Letting $\sigma_3$ be the meet of $\sigma_1$ and $\sigma_2$, we see that $\Sigma(1) \cap \Sigma_{\leq \sigma_1} = \Sigma(1) \cap \Sigma_{\leq \sigma_3}$.
        Saying that $\sigma_3 \leq \sigma_1$ is the same as saying that $P_{\Sigma,\sigma_3}$ is a face of the toric monoid $P_{\Sigma,\sigma_1}$.
        The face $P_{\Sigma,\sigma_3}$ is determined by the one-dimensional faces it contains.
        By construction, these are the same as the one-dimensional faces of $P_{\Sigma,\sigma_1}$, so $\sigma_3 = \sigma_1$.
        Likewise $\sigma_3 = \sigma_2$, so $\sigma_1 = \sigma_2$, as desired.
        \item $j$ is an open immersion Zariski-locally on $\Sigma$.
        Indeed, for any $\sigma \in \Sigma$, the morphism $\Sigma_{\leq \sigma} \to K_{\Sigma}$ agrees with the morphism of stacky cones induced by the face inclusion $\Nc_\Sigma \vec{/} \NN^{\Sigma(1) \cap \Sigma_{\leq \sigma}} \to \Nc_\Sigma \vec{/} \NN^{\Sigma(1)}$.
        This latter morphism of stacky cones is clearly an open immersion.
    \end{itemize}
    Combining these claims shows that $j$ is an open immersion.
\end{proof}

\begin{cor}[Cox construction] \label{cor:cox_construction_stack}
    Let $\Xfr$ be a smooth toric stack.
    Then there exists a universal basic toric stack $\Yfr$ receiving an open immersion $\Xfr \hookrightarrow \Yfr$.
    Furthermore, $\Yfr$ is smooth.
\end{cor}

\begin{proof}
    Use \zcref{thm:classification} to write $\Xfr = \Xfr_\Sigma$ for some unique $\Sigma$.
    The result then follows from \zcref{prop:cox_construction_fan}, since \zcref{prop:open_imm_biconditional} shows that open immersions of stacky fans correspond to open immersions of toric stacks.
\end{proof}

\section{Picard groups and class groups} \label{sec:picard_groups}

In this section we discuss methods for computing the Picard and class groups of toric stacks.
Our main results, \zcref{prop:CLine_computation} and \zcref{cor:cl_sequence}, show how to generalize the standard computations of Picard groups and class groups of toric varieties (e.g.\ \cite[Theorems 4.1.3 and 4.2.1]{cls}) to toric stacks.
The proofs give a nice illustration of the benefits of a local theory of toric stacks: everything can be computed cone by cone and then glued together.
As an application of this theory, we construct new examples of toric schemes which are not quotients of toric varieties (\zcref{prop:counterexample}).

\subsection{Computation of $\CLine$ and Picard groups}

To compute the Picard group of a toric stack $\Xfr_\Sigma$, we first explain how to describe $\CLine(\Xfr_\Sigma)$ in terms of the combinatorial data attached to $\Sigma$.

\begin{prop} \label{prop:CLine_vs_DGale}
    Let $\Xfr_\Sigma$ be a toric stack.
    Then $\CLine(\Xfr_\Sigma) = \GDual(\Oc(\Sigma))$.
    In particular, $\CLine(\Xfr_\Sigma)$ is toric.
\end{prop}

\begin{proof}
    The contravariant functor $\CLine(\Xfr_{-})$ satisfies Zariski descent on $\StkFan$ (by combining the definition $\Xfr_\Sigma = \cup_{K \subset \Sigma} \Xfr_K$ with the fact that $\QCoh(-)$ satisfies Zariski descent).
    The functor $\GDual(\Oc(-))$ also satisfies Zariski descent, since
    \[
        \GDual(\Oc(\Sigma)) = \GDual\Big(\colim_{K \subset \Sigma} \Oc(K)\Big) = \lim_{K \subset \Sigma} \GDual(\Oc(K)).
    \]
    Thus it suffices to show $\CLine(\Xfr_K) = \GDual(\Oc(K))$ for all stacky cones $K$.
    But this is \zcref{lem:linear_reconstruction_geometric}.
\end{proof}

\begin{dfn}
    Let $\Sigma$ be a stacky fan.
    We define the abelian group of (integral) \emph{support functions} on $\Sigma$ as
    \[
        \SF(\Sigma) = \lim_{K \subset \Sigma} \Hom(\PO \Oc(K), \ZZ).
    \]
\end{dfn}

In other words, a support function on $\Sigma$ consists of a collection of linear functions $\PO \Oc(K) \to \ZZ$ which agree on intersections.
Thus this notion extends the notion of support function on classical fans.
Note that we can also describe $\SF(\Sigma)$ as
\[
    \SF(\Sigma) = \lim_{K \subset \Sigma} \Hom\big((\PO \Oc(K))^\gp, \ZZ\big) = \Hom\big((\PO \Oc(\Sigma))^\gp, \ZZ\big),
\]
though this description makes it less clear why elements of $\SF(\Sigma)$ deserve to be called ``support functions.''

We can use the above definition to extract global information about Picard groups.

\begin{prop} \label{prop:CLine_computation}
    Let $\Xfr_\Sigma$ be a toric stack.
    Then
    \[
        \iota_0 \CLine(\Xfr_\Sigma) = \fib_{\tAb}\big( \GDual(\Nc_\Sigma) \to \Bsf\SF(\Sigma)\big).
    \]
\end{prop}

\begin{proof}
    This is a direct computation:
    \begin{align*}
        \iota_0 \CLine(\Xfr_\Sigma) &= \iota_0 \GDual(\Oc(\Sigma))
        \\
        &= \GDual(\Oc(\Sigma)^\gpd) \textrm{ by \zcref{prop:gale_duality_localization}} \\
        &= \GDual\big(\cofib_{\tAb}(\PO \Oc(\Sigma)^\gp \to \iota_0 \Oc(\Sigma))\big) \textrm{ by \zcref{prop:Cgpd_calculation}} \\
        &= \fib_{\tAb}\big(\GDual(\iota_0 \Oc(\Sigma)) \to \GDual(\PO \Oc(\Sigma)^\gp)\big) \\
        &= \fib_{\tAb}\big(\GDual(\iota_0 \Oc(\Sigma)) \to \Bsf \Hom((\PO \Oc(\Sigma))^\gp, \ZZ)\big) \\
        &= \fib_{\tAb}\big( \GDual(\Nc_\Sigma) \to \Bsf\SF(\Sigma)\big) \textrm{ by \zcref{ex:Gale_dual_of_lattice}}. \qedhere
    \end{align*}   
\end{proof}

\begin{cor} \label{cor:pic_sequence}
    Let $\Xfr_\Sigma$ be a toric stack.
    Then there is a natural exact sequence
    \[ 
    \begin{tikzcd}
        0 \rar & \Aut_{\CLine(\Xfr_\Sigma)}(\Osc_{\Xfr_\Sigma}) \rar & \Hom(\pi_0(\Nc_\Sigma), \ZZ) \rar & \SF(\Sigma) \ar[dll, out=-15, in=165] \\
        & \Pic(\Xfr_\Sigma) \rar & \pi_0(\GDual(\Nc_\Sigma)) \rar & 0.
    \end{tikzcd}
    \]
    Furthermore, there is a non-canonical isomorphism
    \[
        \pi_0(\GDual(\Nc_\Sigma)) \cong \Ext^1(\pi_0(\Nc_\Sigma)_\tors, \ZZ) \oplus \Hom(\Aut_{\Nc_\Sigma}(0), \ZZ).
    \]
\end{cor}

\begin{proof}
    This is the long exact sequence in cohomology corresponding to the fiber sequence
    \[
    \begin{tikzcd}
        \iota_0 \CLine(\Xfr_\Sigma) \rar & \GDual(\Nc_\Sigma) \rar & \Bsf\SF(\Sigma)
    \end{tikzcd}
    \]
    in $\TGpd$, where we use the direct computation $\Aut_{\GDual(\Nc_\Sigma)}(0) = \Hom(\pi_0(\Nc_\Sigma), \ZZ)$.
    The non-canonical isomorphism in the final claim follows from computing $\GDual(\Nc_\Sigma)$ as a $\RHom$ group.
\end{proof}

\begin{cor}
    Let $\Xfr_\Sigma$ be a toric stack.
    Then $\iota_0 \CLine(\Xfr_\Sigma / \Tfr_\Sigma) = \Pic(\Xfr_\Sigma / \Tfr_\Sigma) = \SF(\Sigma)$.
\end{cor}

\begin{proof}
    Note that $\Xfr_\Sigma / \Tfr_\Sigma = \Xfr_{\Sigma / \Sigma(0)}$ by \zcref{prop:rigid_eqvar_fan}.
    Because $\Nc_{\Sigma / \Sigma(0)} = 0$ and $\SF(\Sigma / \Sigma(0)) = \SF(\Sigma)$, \zcref{prop:CLine_computation} gives
    \[
        \iota_0 \CLine(\Xfr_\Sigma / \Tfr_\Sigma) = \iota_0 \CLine(\Xfr_{\Sigma / \Sigma(0)}) = \fib_{\tAb}(0 \to \Bsf\SF(\Sigma)) = \SF(\Sigma). \qedhere
    \]
\end{proof}

\subsection{Application: a toric scheme which is not a quotient of a toric variety}

Our Picard group computations can be used to prove that some toric schemes do not arise via the construction of \cite{geraschenko2015toric_1}, giving new counterexamples to the claim of \cite{geraschenko2015toric_2}.

\begin{prop} \label{prop:counterexample}
    There exists a non-separated toric scheme $\Xfr$ which cannot be written as a quotient $X / G$ where
    \begin{itemize}
        \item $X$ is a toric variety with dense torus $T$ and
        \item $G$ is a diagonalizable group acting on $X$ through a homomorphism $G \to T$.
    \end{itemize}
\end{prop}

\begin{proof}
    Let $\Sigma_0$ be the stacky fan of any positive-dimensional complete toric variety $Y = \Xfr_{\Sigma_0}$ with $\Pic(Y) = 0$ (e.g. the toric variety of \cite[Example on pp.\ 25--26]{fulton1993introduction}, also considered in \cite[Example 4.2.13]{cls}).
    Let $K_0 \subset \Sigma_0$ be any subcone containing a maximal point of $|\Sigma_0|$, and let $\Sigma = \Sigma_0 \cup_{K_0} \Sigma_0$.
    We claim that $\Xfr = \Xfr_{\Sigma}$ has the desired properties.

    Because $\Xfr$ is a union of two copies of $Y$, we see that $\Xfr$ is a scheme.
    Furthermore, $\Xfr$ is not separated: for any $\sigma \in |\Sigma_0| \setminus |K_0|$, the fan $\Sigma$ contains two distinct elements $\sigma_1$, $\sigma_2$ such that $\im b_{\sigma_1} = \im b_{\sigma_2} = \im b_\sigma \subset \Nc_\Sigma$, so condition (3) of \zcref{prop:separated_criterion} fails to hold.
    
    We claim that $\Pic(\Xfr) = 0$.
    Because $Y$ is complete, we have $\Aut(\Osc_Y) = k^\times$, so $\Aut_{\CLine(Y)}(\Osc_Y)$ is trivial.
    Plugging the values $\Aut_{\CLine(Y)}(\Osc_Y) = 0$ and $\Pic(Y) = 0$ into \zcref{cor:pic_sequence} (for the toric variety $Y$), we see that $\SF(\Sigma_0) = \Hom(\pi_0(\Nc_{\Sigma_0}), \ZZ))$.
    We also have $\SF(K_0) = \Hom(\pi_0(\Nc_{\Sigma_0}), \ZZ))$ because the rational polyhedral cone corresponding to $K_0$ is a top dimensional cone in $\Nc_{\Sigma_0} \otimes \RR$.
    We can then compute 
    \begin{align*}
        \SF(\Sigma) &= \SF(\Sigma_0) \times_{\SF(K_0)} \SF(\Sigma_0) \\
        &= \Hom(\pi_0(\Nc_{\Sigma_0}), \ZZ)) \times_{\Hom(\pi_0(\Nc_{\Sigma_0}), \ZZ))} \Hom(\pi_0(\Nc_{\Sigma_0}), \ZZ)) \\
        &= \Hom(\pi_0(\Nc_{\Sigma_0}), \ZZ)) \\
        &= \Hom(\pi_0(\Nc_{\Sigma}), \ZZ)).
    \end{align*}
    Inserting this into \zcref{cor:pic_sequence} (for $\Xfr$), we see that
    \[
        \Pic(\Xfr) = \pi_0(\GDual(\Nc_{\Sigma})) = 0.
    \]

    Now suppose for the sake of contradiction that we can write $\Xfr = X / G$ as in the proposition statement.
    Then $X$ is the total space of a $G$-torsor over $\Xfr$.
    Write $G$ as a product of cyclic groups and copies of $\GG_m$, and note that:
    \begin{itemize}
        \item Every $\GG_m$-torsor over $\Xfr$ is trivial.
        \item For all $d > 0$, every $\mu_d$-torsor over $\Xfr$ is trivial by a standard Kummer sequence argument.
    \end{itemize}
    Thus every $G$-torsor over $\Xfr$ is trivial, and in particular $X$ is a product $\Xfr \times G$.
    Since $\Xfr$ is not separated, $X$ cannot be separated, hence cannot be a toric variety.
\end{proof}

\subsection{Defining the class group}

Before computing the class group $\Cl(\Xfr)$, we should give a rigorous definition, since the notion of the class group of a (reasonable) algebraic stack does not seem to be well-developed in the literature.
We offer a quick definition based on the equivalence between the class group and the group of rank-one reflexive sheaves (see e.g.\ \stacks{0EBM}).

\begin{dfn}
    Let $\Xfr$ be a normal, locally integral, locally noetherian algebraic stack.
    Let $\Coh(\Xfr)^\refl$ be the category of \emph{reflexive coherent sheaves} on $\Xfr$, i.e.\ those coherent sheaves $\Fsc$ such that the natural map $\Fsc \to \Fsc^{\vee\vee}$ is an isomorphism (where $(-)^\vee = \scHom_\Xfr(-, \Osc_\Xfr)$).
    As in \stacks{0EBH}, we may view $\Coh(\Xfr)^\refl$ as a symmetric monoidal category with the modified tensor product $\tilde{\otimes}$ given by
    \[
        \Fsc \tilde{\otimes} \Gsc = (\Fsc \otimes_{\Osc_\Xfr} \Gsc)^{\vee\vee}
    \]
    for all $\Fsc, \Gsc \in \Coh(\Xfr)^\refl$.
    The \emph{class group} $\Cl(\Xfr)$ is the group of isomorphism classes of $\tilde{\otimes}$-invertible objects of $\Coh(\Xfr)^\refl$.
\end{dfn}

The standard local results about reflexive sheaves on schemes can be extended to results about reflexive sheaves on algebraic stacks using the following argument.
Suppose $\Xfr$ is a normal, locally integral, locally noetherian algebraic stack.
Then, for any smooth morphism $X \to \Xfr$, the source $X$ must also be normal, locally integral, and locally noetherian (see e.g.\ \stacks{034D}).
The usual smooth descent equivalence $\QCoh(\Xfr) = \lim_{X \to \Xfr} \QCoh(X)$ (where the limit is taken over smooth morphisms to $X$) restricts to an equivalence
\[
    \Coh(\Xfr)^\refl = \lim_{X \to \Xfr} \Coh(X)^\refl.
\]
Thus the data of a reflexive coherent sheaf on $\Xfr$ is the data of reflexive coherent sheaves on $X$ for all smooth $X \to \Xfr$ together with compatibility isomorphisms.

This argument, together with the isomorphism $\Cl(X) = \Pic(X)$ for regular, locally integral, locally noetherian schemes $X$, implies:

\begin{prop} \label{prop:regular_cl_is_pic}
    Let $\Xfr$ be a regular integral locally noetherian algebraic stack.
    Then $\Cl(\Xfr) = \Pic(\Xfr)$.
\end{prop}

Similarly, we may use this argument together with the isomorphism $\Cl(X) = \Cl(U)$ whenever $U \subset X$ is open with $\codim_X (X \setminus U) \geq 2$ to prove:

\begin{prop} \label{prop:codim_2_cl}
    Let $\Xfr$ be a normal integral locally noetherian algebraic stack, and let $\Ufr \subset \Xfr$ be an open substack such that $\codim_\Xfr(\Xfr \setminus \Ufr) \geq 2$.
    Then $\Cl(\Xfr) = \Cl(\Ufr)$.
\end{prop}

\subsection{Computing the class group}

To compute $\Cl(\Xfr_\Sigma)$ for a toric stack $\Sigma$, we will replace the fan $\Sigma$ by the smooth subfan $\Sigma(\leq 1)$.

\begin{prop} \label{prop:Sigma_leq_1_properties}
    Let $\Xfr_\Sigma$ be a toric stack.
    Then $\Xfr_{\Sigma(\leq 1)}$ is a smooth Zariski-open substack of $\Xfr_\Sigma$, and $\codim_{\Xfr_\Sigma}(\Xfr_\Sigma \setminus \Xfr_{\Sigma(\leq 1)}) \geq 2$.
\end{prop}

\begin{proof}
    Since the natural map $\Sigma(\leq\! 1) \to \Sigma$ is an open immersion of stacky fans, the induced morphism $\Sigma(\leq 1) \hookrightarrow \Sigma$ is an open immersion of toric stacks by \zcref{prop:open_imm_biconditional}.

    To see that $\Xfr_{\Sigma(\leq 1)}$ is smooth, we note that for all $\sigma \in \Sigma(\leq 1)$, one of the following holds:
    \begin{itemize}
        \item $\sigma \in \Sigma(0)$, so $P_\sigma = 0$.
        \item $\sigma \in \Sigma(1)$, so $P_\sigma^\gp \cong \ZZ$.
        The only sharp toric monoid with group completion $\ZZ$ is $\NN$, so we must have $P_\sigma \cong \NN$.
    \end{itemize}
    Thus \zcref{prop:smooth_criterion} implies $\Xfr_\Sigma$ is smooth.

    For the claim about codimensions, we note $\Xfr_\Sigma \setminus \Xfr_{\Sigma(\leq 1)}$ decomposes as a union of orbits $\Ofr(\sigma)$ for $\sigma \in \Sigma \setminus \Sigma(\leq 1)$.
    For any such $\sigma$, we have $\codim_{\Xfr_\Sigma} \Ofr(\sigma) \geq 2$ by \zcref{cor:orbit_codimension}.
    Thus 
    \[
    \codim_{\Xfr_\Sigma}(\Xfr_\Sigma \setminus \Xfr_{\Sigma(\leq 1)}) = \min_{\sigma \in \Sigma \setminus \Sigma(\leq 1)} \codim_{\Xfr_\Sigma} \Ofr(\sigma) \geq 2. \qedhere
    \]
\end{proof}

\begin{cor} \label{cor:cl_sequence}
    Let $\Xfr_\Sigma$ be a toric stack.
    Then there is a natural exact sequence
    \[ 
    \begin{tikzcd}
        0 \rar & \Aut_{\CLine(\Xfr_\Sigma)}(\Osc_{\Xfr_\Sigma}) \rar & \Hom(\pi_0(\Nc_\Sigma), \ZZ) \rar & \oplus_{\rho \in \Sigma(1)} \Hom(P_\rho^\gp, \ZZ) \ar[dll, out=-15, in=165] \\
        & \Cl(\Xfr_\Sigma) \rar & \pi_0(\GDual(\Nc_\Sigma)) \rar & 0.
    \end{tikzcd}
    \]
    As in \zcref{cor:pic_sequence}, there is a non-canonical isomorphism
    \[
        \pi_0(\GDual(\Nc_\Sigma)) \cong \Ext^1(\pi_0(\Nc_\Sigma)_\tors, \ZZ) \oplus \Hom(\Aut_{\Nc_\Sigma}(0), \ZZ).
    \]
\end{cor}

\begin{proof}
    This follows by taking the exact sequence of \zcref{cor:pic_sequence} for the toric stack $\Xfr_{\Sigma(\leq 1)}$.
    Indeed, we have
    \begin{itemize}
        \item $\Cl(\Xfr_\Sigma) = \Cl(\Xfr_{\Sigma(\leq 1)}) = \Pic(\Xfr_{\Sigma(\leq 1)})$ by \zcref{prop:codim_2_cl} and \zcref{prop:regular_cl_is_pic}.
        \item $\SF(\Sigma(\leq\! 1)) = \lim_{\rho \in \Sigma(1)} \Hom(P_\rho^\gp, \ZZ)$ because $\Sigma(1)$ is cofinal in $|\Sigma(\leq\! 1)|$, and
        \item $\Aut_{\CLine(\Xfr_{\Sigma(\leq 1)})}(\Osc_{\Xfr_{\Sigma(\leq 1)}}) = \Aut_{\CLine(\Xfr_\Sigma)}(\Osc_{\Xfr_\Sigma})$ by the algebraic Hartogs's theorem. \qedhere
    \end{itemize}
\end{proof}

\part*{References}

\printbibliography[heading=none]

\end{document}